\documentclass[10pt]{article}
\usepackage[a4paper, total={5.5in, 8in}]{geometry}
\usepackage[numbers]{natbib}
\usepackage{amsthm,mathrsfs,amsmath,amsfonts,amssymb,color,enumitem,tensor,mathtools,graphicx,epstopdf,pgfplotstable,tikz,subcaption,xcolor,comment}
\usepackage[breaklinks,bookmarks=false]{hyperref}
\hypersetup{colorlinks, linkcolor=blue, citecolor=blue,
urlcolor=blue, plainpages=false, pdfwindowui=false,
pdfstartview={FitH}}
\usepackage[only,llbracket,rrbracket,llparenthesis,rrparenthesis]{stmaryrd} 
\usepackage{bm}
\usepackage{adjustbox}
\usepackage{bm}

\usepackage[medium,compact]{titlesec}
\titleformat*{\section}{\large\bfseries}
\titleformat*{\subsection}{\bfseries}

\numberwithin{equation}{section}

\newtheorem{theorem}{Theorem}[section]{\bfseries}{\it}
{\bfseries}{\it}
\newtheorem{lemma}[theorem]{Lemma}{\bfseries}{\it}
\newtheorem{corollary}[theorem]{Corollary}{\bfseries}{\it}
\newtheorem{definition}{Definition}[section]{\bfseries}{\it}
\newtheorem{eg}{Example}[section]{\bfseries}{\rmfamily}
{\bfseries}{\it}
\theoremstyle{definition}
\newtheorem{remark}{Remark}[section]{\bfseries}{\rmfamily}
\newcommand{\proofbox}{\qed}

\newcommand{\timeVkforward}{\mathbb{V}_{k}^{+}}
\newcommand{\timeVkbackward}{\mathbb{V}_{k}^{-}}
\newcommand{\LOm}{L^2(\Omega)}
\newcommand{\Dp}{\mathcal{D}_p}

\newcommand{\R}{\mathbb{R}}
\newcommand{\N}{\mathbb{N}}

\newcommand{\jump}[1]{\llbracket #1 \rrbracket}

\newcommand{\dd}{\mathrm{d}}
\renewcommand{\dim}{d}
\newcommand{\p}{\partial}
\newcommand{\abs}[1]{\lvert#1\rvert} 
\newcommand{\tends}{\rightarrow}
\newcommand{\norm}[1]{\lVert#1\rVert}
\newcommand{\tak}{\tilde{a}_k}
\newcommand{\TAK}[1]{\tilde{a}_{k,#1}}
\newcommand{\dualk}{k,*}
\newcommand{\Adualk}{A_\dualk}
\newcommand{\Ip }{\mathcal{I}_{+}}
\newcommand{\In}{\mathcal{I}_{-}}
\newcommand{\Omk}{\Omega,k}
\newcommand{\Omkj}{\Omega,k_j}
\newcommand{\QT}{Q_T}
\newcommand{\DpH}{\mathcal{D}_p H}
\newcommand{\Vk}{\mathbb{V}_k}
\newcommand{\Vkp}{\timeVkforward}
\newcommand{\Vkm}{\timeVkbackward}
\newcommand{\Vkpm}{\mathbb{V}_{k}^{\pm}}
\newcommand{\calVk}{\mathcal{V}_k}
\newcommand{\calVki}{\mathcal{V}_{k,\Omega}}
\newcommand{\Tk}{\mathcal{T}_k}
\DeclareMathOperator{\Card}{card}
\DeclareMathOperator{\supp}{supp}
\DeclareMathOperator{\diam}{diam}
\newcommand{\Ipm}{\mathcal{I}_{\pm}}
\newcommand{\su}{\underline{\sigma}}
\newcommand{\so}{\overline{\sigma}}
\newcommand{\rewrite}[1]{}
\renewcommand{\dim}{d}
\newcommand{\Nk}{N_k}
\newcommand{\calE}{\mathcal{E}}
\newcommand{\calEK}{\calE_K}
\newcommand{\Tke}{\mathcal{T}_{k,E}}
\newcommand{\calVkx}{\mathcal{V}_{k,i}}
\newcommand{\Fki}{F_{K,i}}
\newcommand{\Fkj}{F_{K,j}}
\newcommand{\wEk}{{\omega_{k,E}}}
\newcommand{\wtEk}{{\omega_{k,\widetilde{E}}}}
\newcommand{\Dk}{{D_k}}

\newcommand{\calEki}{\calE_{k,\Omega}}
\newcommand{\Xk}{\mathbb{V}_k,A_k}
\newcommand{\Yk}{{\Vkp,A_k}}

\newcommand{\Vkpo}{\mathbb{V}_{k,0}^{+}}

\newcommand{\logLogSlopeTriangle}[5]
{
	
	\pgfplotsextra
	{
		\pgfkeysgetvalue{/pgfplots/xmin}{\xmin}
		\pgfkeysgetvalue{/pgfplots/xmax}{\xmax}
		\pgfkeysgetvalue{/pgfplots/ymin}{\ymin}
		\pgfkeysgetvalue{/pgfplots/ymax}{\ymax}
		
		\pgfmathsetmacro{\xArel}{#1}
		\pgfmathsetmacro{\yArel}{#3}
		\pgfmathsetmacro{\xBrel}{#1-#2}
		\pgfmathsetmacro{\yBrel}{\yArel}
		\pgfmathsetmacro{\xCrel}{\xArel}
		
		\pgfmathsetmacro{\lnxB}{\xmin*(1-(#1-#2))+\xmax*(#1-#2)} 
		\pgfmathsetmacro{\lnxA}{\xmin*(1-#1)+\xmax*#1} 
		\pgfmathsetmacro{\lnyA}{\ymin*(1-#3)+\ymax*#3} 
		\pgfmathsetmacro{\lnyC}{\lnyA+#4*(\lnxA-\lnxB)}
		\pgfmathsetmacro{\yCrel}{\lnyC-\ymin)/(\ymax-\ymin)} 
		
		\coordinate (A) at (rel axis cs:\xArel,\yArel);
		\coordinate (B) at (rel axis cs:\xBrel,\yBrel);
		\coordinate (C) at (rel axis cs:\xCrel,\yCrel);
		
		\draw[#5]   (A)-- node[pos=0.5,anchor=north] {1}
		(B)-- 
		(C)-- node[pos=0.5,anchor=west] {$1$}
		cycle;
	}
}

 \title{Finite element approximation of time-dependent mean field games with nondifferentiable Hamiltonians}
  \author{Yohance A. P. Osborne\footnotemark[1]~ and Iain Smears\footnotemark[2]}

\begin{document}

\maketitle

 \renewcommand{\thefootnote}{\fnsymbol{footnote}}

\footnotetext[1]{Department of Mathematical Sciences, Durham University, Stockton Road, DH1 3LE Durham, United Kingdom (\texttt{yohance.a.osborne@durham.ac.uk}).}
\footnotetext[2]{Department of Mathematics, University College London, Gower
	Street, WC1E 6BT London, United Kingdom (\texttt{i.smears@ucl.ac.uk}).}

\begin{abstract}
The standard formulation of the PDE system of Mean Field Games (MFG) requires the differentiability of the Hamiltonian. However in many cases, the structure of the underlying optimal problem leads to a convex but nondifferentiable Hamiltonian.
For time-dependent MFG systems, we introduce a generalization of the problem as a Partial Differential Inclusion (PDI) by interpreting the derivative of the Hamiltonian in terms of the subdifferential set.
In particular, we prove the existence and uniqueness of weak solutions to the resulting MFG PDI system under standard assumptions in the literature.
We propose a monotone stabilized finite element discretization of the problem, using conforming affine elements in space and an implicit Euler discretization in time with mass-lumping.
We prove the strong convergence in $L^2(H^1)$ of the value function approximations, and strong convergence in $L^p(L^2)$ of the density function approximations, together with strong $L^2$-convergence of the value function approximations at the initial time.
\end{abstract}

\section{Introduction}\label{sec1}
Mean Field Games (MFG), which were introduced by Lasry and Lions \cite{lasry2006jeux,lasry2006jeux1,lasry2007mean} and independently by Huang, Caines and Malham\'e \cite{huang2006large}, describe the Nash equilibria of large numbers of players involved in a game of stochastic optimal control. 
This leads to a system of a Hamilton--Jacobi--Bellman (HJB) equation  coupled with a  Kolmogorov--Fokker--Planck (KFP) equation, where the unknowns are the value function of the underlying stochastic optimal control problem, and the density function of the player distribution within the state space of the game.
 For extensive reviews on the theory and applications of MFG problems, we refer the reader to \cite{gomes2016regularity,achdou2020mean,GueantLasryLions2003,GomesSaude2014}.

We consider here the model problem
\begin{equation}\label{mfg-classical}
	\begin{aligned}
		-\partial_tu- \nu\Delta u+{H}(t,x,\nabla u)&={F}[m](t,x) &&\text{ in }(0,T)\times\Omega, 
		\\
		\partial_tm-\nu\Delta m -  \mathrm{div}\left(m \frac{\partial H}{\partial p}(t,x,\nabla u)\right) &= G(t,x) &&\text{ in }(0,T)\times\Omega, 
	\\
m(0,x)=m_0(x),\quad u(T,x)&=S[m(T,\cdot)](x)&&\text{ on }\Omega, 
		\\
m=0, \qquad u&=0 &&\text{ on } (0,T)\times\partial\Omega, 
	\end{aligned}
\end{equation}
where $\Omega\subset\R^\dim$, {$\dim\geq 2$,} denotes a bounded polyhedral open connected set with Lipschitz boundary $\partial\Omega$, $T>0$ is the time horizon, and $\nu>0$ is constant. The assumptions on the coupling $F$, the terminal cost $S$, the source term $G,$ and the initial density $m_0$ are specified in Section~\ref{sec2} below. We note from the onset that we allow for wide classes of both local and nonlocal operators $F$ and $S$.
The Dirichlet boundary conditions and the source term $G$ arise in models where players may enter or exit the game.
The Hamiltonian $H$ in \eqref{mfg-classical} is determined by the underlying optimal control problem, and is defined by
\begin{equation}\label{H-time-homogeneous}
	{H}(t,x,p):=\sup_{\alpha\in\mathcal{A}}\left({b}(t,x,\alpha)\cdot p-{f}(t,x,\alpha)\right),
\end{equation}
where $\mathcal{A}$ is the control set, where $b:[0,T]\times\overline{\Omega}\times \mathcal{A}\to\R^\dim$ is the opposite control-dependent drift, and $f:[0,T]\times\overline{\Omega}\times\mathcal{A}\to\mathbb{R}$ is a control-dependent running-cost component. Note that $p\mapsto H(t,x,p)$ is convex on $\mathbb{R}^d$. 
It is well-known in the optimal control literature that in many cases, the optimal controls are not necessarily unique, and are often of \emph{bang-bang} type. 
In these cases the Hamiltonian $H$ may be nondifferentiable, thus raising the question of how to interpret the KFP equation in~\eqref{mfg-classical} when $\frac{\partial H}{\partial p}$ is not well-defined as a function.
From a modelling perspective, this connects to the question as to how players of the game should choose among multiple optimal controls, and how this determines the evolution of the player density.

Whereas existing works on the analysis of MFG often assume that the Hamiltonian is differentiable or even $C^1$-regular with respect to the gradient, there are currently only a handful of works on the nondifferentiable case. In particular the special case of Hamiltonians of eikonal type, but also possibly dependent on $m$, is studied in~\cite{ducasse2020second}.
In~\cite{osborne2022analysis,osborne2024erratum}, we analysed the steady-state analogue of~\eqref{mfg-classical}, in the setting of Lipschitz continuous, but not necessarily differentiable Hamiltonians. We showed that the KFP equation in~\eqref{mfg-classical} can be generalized to a \emph{Partial Differential Inclusion} (PDI) which, in the current setting, is
\begin{equation}
\begin{aligned}
\partial_t m-\nu\Delta m - G \in \mathrm{div}\left( m\partial_p H(t,x,\nabla u )\right)  &&&\text{in }(0,T)\times\Omega, 
\end{aligned}
\end{equation}
where $\partial_pH$ denotes the Moreau--Rockafellar pointwise partial subdifferential of $H$ w.r.t.\ $p$. 
From a modelling perspective, the PDI corresponds to the natural intuition that,  in order to maintain a Nash equilibrium, players in the same state might be required to make distinct choices of optimal controls relative to one another. This entails that the drift term for the density is expected to be a (weighted) average of the chosen drift directions of the individual players, which can be shown to be an element of the subdifferential of the Hamiltonian, {see also \cite{osborne2024regularization} for a heuristic derivation.}
In terms of the analysis, we showed in~\cite{osborne2022analysis,osborne2024erratum} that, for steady-state problems, the MFG PDI system admits a notion of weak solution, with the existence of a solution guaranteed under very mild assumptions on the data, along with uniqueness of the solution guaranteed under a similar monotonicity condition on the coupling as made in~\cite{lasry2007mean}. {Moreover, in \cite{osborne2024regularization} we analysed the connection between MFG PDI systems and their classical PDE counterparts via regularization of the Hamiltonians.}

Our first main contribution in this work is to develop the analysis of the MFG PDI in the time-dependent setting, in extension of our work on the elliptic problem in~\cite{osborne2022analysis,osborne2024erratum}. We show the existence of a weak solution to the problem, and we show that the uniqueness of the weak solution holds under standard monotonicity conditions on the problem data.
 Note that it is known from examples in~\cite{BardiFischer19} that uniqueness of the solution does not necessarily hold if the couplings are not monotone.
The second main contribution of this work is the design and analysis of a monotone Finite Element Method (FEM) of the MFG PDI system.
For problems with differentiable Hamiltonians, a range of numerical methods have been proposed and studied.
This includes monotone finite difference methods on Cartesian grids~\cite{achdou2010mean,achdou2013mean,achdou2016convergence,bonnans2022error}, methods based on optimization reformulations~\cite{AriasKaliseSilva2018}, as well as semi-Lagrangian scheme~\cite{CarliniSilva14,CarliniSilve15}. {In addition, in \cite{osborne2024near} we established the near and full quasi-optimality of a monotone stabilized FEM for stationary MFG PDE systems with ${C}^{1,1}$ Hamiltonians and derived results on optimal convergence rates for the method.}
For problems with nondifferentiable Hamiltonians, the only work so far is~\cite{osborne2022analysis,osborne2024erratum} which covers the steady-state case.

The monotone FEM that we propose here involves continuous piecewise affine finite element spaces for the spatial discretization, and an implicit Euler with mass-lumping for the temporal discretization.
The discretized system uses the exact Hamiltonian and its subdifferential, which greatly helps the consistency and convergence analysis. 
The first-order terms of the operators are stabilized in order to satisfy a Discrete Maximum Principle (DMP) and ensure the positivity of the density approximations, and also to maintain stability in the small diffusion regime.
Although there exist many choices of stabilization to achieve a DMP for FEM \cite{ciarlet1973maximum,tabata1977finite,MIZUKAMI1985181,burman2002nonlinear,burman2005stabilized,barrenechea2017edge,barrenechea2018unified}, we construct an original volume-based stabilization that has a simple analytical form. This helps to simplify the stability analysis in some discrete Bochner--Sobolev norms, and it is also easy to implement in practice.
The analysis allows for unstructured shape-regular simplicial meshes in arbitrary dimensions, requiring only the standard Xu--Zikatanov condition~\cite{xu1999monotone}. This is not restrictive in practice, as it allows for many examples of meshes with nonacute elements, local refinements, etc.
There is no restriction linking the time-step size to the mesh-size, and there is also no restriction on the spatial mesh-grading.

We show that the existence and uniqueness of the discrete approximations hold under the same assumptions as for their continuous counterparts, see Theorems~\ref{discrete-mfg-existence} and \ref{discrete-mfg-uniqueness} below.
We aim here to show the basic convergence of the method without requiring stronger assumptions on the problem data or on the regularity of the exact solution. This is motivated by the fact that examples in~\cite{osborne2022analysis,osborne2024erratum} show that the solution regularity can be quite limited, even in the interior of the domain.
Our main convergence result is in Theorem \ref{conv-main-thm}, which shows the convergence of numerical approximations under the same assumptions on the data as required for the uniqueness results.
In particular, we prove strong convergence of the value function approximations to $u$ in $L^2(0,T;H_0^1(\Omega))$ and the strong convergence of the density function approximations to $m$ in $L^p(0,T;L^2(\Omega))$ for any $p\in [1,\infty)$, along with weak convergence to $m$ in $L^2(0,T;H_0^1(\Omega))$. 
These results do not require any additional regularity assumptions on the solution.
The overall strategy of the analysis is to use the uniform boundedness of the numerical solutions in discrete Bochner--Sobolev norms, including negative-order norms for the time derivatives, to extract subsequences with suitable weak convergence properties, and with suitable strong convergence properties by discrete compact embedding results. Then we use the stability and consistency of the discretization to augment weak convergence to strong convergence in the appropriate norms.
  One of the key challenges that we address here is the problem of showing that the limit of the approximations satisfies the nonlinear final time coupling $u(T,x)=S[m(T,\cdot)](x)$, which we handle by showing strong convergence in $L^2(\Omega)$ of the density approximations at the final time $T$ through a duality argument and stability bounds in weighted discrete Bochner--Sobolev norms.

The paper is organized as follows.
{In Section~\ref{sec2} we describe the basic notation and setting of this work.} The definition the MFG PDI generalization and its notion of weak solution, along with the main theorems on existence and uniqueness,  are presented in Section~\ref{sec3}.
 Section~\ref{sec4} introduces the setting for discretization, including the mass-lumping and spatial stabilization.
The FEM is then introduced in Section \ref{sec5}, along with the statements of the theorems on existence, uniqueness, and convergence of the discrete approximations.
Section~\ref{sec6} gives the proof of the existence, uniqueness and stability of the numerical methods. Section~\ref{sec7} details the discrete functional analytic compactness tools in a self-contained exposition. The proof of convergence of the method is shown in~Section~\ref{sec8:proofs_of_convergence}.
We present the results of a numerical experiment in~Section~\ref{sec9}.
Appendices~\ref{sec:app:dmp} and~\ref{sec:app:discrete_wellposedness} give proofs of some of the supporting results, such as the proof of the DMP under the proposed stabilization.

\section{Setting and notation}\label{sec2}
{\subsection{Basic notation}}
Let $\Omega$ be a bounded, open connected subset of $\R^\dim$ with Lipschitz boundary $\partial \Omega$.
We use the following notation for $L^2$-norms {and inner products}: for a Lebesgue measurable set $\omega \subset \R^\dim$, $d\geq 2$, let $\lVert \cdot \rVert_\omega$ and $(\cdot,\cdot)_{\omega}$denote respectively the standard $L^2$-norm and $L^2$-inner product for scalar- and vector-valued functions on $\omega$. 

Given $1\leq p,q <\infty$, let $L^p(0,T;L^q(\Omega))$ denote the space of mappings $w:(0,T)\to L^q(\Omega)$ that are strongly measurable (see \cite[Ch.\ V]{yosida1980functional}) with $\int_0^T\|w\|_{L^q(\Omega)}^p\mathrm{d}t<\infty$. In particular, we denote the norm on $L^2(0,T;L^2(\Omega))$ by  $\|\cdot\|_{Q_T}$ where $Q_T\coloneqq (0,T)\times\Omega$. When using notation for $L^2$-norms and the $L^2$-inner product, the dependence of functions on the state variable $x\in\Omega$ will be suppressed. It is known  that the spaces $L^p(Q_T)$ and $L^p(0,T;L^p(\Omega))$ are isometrically isomorphic to each other for every $1\leq p<\infty$ {(see \cite[Ex.\ 5.0.32]{fattorini1999infinite}).}
As such, when $1\leq p<\infty$, we identify the spaces $L^p(Q_T)$ and $L^p(0,T;L^p(\Omega))$.
Note that $L^{\infty}(Q_T)$ is not isomorphic to $L^{\infty}(0,T;L^{\infty}(\Omega))$.

Let {$\langle\cdot,\cdot\rangle$} denote the duality pairing between $H_0^1(\Omega)$ and {its dual space} $H^{-1}(\Omega)$. 
{Let the norm $\|\cdot\|_{H^{-1}(\Omega)}$ be defined by} 
\begin{equation}
\|w\|_{H^{-1}(\Omega)}\coloneqq \sup_{v\in H_0^1(\Omega)\setminus\{0\}}\frac{\langle w,v\rangle}{\|\nabla v\|_{\Omega}} \quad \forall w \in H^{-1}(\Omega).
\end{equation}
Let ${X}\coloneqq L^2(0,T;H_0^1(\Omega))$ 
and {$Y\coloneqq L^2(0,T;H^1_0(\Omega))\cap H^1(0,T;H^{-1}(\Omega))$, with respective norms 
\begin{equation}\label{eq:continuous_norms}
\begin{aligned}
\norm{v}_{X}^2\coloneqq \int_{0}^{T}{\norm{\nabla v}_{\Omega}^2}\mathrm{d}t, &&&
\norm{w}_{Y}^2\coloneqq \int_0^T \norm{\partial_t w}_{H^{-1}(\Omega)}^2 + \norm{\nabla w}_{\Omega}^2 \mathrm{d}t,
\end{aligned}
\end{equation}
for all $v \in X$ and all $w\in Y$. Let $Y_{0}\subset Y$ denote the subspace of functions that vanish at $t=0$.} 

Let the source term ${G}\in L^2(0,T;H^{-1}(\Omega))$ and the initial distribution $m_0\in \LOm$. We will say that ${G}$ is nonnegative in the sense of distributions if $\int_0^T\langle {G},v\rangle_{}\mathrm{d}t\geq 0$ whenever $v\in L^2(0,T,H_0^1(\Omega))$ is such that $v\geq 0$ a.e.\ in $\QT$.

Next, let the coupling term $ {F}: L^2(0,T;L^2(\Omega))\to L^2(0,T;H^{-1}(\Omega))$ be a continuous operator.
 Moreover, we assume that there exists a constant $C_F>0$ such that  
\begin{equation}
		\|F[v]\|_{L^2(0,T;H^{-1}(\Omega))}\leq C_F\left(\|v\|_{Q_T}+1\right) \quad\forall v\in L^2(0,T;L^2(\Omega)).\label{F-boundedness}
\end{equation}
We say that ${F}$ is \emph{strictly monotone on $L^2(0,T;H_0^1(\Omega))$} if
$$\int_{0}^T\langle {F}[w]-{F}[v], w-v\rangle\mathrm{d}t\leq 0\quad \Longrightarrow \quad w=v\quad \text{in }L^2(0,T;H_0^1(\Omega))$$ whenever $w,v\in  L^2(0,T;H_0^1(\Omega))$.
{Note that although $F$ is defined on $L^2(0,T;\LOm)$, strict monotonicity will only be used later on the smaller space $L^2(0,T;H^1_0(\Omega))$}.
Finally, we suppose that the terminal cost operator ${S}:\LOm\to \LOm$ is continuous and that there exists a constant $C_S\geq 0$ such that
\begin{equation}\label{S2}
		\|S[v]\|_{\Omega}\leq C_S(\|v\|_{\Omega}+1)  \quad \forall v\in \LOm.
\end{equation}
We say that ${S}$ is \emph{monotone on $\LOm$} if 
$$\left({S}[w]-{S}[v],w-v\right)_{\Omega}\geq 0\quad \forall w,v\in \LOm.$$

\begin{eg}
	The above hypotheses allow for a broad class of coupling operators $F$. This class includes for example local operators $F:L^2(0,T;L^2(\Omega))\to L^2(Q_T)$ of the form $F[z](t,x)\coloneqq f(z(t,x))$, for a.e.\ $(t,x)\in Q_T$, where $z\in L^2(0,T;L^2(\Omega))$ is given and the function $f:\mathbb{R}\to\mathbb{R}$ is strictly monotone and Lipschitz continuous. This class also includes nonlocal smoothing operators of the form $F\coloneqq (\partial_t-\nu\Delta)^{-1}: L^2(0,T;L^2(\Omega))\to Y$ where, for each $z\in L^2(0,T;L^2(\Omega))$, $F[z]$  is defined as the unique function in $Y$ satisfying $\int_0^T\langle\partial_t F[z],\phi\rangle +\nu(\nabla F[z],\nabla \phi)_{\Omega}\mathrm{d}t = \int_0^T(z,\phi)_{\Omega}\mathrm{d}t$ for all $\phi\in X$.
	Moreover, this class admits examples of differential-type operators, e.g.\ operators $F:L^2(0,T;L^2(\Omega))\to L^2(0,T;H^{-1}(\Omega))$ which take the form $ \langle F[z](t),\phi \rangle\coloneqq -(z\bm{v}\cdot \nabla \phi)_{\Omega}$ for $z\in L^2(0,T;L^2(\Omega))$, $\phi\in H_0^1(\Omega)$ and for a.e.\ $t\in (0,T)$, where $\bm{v}\in  C^{1}(\overline{\Omega};\mathbb{R}^d)$ is some vector field that satisfies $\text{div}(\bm{v})>0$ in $\Omega$. In this case, for all $z_1,z_2\in X$, we have $\int_0^T\langle F[z_1](t) - F[z_2](t),z_1-z_2 \rangle\mathrm{d}t= \frac{1}{2}\int_0^T(\text{div}(\bm{v}),(z_1-z_2)^2)_{\Omega}\mathrm{d}t$ so $F$ is strictly monotone on $L^2(0,T;H^1_0(\Omega))$.
	 
\end{eg}

\subsection{Subdifferential of the Hamiltonian}

Recall that the Hamiltonian $H$ is defined in~\eqref{H-time-homogeneous}.
We assume that $\mathcal{A}$ is a compact metric space, and that ${b}:[0,T]\times\overline{\Omega}\times \mathcal{A}\to\R^\dim$ and ${f}:[0,T]\times\overline{\Omega}\times\mathcal{A}\to\mathbb{R}$ are uniformly continuous functions.
It then follows that the Hamiltonian $H$ given by \eqref{H-time-homogeneous} is Lipschitz continuous with respect to its third argument, i.e.\
\begin{equation}\label{eq:Lipschitz_H}
\abs{H(t,x,p_1)-H(t,x,p_2)}\leq L_H \abs{p_1-p_2} \quad \forall p_1,\,p_2\in \R^\dim,\; \forall (t,x)\in (0,T)\times\Omega,
\end{equation}
where $L_H \coloneqq \|b\|_{C([0,T]\times\overline{\Omega}\times\mathcal{A};\R^\dim)}$.

Let {the set-valued map $ \partial_p {H}\colon Q_T\times\R^\dim\rightrightarrows\R^\dim$ denote the pointwise Moreau-Rockafellar partial subdifferential of $H$} with respect to {$p\in \R^\dim$}, which is defined by 
\begin{equation}\label{subdifferential-space-time}
	\partial_p {H}(t,x,p)\coloneqq\left\{z\in\R^\dim:{H}(t,x,q)\geq {H}(t,x,p)+z\cdot(q-p)\quad\forall q\in\R^\dim\right\}.
\end{equation}
Note that $\partial_p {H}(t,x,p)$ is nonempty for all $x\in \Omega$ and $p\in\R^\dim$ because $H$ is real-valued and convex in $p$ for each fixed $(t,x)\in Q_T$.

Let $w\in L^1(0,T;W^{1,1}(\Omega))$ be given.
Following \cite{osborne2022analysis,osborne2024erratum}, we give notation for a collection of measurable selections of $\partial_pH(\cdot,\cdot,\nabla w)$. We say that a vector field $\tilde{b}\in L^{\infty}(Q_T;\R^\dim)$ is a \emph{measurable selection of $\partial_p{H}(\cdot,\cdot,\nabla w)$} if $\tilde{b}(t,x)\in \partial_p {H}(t,x,\nabla {w}(t,x))$ for a.e.\ $(t,x)\in Q_T$. 
\begin{definition}\label{DpH-space-time}
	Let ${H}$ be the function given by \eqref{H-time-homogeneous}. We define the set-valued map ${\Dp}{H}\colon L^1(0,T;W^{1,1}(\Omega))\rightrightarrows L^{\infty}(Q_T;\R^\dim)$ via
	$${\Dp}{H}[v]\coloneqq \left\{\tilde{b}\in L^{\infty}(Q_T;\R^\dim):\tilde{b}(t,x)\in  \partial_p {H}(t,x,\nabla v(t,x)) \text{ \emph{for a.e.\ }}(t,x)\in Q_T\right\}.$$
\end{definition}

We begin by showing that ${\Dp}H$ {has nonempty images for every argument in $L^1(0,T;W^{1,1}(\Omega))$}.
Define the set-valued map $\Lambda\colon Q_T\times\R^\dim\rightrightarrows \mathcal{A}$ via
\begin{equation}\label{eq:Lamda_set}
\Lambda(t,x,p)\coloneqq\text{argmax}_{\alpha\in\mathcal{A}}\{b(t,x,\alpha)\cdot p-f(t,x,\alpha)\}\quad\forall (t,x,p)\in Q_T\times\R^\dim.
\end{equation}
Observe that $\Lambda(t,x,p)$ is nonempty for all $(t,x,p)\in Q_T\times\R^\dim$ since $\mathcal{A}$ is compact and the functions $b$ and $f$ are uniformly continuous on $[0,T]\times\overline{\Omega}\times \mathcal{A}$. For a given $v\in L^1(0,T;W^{1,1}(\Omega))$, let $\Lambda[v]$ denote the set of all Lebesgue measurable functions $\alpha^*:Q_T\to\mathcal{A}$ that satisfy $\alpha^*(t,x)\in\Lambda(t,x,\nabla {v}(t,x))$ for a.e.\ $(t,x)\in Q_T$. We will refer to each element of $\Lambda[v]$ as \emph{a measurable selection of }$\Lambda(\cdot,\cdot,\nabla {v}(\cdot,\cdot))$. 
It can be shown that $\Lambda[v]$ is nonempty for each $v\in L^1(0,T;W^{1,1}(\Omega))$, see e.g.\ \cite[Theorem~10]{smears2014discontinuous}, where the proof of the existence of measurable selections ultimately rests upon the Kuratowski and Ryll--Nardzewski Selection Theorem~\cite{kuratowski1965general}.  

\begin{lemma}\label{Prop1-time}
{For every $v\in L^1(0,T;W^{1,1}(\Omega))$, the set~$\DpH[v]$ is a nonempty subset of $L^\infty(Q_T;\R^\dim)$ and}

	\begin{equation}\label{Hsubdiff-bound-time}
		\sup_{\tilde{b}\in {\Dp}{H}[v]}\|\tilde{b}\|_{L^{\infty}(Q_T;\R^\dim)}\leq {L_H}.
	\end{equation}
\end{lemma}
\begin{proof}
	Let $v\in L^1(0,T;W^{1,1}(\Omega))$ be given. 
{As explained above, the set $\Lambda[v]$ is nonempty}, so after choosing some $\alpha^* \in \Lambda[v]$, we deduce from~\eqref{eq:Lamda_set} that
 \begin{equation}
 H(t,x,\nabla {v}(t,x))=b(t,x,\alpha^*(t,x))\cdot\nabla {v}(t,x)-f(t,x,\alpha^*(t,x)),
 \end{equation}
 for a.e.\ $(t,x)\in Q_T$.
Let $b^*\colon(t,x)\mapsto b(t,x,\alpha^* (t,x))$ for a.e.\ $(t,x)\in \QT$. It then follows that $ b^*(t,x)$ is Lebesgue measurable and $b^*(t,x)\in \partial_p H(t,x,\nabla {v}(t,x))$ for a.e.\ $(t,x)\in Q_T$.
Furthermore, we have that $b^* \in L^\infty(Q_T;\R^\dim)$ since $\abs{b^*(t,x)}\leq \|b\|_{C(\overline{Q_T}\times\mathcal{A};\R^\dim)}{=L_H}$ for a.e.\ $(t,x)\in \QT$.
Hence, $b^*\in {\Dp}H[v]$, thus establishing the nonemptiness of ${\Dp}H[v]$ for each $v\in L^1(0,T;W^{1,1}(\Omega))$.
To obtain the uniform bound~\eqref{Hsubdiff-bound-time}, {note that the Lipschitz continuity of $H$, c.f.\ \eqref{eq:Lipschitz_H}, readily implies that, for all $(t,x,p)\in (0,T)\times\Omega\times \R^\dim$, the subdifferential $\partial_p H(t,x,p)$ is contained in a closed ball in $\R^\dim$ centred at the origin and of radius $L_H$. This implies that  any $\tilde{b} \in \DpH[v]$ satisfies $\abs{\tilde{b}(t,x)}\leq L_H$ for a.e.\ $(t,x)\in \QT$. }
\end{proof}

The second preliminary result describes a closure property of the set-valued map ${\Dp}{H}$ that will be key to the foregoing analysis. 
{Recall that the space $X=L^2(0,T;H^1_0(\Omega))$ was defined above.}
\begin{lemma}\label{closure}
{Let $\{v_k\}_{k\in\N}$ be a sequence in $X$ that converges strongly to a limit $v \in X$. 
If a sequence $\{\tilde{b}_k\}_{k\in\N}$ satisfies $\tilde{b}_k\in \DpH[v_k]$ for all $k\in \N$, and if $\{\tilde{b}_k\}_{k\in\N}$ converges weakly to some $\tilde{b}$ in $L^{2}(\QT;\R^\dim)$ as $k\tends \infty$, then $\tilde{b}\in \DpH[v]$.} 

\end{lemma}
This is proved similarly to the approach taken in \cite[Lemma 3]{osborne2022analysis}.


\section{Continuous Problem}\label{sec3}
We now give the definition of the weak solution of the MFG PDI problem.
\begin{definition}[Weak Solution]\label{weakdef-space-time}
	We say that a pair $(u,m)\in X\times Y$ is a weak solution of~\eqref{mfg-classical} if $m(0)=m_0$ in $\LOm$
	and there exists {a} $\tilde{b}_*\in {\Dp}H[u]$ such that, for all $(\psi,\phi)\in Y_{0}\times X$,
	\begin{subequations}\label{eq:weakform-space-time}
		\begin{gather}
			\begin{split} 
				\int_0^T\left\langle\partial_t\psi,u\right\rangle +\nu(\nabla u,\nabla  \psi)_{\Omega}+(H[\nabla u],\psi)_{\Omega} \mathrm{d}t
				=&\int_0^T\langle {F}[m],\psi\rangle\mathrm{d}t
				\\
				&+(S[m(T)],\psi(T))_{\Omega},
			\end{split}\label{weakform1-space-time}
			\\ 
			\int_0^T\left\langle\partial_tm,\phi\right\rangle +\nu(\nabla m,\nabla \phi)_{\Omega}+(m\tilde{b}_*,\nabla \phi)_{\Omega} \mathrm{d}t=\int_0^T\langle {G}(t),\phi\rangle_{}\mathrm{d}t.\label{weakform2-space-time} 
		\end{gather} 
	\end{subequations}
\end{definition}

\begin{remark}
In order to see how~Definition~\ref{weakdef-space-time} leads to a PDI, define the set
\begin{equation}
\mathrm{div}\left( m \mathcal{D}_pH[u] \right)\coloneqq \left\{ g\in L^2(0,T;H^{-1}(\Omega));\; \exists \tilde{b}\in \mathcal{D}_pH[u] \text{ s.t. } \; g =  \mathrm{div}(\tilde{b} m) \right\},
\end{equation}
where the equality above is understood in the sense of distributions. Then equation~\eqref{weakform2-space-time} can be written as $\p_t m -\nu \Delta m -G \in \mathrm{div}\left( m \mathcal{D}_pH[u] \right)$ in the sense of distributions in $L^2(0,T;H^{-1}(\Omega))$.
We stress that Definition~\ref{weakdef-space-time} does not require uniqueness of the term $\tilde{b}_*$ appearing in~\eqref{eq:weakform-space-time}, although in some special cases its uniqueness can be shown, see~\cite{osborne2022analysis,osborne2024erratum}.
It is furthermore known from examples in~\cite{osborne2022analysis,osborne2024erratum} that the regularity of the solution can be quite limited even in the interior of the domain, in particular for the steady-state problem, there are examples where $m \in H^{3/2-\epsilon}$ for arbitrarily small $\epsilon>0$ but $m\notin H^{3/2}$. 
\end{remark}

\begin{remark}
	The choice of test and trial spaces here, together with the condition on the transport vector $\tilde{b}_*$, are sufficient to ensure that each term in~\eqref{eq:weakform-space-time} is well-defined and finite. It is also possible to write other formulations of the problem that are ultimately equivalent, such as casting the time-derivative onto $u$ in~\eqref{weakform1-space-time}, since if~$(u,m)\in X\times Y$ is a solution of \eqref{eq:weakform-space-time}, then $u$ admits a distributional time derivative $\p_t u \in L^2(0,T;H^{-1}(\Omega))$ so that $u\in Y$, and moreover that $u(T)=S[m(T)]$ a.e.\ in~$\Omega$.
\end{remark}

We now state the main result on the existence of a solution of~\eqref{eq:weakform-space-time}.
\begin{theorem}[Existence of Solutions]\label{existence-space-time}
	There exists at least one solution $(u,m)\in Y\times Y$ of~\eqref{eq:weakform-space-time} in the sense of Definition \ref{weakdef-space-time}. 
\end{theorem}

{In this paper, we give a constructive proof of Theorem~\ref{existence-space-time} in the case where the domain~$\Omega$ is a polytope by showing the convergence of finite element approximations of the problem. The proof of Theorem~\ref{existence-space-time} is therefore given later in Section~\ref{sec8:proofs_of_convergence}.
We can also prove the existence of a solution of~\eqref{eq:weakform-space-time} for more general domains by a more direct analysis, although we omit this second approach for the sake of brevity; see the PhD thesis~\cite{YohancePhD} of the first author for a complete analysis.

{Under standard monotonicity assumptions, we recover the uniqueness of the solution in the sense of Definition~\ref{weakdef-space-time}.}
{Crucially, by choosing the subdifferential set $\DpH$ as the natural generalization of the partial derivative of the Hamiltonian to the case of nondifferentiable Hamiltonians, we obtain a straightforward generalization of the uniqueness result of~\cite{lasry2007mean}, see also~\cite{osborne2022analysis,osborne2024erratum}.}

\begin{theorem}[Uniqueness of Solutions]\label{uniqueness-space-time} 
	If  the initial density $m_0\in \LOm$ is nonnegative a.e.\ in $\Omega$, the source term $G$ is nonnegative in the sense of distributions in $L^2(0,T;H^{-1}(\Omega))$, the coupling term $F$ is strictly monotone on $X$ and the terminal cost $S$ is monotone on $\LOm$, then there is at most one weak solution $(u,m)$ to \eqref{eq:weakform-space-time} in the sense of Definition \ref{weakdef-space-time}.
\end{theorem}

\begin{proof}

	Suppose that there are two solutions $(u_i,m_i)$, $i\in\{1,2\}$, in the sense of Definition \ref{weakdef-space-time}. Then, we have, for each $i\in\{1,2\}$,
	\begin{equation}\label{mdiff}
		\begin{split} 
			\int_0^T\left\langle \partial_tm_i,\phi\right\rangle_{}+\nu(\nabla m_i,\nabla \phi)_{\Omega}+(m_i\tilde{b}_i,\nabla \phi)_{\Omega}\mathrm{d}t=\int_0^T\langle {G}(t),\phi\rangle_{}\mathrm{d}t,
		\end{split} 	
	\end{equation} 
	for all $\phi\in X$, for some $ {\tilde{b}_i}\in {\Dp}{H}[u_i]$, {and}
	\begin{equation}\label{udiff}
		\begin{split}
			\int_0^T\left\langle \partial_t\psi,u_i\right\rangle_{}+\nu(\nabla u_i,\nabla \psi)_{\Omega}+(H[\nabla u_i],\psi)_{\Omega}\mathrm{d}t=&\int_0^T\langle F[m_i],\psi\rangle\mathrm{d}t
			\\
			&+(S[m_i(T)],\psi(T))_{\Omega},
		\end{split}
	\end{equation}
	for all $\psi\in Y_{0}$, with $m_i({0})=m_0$ in $L^2(\Omega)$. {Since $m_0\geq 0$ a.e.\ in $\Omega$, since $G$ is nonnegative in the sense of distributions and since $m_i|_{\partial \Omega} =0$ for a.e.\ $t\in (0,T)$, the weak maximum principle (see, e.g.\ \cite{aronson1967local}) implies that} $m_i \geq 0$ a.e.\ {in $\QT$, for each} $i\in\{1,2\}$. After choosing test functions $\psi = m_1-m_2\in Y_0$ and $\phi = u_1-u_2$ in \eqref{udiff} and \eqref{mdiff}, respectively, and subtracting the two equations, we deduce that
	\begin{equation}\label{7'-time}
		\begin{split} 
			\int_0^{T}\int_{\Omega}m_1{\lambda_{12}}+m_2{\lambda_{21}}\mathrm{d}x\mathrm{d}{t}=& \int_0^{T}\left\langle {F}[m_1]-{F}[m_2], m_1-m_2\right\rangle\mathrm{d}t
			\\
			&+({S}[m_1(T)]-{S}[m_2(T)], m_1(T)-m_2(T))_{\Omega},
		\end{split} 
	\end{equation}
	where {$\lambda_{ij}\in L^2(0,T;\LOm)$, $i,j\in\{1,2\}$, is defined by
	$$\lambda_{ij}:=-H[\nabla {u_j}]+H[\nabla {u_i} ]+ \tilde{b}_i \cdot\nabla ({u_j}-{u_i}). $$
B}y definition of ${\Dp}{H}$ {in~Definition~\ref{DpH-space-time}}, it follows that ${\lambda_{ij}}\leq 0$ a.e.\ in $\QT$. 
{Using the nonnegativity of $m_i$, $i\in \{1,2\}$, we then} deduce from \eqref{7'-time} that 
	\begin{equation}
		\begin{split}
			 0\geq 
			 \int_0^{T}\langle {F}[m_1]-&{F}[m_2], m_1-m_2\rangle\mathrm{d}t
			\\
			&+({S}[m_1(T)]-{S}[m_2(T)], m_1(T)-m_2(T))_{\Omega}.
		\end{split}
	\end{equation}
{The strict monotonicity of $F$ and the monotonicity of $S$ then imply that $m_1=m_2$ in $\QT$.}
{Since $Y$ is continuously embedded in $C([0,T];L^2(\Omega))$, it follows that $m_1(T)=m_2(T)$ a.e.\ in $\Omega$.}
Consequently, {it follows that $u_1$ and $u_2$ solve~\eqref{weakform1-space-time} with common right-hand side and common boundary conditions,  and thus $u_1=u_2$ in $\QT$ by uniqueness of the solution of the corresponding HJB equation}.
\end{proof}



\section{Notation and setting for discretization}\label{sec4} 
In this section we introduce a monotone finite element scheme for approximating solutions to the weak formulation \eqref{eq:weakform-space-time}. In the sequel,  we shall further assume that $\Omega$ is a polyhedron, in addition to the earlier assumption that it is a bounded connected open set with Lipschitz boundary.

\subsection{Meshes}
 
Let $\{\mathcal{T}_k\}_{k\in\mathbb{N}}$ be a {shape-regular} sequence of conforming simplicial meshes of the domain $\Omega$, c.f.\ \cite[p.~51]{Ciarlet1978}.
In addition we assume that the meshes $\{\mathcal{T}_k\}_{k\in\mathbb{N}}$ are nested, {i.e.\ every element in $\mathcal{T}_{k+1}$ is either in $\Tk$ or is a subdivision of an element in $\Tk$.}
 {For each $k\in\N$, define $h_{\Tk}\in L^\infty(\Omega)$ the mesh-size function by $h_{\Tk}|_K\coloneqq\diam K$ for each element $K\in\Tk$, where $\diam K$ denotes the diameter $K$.
The maximum element size in $\Tk$ is denoted by $h_k\coloneqq \norm{h_{\Tk}}_{L^\infty(\Omega)}$.}
We assume that $h_k\to 0$ as $k\to \infty$.
{Shape-regularity of $\{\Tk\}_{k\in\N}$} refers to the condition that there exists a real-number $\delta>1$, independent of $k\in\mathbb{N}$, such that $ h_{\Tk}|_K\leq \delta \rho_{K}$ for all $K\in\Tk$, for all $k\in\N$, where $\rho_K$ denotes the radius of the largest inscribed ball in the element $K$.

{\paragraph{Notation for inequalities}
For real numbers $a$, $b$, we write $a\lesssim b$ if $a\leq C b$ for some constant $C$ that depends only on the problem data $\Omega$, $T$, $\mathcal{A}$, $b$, $\nu$, $H$, $F$ and $S$, and on the shape-regularity constant $\delta$, but is otherwise independent of the mesh $\Tk$ and the mesh-size $h_k$. We write $a \eqsim b$ if $a\lesssim b$ and $b\lesssim a$.}

\paragraph{Sets of vertices and edges} {For each $k\in \N$, let $\calVk$ denote the set of all vertices of $\Tk$ and let $\calVki=\calVk\cap \Omega$ denote the set of interior vertices of $\Tk$.} 
Let $\{x_i\}_{i=1}^{\Card\calVk}$ be an enumeration of $\calVk$. Without loss of generality, we may choose the ordering such that $x_i\in \calVki$ if and only if $i\leq M_k\coloneqq \Card\calVki$.
Two distinct vertices in $\calVk$ are called neighbours if they belong to a common element of $\Tk$. For a vertex $x_i \in \calVk $, the set of neighbouring vertices of $x_i$ is denoted by $\calVkx$.

For each $k\in\mathbb{N}$, let $\calE_k$ denote the set of edges of the mesh $\mathcal{T}_k$, i.e.\ the set of all closed line segments formed by all pairs of neighbouring vertices.
Given an edge $E\in\mathcal{E}_k$, let $\Tke\coloneqq \{K\in\mathcal{T}_k:E\subset K\}$ denote the set of elements of $\Tk$ containing~$E$.
 We say that an edge $E\in\calE_k$ is an internal edge if $E$ contains at least one vertex in~$\calVki$.
The set of all internal edges is denoted by $\calEki$.
For each vertex $x_i {\in \calVk}$, let $\mathcal{E}_{k,i}\coloneqq\{E\in\mathcal{E}_k:x_i\in E\}$ denote the set of edges containing $x_i$.
For each $K\in\Tk$, let $\calE_K\coloneqq \{E\subset K: E\in\mathcal{E}_{k,\Omega}\}$ denote the set of edges of the simplex $K$ that are internal edges.

For each element $K\in\mathcal{T}_k$ and a given vertex $x_i\in K$, let $F_{K,i}$ denote the $(d-1)$-dimensional face of $K$ that is opposite $x_i$, i.e.\ $F_{K,i}$ is the convex hull of all vertices of $K$ except $x_i$.
Let $\theta_{ij}^K$ denote the dihedral angle between the faces $\Fki$ and $\Fkj$. 
We assume that the family of meshes $\{\mathcal{T}_k\}_{k\in\mathbb{N}}$ satisfies the following hypothesis of Xu and Zikatanov (cf.\ \cite{xu1999monotone}): for any $k\in\mathbb{N}$ and for any internal edge $E\in\mathcal{E}_k$ formed by neighbouring vertices $x_i$ and $x_j$, there holds
\begin{equation}\label{XZ-condition}
	\sum_{K\in \Tke} |\Fki\cap \Fkj|_{d-2}\cot(\theta_{ij}^K)\geq 0,
\end{equation}
where $|\cdot|_{d-2}$ denotes the $d-2$ dimensional Hausdorff measure (counting measure for $d=2$).
This condition ensures that the stiffness matrix for the Laplacian on $V_k$ is an $M$-matrix~\cite{xu1999monotone}.
In the case where $d=2$, this condition requires that the sum of the angles opposite to any edge should be less than or equal to $\pi$. {Furthermore, meshes satisfying the condition \eqref{XZ-condition} in the case $d=2$ are equivalently Delaunay triangulations \cite{xu1999monotone}. This includes, for instance, meshes with nonobtuse triangles and also some meshes with obtuse triangles. Note also that such meshes can be often be obtained by standard mesh refinement routines. For example, if the initial mesh $\mathcal{T}_1$ consists of right-angled triangles then the newest vertex bisection method creates a sequence of meshes of only right-angled triangles, or if the initial mesh is acute then red refinement yields a sequence of meshes of acute triangles \cite[p.\ 66]{Verfurth2013}.}  

\subsection{Spatial finite element spaces}
Given an element $K\subset\R^\dim$, we let $\mathcal{P}_1(K)$ denote the vector space of $d$-variate real-valued polynomials on $K$ of total degree at most one. 
For each $k\in\N$, let the spatial finite element space $V_k$ be defined by
\begin{equation}
V_k \coloneqq \{v \in H^1_0(\Omega),\; v|_K \in \mathcal{P}_1(K) \quad \forall K\in\mathcal{T}_k \}.
\end{equation}
For each $k\in\N$, let $\{ \xi_i \}_{i=1}^{\Card\calVk}$ denote the standard nodal Lagrange basis for the space of all continuous piecewise linear functions on $\overline{\Omega}$ with respect to $\mathcal{T}_k$, where $\xi_i(x_j)=\delta_{ij}$ for all $i,\,j \in \{1,\dots,\Card\calVk\}$ for the chosen enumeration $\{x_i\}_{i=1}^{\calVk}$ of $\calVk$.
As there is no risk of confusion, we omit the dependence of the nodal basis on the index $k$ of the mesh in the notation.
Note that $\{\xi_i\}_{i=1}^{\Card\calVk}$ form a partition of unity on $\Omega$.
Recalling that $x_i \in \calVki$ if and only if $i\leq M_k\coloneqq \Card\calVki$, we see that $\{\xi_i\}_{i=1}^{M_k}$ is the standard nodal basis of $V_k$.

Let $V_k^*$ denote the space of continuous linear functionals on $V_k$, where we let the duality pairing between $V_k$ and $V_k^*$ be denoted by $\langle\cdot,\cdot\rangle_{V_k^*\times V_k}$. 
We equip $V_k^*$ with the {standard dual norm $\norm{\cdot}_{V_k^*}$}.
For any operator $\mathcal{L}:V_k\to V_k^*$ we define the adjoint operator $\mathcal{L}^*:V_k\to V_k^*$ by $\langle \mathcal{L}^*w, v\rangle_{V_k^*\times V_k}\coloneqq\langle \mathcal{L} v, w\rangle_{V_k^*\times V_k}$ for all $w,v\in V_{k}$.

\subsection{Mass lumping}\label{stabilisation-temporal}
{In order to obtain a numerical scheme satisfying a discrete maximum principle, we will use mass-lumping of the $L^2$-inner products in the discrete setting. 
We use the standard mass-lumping technique of defining a discrete inner-product by a quadrature approximation of the integral appearing in the $L^2$-inner product.
This leads to a diagonal (lumped) mass matrix for the standard nodal basis of $V_k$.}
{In particular, for each $k\in\N$,} let the inner product $(\cdot,\cdot)_{\Omk}$ on $V_k$ {be defined} by 
\begin{equation}\label{mass-lump-formula}
	(w,v)_{\Omk}\coloneqq\int_{\Omega}I_k(wv)\mathrm{d}x \quad \forall w,v\in V_k,
\end{equation}
where $I_k\colon C(\overline{\Omega})\cap H^1_0(\Omega)\tends V_k$ denotes {the canonical Lagrange interpolation operator, i.e.\ $I_k v = \sum_{i=1}^{M_k}v(x_i)\xi_i $ for all $v\in C(\overline{\Omega})\cap H^1_0(\Omega)$, where $\calVki=\{x_i\}_{i=1}^{M_k}$.}
It is straightforward to check that the mass-lumped inner product satisfies 
\begin{equation}
(w,v)_{\Omk} = \sum_{i=1}^{M_k}(\xi_i,1)_{\Omega}w(x_i)v(x_i)\quad\forall w,v\in V_k.
\end{equation}
Moreover, $(\cdot,\cdot)_{\Omk}$ induces a norm on $V_k$ that is given by $\|w\|_{\Omk}\coloneqq \sqrt{(w,w)_{\Omk}}$ for all $w\in V_k$.
{It is straightforward to show that
\begin{equation}\label{eq:L^2-mass_lumped_bound}
\norm{v}_{\Omega}\leq \norm{v}_{\Omk} \lesssim \norm{v}_\Omega \quad\forall v \in V_k,
\end{equation}
where the hidden constant in the second inequality depends only on the dimension $d$ and the shape-regularity of $\{\Tk\}_{k\in\N}$. {The first inequality follows immediately from the fact that $v^2\leq I_k(v^2)$ in $\Omega$ for all $v\in V_k$. Indeed, since the basis functions $\{\xi_i:1\leq i\leq M_k\}$ are nonnegative everywhere and satisfy $\sum_{i=1}^{M_k}\xi_i|_{K}=1$ for each element $K\in\mathcal{T}_k$, we deduce from Jensen's inequality that 
$$v^2|_{K}=\left(\sum_{i=1}^{M_k}v(x_i)\xi_i|_{K}\right)^2\leq \sum_{i=1}^{M_k}v(x_i)^2\xi_i|_{K}=I_k(v^2)|_K\quad\forall K\in\mathcal{T}_k,$$ for each $v\in V_k$.} As such, $v^2\leq I_k(v^2)$ in $\Omega$ holds for all $v\in V_k$, thereby showing the first inequality in~\eqref{eq:L^2-mass_lumped_bound}. 
We stress that there is no generic constant in the first inequality in~\eqref{eq:L^2-mass_lumped_bound}, which will be important for the weak compactness arguments that will follow in later sections.
The second inequality in~\eqref{eq:L^2-mass_lumped_bound} is easily obtained by a standard scaling and equivalence-of-norms argument.
}
 
For each $k\in\mathbb{N}$, let $R_k:L^2(\Omega)\to V_k$ be the {linear operator} defined by 
\begin{equation}\label{quasi-interp-op}
	R_k w \coloneqq \sum_{i=1}^{M_k}\frac{(\xi_i,w)_{\Omega}}{(\xi_i,1)_{\Omega}}\xi_i\quad\forall w\in L^2(\Omega).
\end{equation}
It is clear that {$R_k$ is a Riesz-map between the $L^2$-inner product and $(\cdot,\cdot)_{\Omk}$,} with 
\begin{equation}\label{eq:Rk_riesz}
(v,R_kw)_{\Omk} = (v,w)_{\Omega} \quad \forall v \in V_k, \quad \forall w \in L^2(\Omega).
\end{equation}
Furthermore, $R_k$ is $L^2$-stable with the bound $\|R_kw\|_{\Omega} \lesssim \|w\|_{\Omega}$ for all $w\in L^2(\Omega)$.
{We also have the well-known bound (see e.g.~\cite[p.~108]{Verfurth2013})
\begin{equation}\label{quasi-interp-approx-property}
	\norm{ 	h_{\Tk}^{-1} (w - R_kw)}_{\Omega} + \norm{\nabla(w-R_k w)}_{\Omega}\lesssim \norm{\nabla w}_{\Omega}\quad\forall w\in H_0^1(\Omega).
\end{equation}
It follows that $R_k$ is also $H^1_0(\Omega)$-stable, i.e.\ $\norm{ \nabla R_k w}_{\Omega}\lesssim \norm{\nabla w}_\Omega$ for all $w \in H^1_0(\Omega)$. }

We now introduce the following discrete dual norm $\norm{\cdot}_{\dualk}\colon V_k\tends \R_{\geq 0}$ defined by
 \begin{equation}\label{eq:dualknorm}
	\|w\|_{\dualk}\coloneqq \sup_{v\in V_k\backslash \{0\}}\frac{(w,v)_{\Omk}}{\|\nabla v\|_{\Omega}} \quad \forall w\in V_k.
\end{equation}

{\begin{lemma}\label{mass-lump-H10-dual-bound}
We have $\|v\|_{H^{-1}(\Omega)}  \eqsim \|v\|_{\dualk}$ for all $v\in V_k$, where the hidden constants depend only on the dimension $d$ and the shape-regularity of $\{\Tk\}_{k\in\N}$.
\end{lemma}
\begin{proof}
For any $v\in V_k$, $k\in\N$, and $w\in H^1_0(\Omega)$, we have $(v,w)_\Omega =(v,R_k w)_{\Omk}$ and thus the definition of $\norm{\cdot}_{\dualk}$ and the $H^1_0$-stability of $R_k$ imply that $(v,w)_\Omega \leq \norm{v}_{\dualk}\norm{\nabla R_k w}_\Omega \lesssim \norm{v}_{\dualk}\norm{\nabla w}_\Omega$. This shows that $\norm{v}_{H^{-1}(\Omega)}\lesssim \norm{v}_{\dualk}$. To show the converse bound, now consider $w\in V_k$, and write $(v,w)_{\Omk}=(v,w)_\Omega+(v,w - R_k w )_{\Omk}$.
The Cauchy--Schwarz inequality and standard inverse inequalities imply that $(v,w - R_k w )_{\Omk} = \int_\Omega I_k(  v(w-R_kw ) )\mathrm{d}x \lesssim \norm{h_{\Tk} v}_\Omega \norm{h_{\Tk}^{-1}(w-R_k w)}_\Omega $. Note that $\norm{h_{\Tk} v}_\Omega \lesssim \norm{v}_{H^{-1}(\Omega)}$, which is deduced from the local bound $\norm{h_{\Tk} v}_K \lesssim \norm{v}_{H^{-1}(K)}$ for all $K\in\Tk$, see e.g.~\cite[p.~113]{Verfurth2013}, and from $\sum_{K\in\Tk}\norm{v}_{H^{-1}(K)}^2\leq \norm{v}_{H^{-1}(\Omega)}^2$.
In conclusion, we then use~\eqref{quasi-interp-approx-property} to find that $(v,w - R_k w )_{\Omk} \leq  \norm{v}_{H^{-1}(\Omega)} \norm{\nabla w}_\Omega$ for all $w\in V_k$, which implies that $\norm{v}_{\dualk}\lesssim \norm{v}_{H^{-1}(\Omega)}$.
\end{proof}}


\subsection{Spatial stabilization}\label{stabilisation-spatial}
We say that a linear operator $L:V_k\to V_k^*$ {satisfies} the \emph{discrete maximum principle} (DMP) provided that the following condition holds: if $v\in V_{k}$ and $\langle Lv,\xi_i\rangle_{V_k^*\times V_k}\geq 0$ for all $ i\in\{1,\cdots,M_k\}$, then $v\geq 0$ in $\Omega$.
Recall that under the condition~\eqref{XZ-condition}, the stiffness matrix of the Laplace operator on $V_k$ is an $M$-matrix, which implies monotonicity and the discrete maximum principle for the Laplace operator on $V_k$.
In order to handle the advective terms of the problem, we construct in this section an original volume-based self-adjoint linear stabilization that maintains the adjoint relationship between the equations of the MFG system.
For each $k\in\N$ and each edge $E\in\calEki$, let $\wEk\geq 0$ be a nonnegative weight to be chosen below, and let $\bm{t}_E$ be a chosen unit tangent vector to the edge $E$. The choice of orientation of $\bm{t}_E$ does not have any effect on the following.
Then, we define the stabilization diffusion matrix $\Dk\in L^\infty(\Omega;\R^{d\times d})$ elementwise over $\Tk$ by
\begin{equation}\label{edge-tensor-formula}
	\Dk|_K\coloneqq \sum_{E\in\mathcal{E}_K}\wEk\bm{t}_E\otimes\bm{t}_E\quad \forall K\in\mathcal{T}_k,
\end{equation}
where $\otimes$ denotes the outer-product of vectors, i.e.\ $\bm{t}_E\otimes\bm{t}_E\coloneqq \bm{t}_E\bm{t}_E^T\in\mathbb{R}^{d\times d}$, and where we recall that $\calEK$ denotes the set of edges of $K$ that are internal. 
Note that $\Dk$ is independent of the choice of orientations of the tangent vectors $\bm{t}_E$.
Since the weights $\wEk$ are nonnegative, it follows that $\Dk|_K$ is a symmetric, positive semi-definite matrix in $\mathbb{R}^{d\times d}$ for each $K\in\mathcal{T}_k$.
Then, let $W(V_k,\Dk)$ denote the set of all linear operators $L:V_{k}\to V_{k}^*$ of the form 
\begin{equation}\label{L-operator}
	\langle L v, w\rangle_{V_k^*\times V_k}\coloneqq \int_{\Omega}A_k\nabla v\cdot\nabla w+\tilde{b}\cdot\nabla v w\,\mathrm{d}x\quad \forall w,v\in V_k,
\end{equation}
where $A_k\coloneqq \nu\mathbb{I}+\Dk$ and where $\tilde{b} \in L^\infty(\Omega;\R^d)$ is some vector field that satisfies $\|\tilde{b}\|_{L^{\infty}(\Omega;\R^\dim)}\leq L_H$. 
{Observe also that $A_k \geq \nu \mathbb{I}$ in the sense of semi-definite matrices since $\Dk$ is positive semi-definite.}
Recall that $\delta$ denotes the shape-regularity parameter of the meshes $\{\mathcal{T}_k\}_{k\in\N}$.
\begin{theorem}[DMP]\label{DMP-edge-stabilisation-result}
	Suppose that the weights in~\eqref{edge-tensor-formula} are chosen such that
	\begin{equation}\label{eq:weight_condition}
			\frac{\delta L_H\diam E}{2(d+1)} < \wEk \lesssim L_H \diam E \quad \forall E \in \mathcal{E}_{k,\Omega}.
	\end{equation}
	Then, the following properties hold:
	\begin{itemize}
		\item $\|\Dk|_K\|_{L^{\infty}(K;\mathbb{R}^{d\times d})}\lesssim h_K$ for all $K\in\mathcal{T}_k$,
 		\item if $L\in W(V_k,\Dk)$ then $L$ and its adjoint $L^*$ satisfy the discrete maximum principle.
	\end{itemize}
\end{theorem} 
The proof of Theorem~\ref{DMP-edge-stabilisation-result} is postponed to Appendix~\ref{sec:app:dmp}.
We remark that $\Dk$ can be chosen independently of the diffusion parameter $\nu$ of the problem.
Furthermore, it is clear that a monotone discretization of the parabolic operators of the MFG system can be obtained by combining the spatial stabilization above with a mass-lumped implicit Euler discretization for the time derivative. Indeed, this is easily seen from the fact that if $L \in W(V_k,\Dk)$, then the operator $V_k\ni w\mapsto \langle L w,\cdot \rangle_{V_k^*\times V_k} + \frac{1}{\tau_k}(w,\cdot)_{\Omk}$ also satisfies the discrete maximum principle for any $\tau_k>0 $.
Finally, note that if $L\colon V_k\tends V_k^*$ satisfies the DMP, then $L$ is necessarily a bijection since $V_k$ is finite dimensional.

\subsection{Space-time finite element spaces}
Given $k\in\mathbb{N}$, let $\mathcal{J}_k\coloneqq \{I_n\}_{n=1}^{\Nk}$ denote the partition of $[0,T]$ where $I_n\coloneqq (t_{n-1},t_n)$ with $t_n\coloneqq n\tau_k$ for $n\in\{1,\cdots,\Nk\}$, $t_0\coloneqq 0$, and the time-step $\tau_k\coloneqq T/\Nk$, with $\Nk \tends \infty$ as $k\tends \infty$. 
To simplify the analysis, we assume that {the temporal meshes are} nested, and thus $\tau_{k+1}\leq \tau_k$ for all $k\in\N$, and we also assume that the initial time-step satisfies
\begin{equation}\label{time-step-size}
	\tau_1<\nu L_H^{-2}.
\end{equation} 
{The above assumption \eqref{time-step-size} on the initial time-step is not essential as we introduce it solely for the sake of brevity of the analysis. For treatment of the analysis in the case where this assumption is not satisfied see \cite[Ch.\ 6]{YohancePhD}.}

For each $k\in \N$, we define $\Vk$ the space of piecewise constant $V_k$-valued functions defined on $[0,T] $ by
\begin{equation}
\mathbb{V}_k\coloneqq \left\{v\in  X,\; v|_{I_n}\in V_k \text{ is constant } \forall n\in\{1,\cdots,\Nk\}\right\}.
\end{equation}
For a function $v\in\Vk$ and $n\in\{1,\cdots,\Nk-1\}$, we let $v(t_n^{-})$ and $v(t_n^{+})$  denote respectively the left- and right-limits of $v$ at time $t_n$.
{A function $v\in\Vk $ also admits a right-limit $v(t_{0}^+)=v|_{I_1}$ at $t_0=0$ and a left-limit $v(t_{\Nk}^-)=v|_{I_{\Nk}}$ at $t_{\Nk}=T$.}
Note that since $v\in \mathbb{V}_k$ is piecewise constant, we have $v(t_n^{-})=v|_{I_n}$ for all $n\in\{1,\dots,\Nk\}$, and that $v(t_{n}^+)=v|_{I_{n+1}}$ for all $n\in\{0,\dots,\Nk-1\}$.
{In the following, it will help to also consider spaces of functions that are defined everywhere on $[0,T]$ by left- or right-continuity.
For a function $v\colon [0,T]\tends V_k$ with $v|_{(0,T)}\in \Vk$, we say that $v$ is left-continuous if $v(t_n)=v(t_{n}^-)$ for all $n=1,\dots,\Nk$, and that $v$ is right-continuous if $v(t_n)=v(t_{n}^+)$ for all $n=0,\dots,\Nk-1$.}
We then define
\begin{subequations}
\begin{align}
&\timeVkforward\coloneqq \left\{v\colon[0,T]\tends V_k,\; v|_{(0,T)}\in \Vk, \text{ and }v\text{ is {left}-continuous}\right\},\\
&\timeVkbackward\coloneqq \left\{v\colon[0,T]\tends V_k,\; v|_{(0,T)}\in \Vk, \text{ and }v\text{ is {right}-continuous}\right\}.
\end{align}
\end{subequations}
Note that here we adopt the point of view that functions in $\Vkpm$ are defined at all points in $[0,T]$, and two functions in $\Vkpm $ that agree up to a subset of measure zero of $[0,T]$ are not identified. When $v\in\timeVkforward$ and $w\in\timeVkbackward$, we set $v(t_0^{-})\coloneqq v(0)$ and $w(t_{\Nk}^+)\coloneqq w(T)$. For $v\in \mathbb{V}_k$, we define the jump $\jump{v}_{n}$ of $v$ at time $t_n\in(0,T)$ by
\begin{equation}
\jump{v}_n\coloneqq v(t_n^{-})-v(t_n^{+}).
\end{equation}
We define also the jump $\jump{v}_{0}=v(0)-v(0^+)$ for $v\in\Vkp$ , and $\jump{w}_{\Nk}=w(T^-)-w(T)$ for $w\in \Vkm$.

\paragraph{Reconstruction of time derivatives}

Let the \emph{forward-in-time} reconstruction operator $\mathcal{I}_+:\timeVkforward\to {C[0,T;V_k]}$ be defined by 
\begin{equation}
(\Ip v)(t)\coloneqq {v(t)+\frac{t_n-t}{\tau_k}\jump{v}_{n-1}},\quad t\in {(t_{n-1},t_n]},\;n\in\{1,\cdots,\Nk\}, \;v \in \Vkp,
\end{equation}
{where $C[0,T;V_k]$ is the space of continuous functions from $[0,T]$ to $V_k$.}
We also define the \emph{backward-in-time} reconstruction operator $\In :\timeVkbackward\to C[0,T;V_k]$ by 
\begin{equation}
 (\In w)(t)\coloneqq {w(t)-\frac{t-t_{n-1}}{\tau_k}\jump{w}_{n}},\quad {t\in [t_{n-1},t_n)},\; n\in\{1,\cdots,\Nk\}, \; w\in\Vkm.
 \end{equation}
As there is no risk of confusion, we omit the dependence of $\Ip$ and $\In$ on $k$ to simplify the notation.
Since $v\in\Vkp$ is by definition left-continuous, it is easy to check that $\Ip v$ is continuous on $[0,T]$ and satisfies $\Ip v(t_{n})=v(t_n)=v(t_{n}^-)$ for all $v\in\Vkp$ and $n\in\{0,\dots,\Nk\}$. In particular, $\Ip v(0)=v(0)$. Similarly, $\In w$ is continuous on $[0,T]$ and $\In w(t_n)=w(t_n)=w(t_{n}^+)$ for all $n\in\{0,\dots,\Nk\}$ and $\In v(T)=v(T)$.
It is also straightforward to verify that $\Ip v$ and $\In w$ are piecewise continuously differentiable with respect to $\mathcal{J}_k$ for any $v\in\Vkp$ and any $w\in\Vkm$, and that
  \begin{equation}\label{derivatives_Ip}
 \begin{aligned}
\p_t \Ip v|_{I_n}= -\frac{1}{\tau_k}\jump{v}_{n-1} , && \p_t \In w|_{I_n} = -\frac{1}{\tau_k} \jump{w}_n, && n\in\{1,\dots,\Nk\}.
\end{aligned}
 \end{equation}
This shows that $\p_t \Ip$ and $\p_t \In$ are respectively related to the first-order backward and forward difference operators.
Furthermore, the operators $\Ip$ and $\In$ satisfy the discrete integration-by-parts formula
	\begin{equation}\label{eq:discrete_ibp}
		\begin{split}
			\int_0^T\left(\partial_t\Ip v,w\right)_{\Omk} + \left(v ,\p_t \In w\right)_{\Omk}\mathrm{d}t=(v(T),w(T))_{\Omk}-(v(0),w(0))_{\Omk}.
		\end{split} 
	\end{equation}
	for	all $(v,w)\in \timeVkforward\times\timeVkbackward$.

\paragraph{Discrete norms}
	Let the norms $\norm{\cdot}_{\Vkp}$ on $\Vkp$ and $\norm{\cdot}_{\Vkm}$ on $\Vkm$ be defined by
\begin{subequations}\label{eq:discrete_vkp_combined}
\begin{align}
&\norm{v}_{\Vkp}^2 \coloneqq \int_0^T \norm{\p_t \Ip v}_{\dualk}^2 + \norm{\nabla v}_\Omega^2 \mathrm{d}t + \norm{v(0)}_{\Omk}^2 \quad \forall v \in \Vkp,\label{eq:discrete_vkp_norm} \\
& \norm{w}_{\Vkm}^2 \coloneqq \int_0^T \norm{\p_t \In w}_{\dualk}^2 + \norm{\nabla w}_\Omega^2 \mathrm{d}t + \norm{w(T)}_{\Omk}^2 \quad \forall w \in \Vkm. \label{eq:discrete_vkm_norm}
\end{align}
\end{subequations}

The following result shows that the norms $\norm{\cdot}_{\Vkp}$ and $\norm{\cdot}_{\Vkm}$ are upper bounds on the $L^\infty(L^2)$ norm.
\begin{lemma}[Boundedness in $L^\infty(L^2)$]\label{lem:vkp_infty_bound}
We have $\norm{v(t)}_{\Omega}\leq \norm{v}_{\Vkpm}$ for all $v\in\Vkpm$, and all $t\in[0,T]$.
\end{lemma}
\begin{proof}
It is enough to prove the result in the case of $\Vkp$ since the case of $\Vkm$ follows the same argument up to a change o{}f the time variable. Let $v\in \Vkp$ be arbitrary.
The inequality $\norm{v(0)}_{\Omega}\leq \norm{v}_{\Vkp}$ follows immediately from~\eqref{eq:L^2-mass_lumped_bound} and~\eqref{eq:discrete_vkp_norm}. It remains only to consider the case $t\in(0,T]$. Since $v$ is piecewise constant in time and left-continuous, there exists a $n\in\{1,\dots, \Nk\}$ such that $t\in (t_{n-1},t_n]$ and~$v(t)=v|_{I_n}$.
The first inequality in~\eqref{eq:L^2-mass_lumped_bound} and the  identity $2\int_0^{t_n} (\p_t \Ip v,v)_{\Omk}\mathrm{d}t = \norm{v|_{I_n}}_{\Omk}^2 + \sum_{i=0}^{n-1} \norm{\jump{v}_i}_{\Omk}^2 - \norm{v(0)}_{\Omk}^2$ for any $n\in\{1,\dots,\Nk\}$ imply that
\[
\norm{v(t)}_{\Omega}^2=\norm{v|_{I_n}}_\Omega^2\leq \norm{v|_{I_n}}_{\Omk}^2\leq 2 \int_0^{t_n} (\p_t \Ip v,v)_{\Omk}\mathrm{d}t+ \norm{v(0)}_{\Omk}^2 \leq \norm{v}_{\Vkp}^2,
\] 
where in the final inequality we have used $(\p_t \Ip v,v)_{\Omk}\leq \norm{\p_t \Ip v}_{\dualk}\norm{\nabla v}_\Omega$ for a.e $t\in (0,T)$ and Young's inequality.
\end{proof}

\section{Finite element approximation and main results}\label{sec5}
We now introduce the finite element discretization of~\eqref{eq:weakform-space-time}.
For each $k\in \N$, let $\Dk \in L^\infty(\Omega;\R^{\dim\times\dim})$ of the form~\eqref{edge-tensor-formula} be as in~Theorem~\ref{DMP-edge-stabilisation-result}, and recall that $A_k\coloneqq \nu \mathbb{I}+\Dk$ and that $\norm{\Dk}_{L^\infty(\Omega;\R^{\dim\times\dim})}\lesssim h_k$.
For each $k\in\mathbb{N}$, the discrete problem is to find $(u_{k},m_{k})\in {\Vk}\times\timeVkforward$ such that, {for some} $\tilde{b}_{k}\in {\Dp} H[u_{k}]$, 
\begin{subequations}\label{weakform-space-time-discrete}
	\begin{gather}
		\begin{split} 
			\int_0^T\left(\partial_t\mathcal{I}_+\psi,u_{k}\right)_{\Omk}+(A_k\nabla u_{k},\nabla \psi)_{\Omega}+&(H[\nabla u_{k}],\psi)_{\Omega} \mathrm{d}t
			\\
			=&\int_0^T\langle {F}[m_{k}],\psi\rangle\mathrm{d}t
			+\left(R_kS[m_{k}(T)],\psi(T)\right)_{\Omk},
		\end{split}\label{weakform1-space-time-discrete}
		\\ 
		\begin{split}
		\int_0^T\left(\partial_t\mathcal{I}_+m_{k},\phi\right)_{\Omk}+(A_k\nabla m_{k},\nabla \phi)_{\Omega}+(m_{k}\tilde{b}_{k},\nabla \phi)_{\Omega} \mathrm{d}t
		=\int_0^T\langle {G},\phi\rangle_{}\mathrm{d}t,
		\end{split}\label{weakform2-space-time-discrete}
	\end{gather} 
\end{subequations}
for all $(\psi,\phi)\in {\mathbb{V}_{k,0}^{+}}\times \mathbb{V}_{k}$, and such that {$m_{k}(0) =  R_k m_0 \in V_k$.} 

\begin{remark}
The FEM is presented above in a space-time variational form, as this reflects the coupled space-time nature of the problem, offers a succinct notation, and is also most convenient for the discrete functional analytic framework set out below.
By choosing test functions supported on individual time-step intervals, it is easily seen that the discretization consists of a backward-in-time implicit Euler discretization  for the HJB equation and a forward-in-time implicit Euler discretization for the KFP equation. Therefore, for practical computation the system can be re-written in a coupled forward-backward time-stepping form. 
The efficient solution of these coupled discrete systems remains an important challenge which is beyond the scope of this work, however one promising approach is to consider time- or space-time-parallel solvers, see~\cite{NeumullerSmears19} and the references therein. 
\end{remark}

	Observe that if $(u_k,m_k)\in \Vk\times \Vkp $ solves~\eqref{weakform-space-time-discrete}, then we can adopt a different point of view and re-define $u_k$ to a right-continuous function in $\Vkm$ (without change of notation) by setting $u_k(T)\coloneqq R_k S[m_k(T)]$ and $u_k(t_{n}^{+})\coloneqq u_k|_{I_n}$ for all $n\in\{0,\dots,M_{k}-1\}$.
	 Using the discrete integration-by-parts identity~\eqref{eq:discrete_ibp}, we can equivalently rewrite~\eqref{weakform1-space-time-discrete} as
	\begin{equation}\label{eq:strong_in_time_discrete_HJB}
	\int_0^T - \left(\psi,\p_t\In u_k\right)_{\Omk}+(A_k\nabla u_{k},\nabla \psi)_{\Omega}+(H[\nabla u_{k}],\psi)_{\Omega} \mathrm{d}t= \int_0^T\langle {F}[m_{k}],\psi\rangle \mathrm{d}t
	\end{equation}
	for all $\psi \in \mathbb{V}_{k,0}^+$. Since the terms in~\eqref{eq:strong_in_time_discrete_HJB} depend only on $\psi|_{(0,T)}$, it is clear that the set of test functions in~\eqref{eq:strong_in_time_discrete_HJB} can then be extended to all $\psi \in \Vk$.
	In practical terms for computation, there is no difference between these two points of view.

\subsection{Main results}
The first main result states the existence of a  solution to the discrete problem and gives a stability bound for any such discrete solution in terms of the data of the problem.
 \begin{theorem}[Existence of discrete solutions {and stability bound}]\label{discrete-mfg-existence}
	{F}or every $k\in\mathbb{N}$ there exists a discrete pair $(u_{k},m_{k})\in \timeVkbackward\times\timeVkforward$ satisfying \eqref{weakform-space-time-discrete}.
{Any solution $(u_k,m_k)$ of \eqref{weakform-space-time-discrete} satisfies
\begin{equation}\label{eq:discrete_mfg_existence}
\norm{u_k}_{\Vkm}+\norm{m_k}_{\Vkp} \lesssim 1 + \norm{G}_{L^2(0,T;H^{-1}(\Omega))}+\norm{m_0}_\Omega.
\end{equation}}
\end{theorem}
The proof of Theorem~\ref{discrete-mfg-existence} is given in Section~\ref{sec6}.
The next main result shows that the uniqueness of the numerical solution holds under the same assumptions as those that were considered in the continuous setting of Theorem~\ref{uniqueness-space-time}.
\begin{theorem}[Uniqueness]\label{discrete-mfg-uniqueness}
	Suppose that the initial density $m_0\in \LOm$ is nonnegative a.e.\ in $\Omega$, the source term $G$ is nonnegative in the sense of distributions in $L^2(0,T;H^{-1}(\Omega))$, the coupling term $F$ is strictly monotone on $X$, and that the terminal cost $S$ is monotone on $\LOm$. Then, for every $k\in\mathbb{N}$, there exists a unique solution  $(u_{k},m_{k})\in \timeVkbackward\times\timeVkforward$ to \eqref{weakform-space-time-discrete}. 
\end{theorem}
The proof of Theorem~\ref{discrete-mfg-uniqueness} is given in Section~\ref{sec6}.

Our main convergence result shows that strong convergence of the value function approximations in $L^2(H^1_0)$, and of the density approximations in $L^p(L^2)$ for any $1\leq p < \infty$, along with weak convergence of the density approximations in $L^2(H^1_0)$.

\begin{theorem}[Convergence of FEM]\label{conv-main-thm}
	Assume that the hypotheses of Theorem \ref{discrete-mfg-uniqueness} hold. Then, there exists a unique solution $(u,m)\in Y\times Y$ in the sense of Definition \ref{weakdef-space-time}. Moreover, we have
	\begin{subequations}\label{eq:eulerFEMconv}
	\begin{align}
		&{u}_{k} \to u\quad\text{in}\quad L^2(0,T;H_0^1(\Omega)), && {u}_{k} \to u\quad\text{in}\quad L^p(0,T;\LOm),\label{u-conv-3}
		\\
		&{m}_{k} \rightharpoonup m\quad\text{in}\quad L^2(0,T;H_0^1(\Omega)), && {m}_{k}\to m \quad \text{in}\quad  L^p(0,T;\LOm),\label{m-conv-2}
		\\
		&{u}_{k}(0) \to u(0)\quad\text{in}\quad \LOm, &&{m}_{k}(T)\to m(T) \quad \text{in}\quad  \LOm,\label{m-conv-3}
	\end{align}
	\end{subequations}
	as $k\tends \infty$, for any $1\leq p<\infty$. 
\end{theorem}
The proof of Theorem~\ref{conv-main-thm} is given in~Section~\ref{sec8:proofs_of_convergence} where we present a compactness argument that uses functional analytic tools developed in Section \ref{sec7}. {We note that, if the conditions for uniqueness given in Theorem \ref{discrete-mfg-uniqueness} are absent, then the proof we present still establishes the convergence \eqref{eq:eulerFEMconv} along subsequences.}

Note that $m_k\rightharpoonup m$ in $L^2(0,T;H_0^1(\Omega))$ implies that $\nabla m_k$ converges weakly to $\nabla m$ in $L^2(\QT;\R^d)$. At present, it is not currently known if the weak convergence of $\nabla m_{k}$ to $\nabla m$ is also strong in general. 
The difficulty lies in the fact that for nondifferentiable Hamiltonians, the convergence $\nabla u_k \tends \nabla u$ is not sufficient to prove convergence in sufficiently strong norms of $\tilde{b}_{k}\in \mathcal{D}_pH[u_k]$ from \eqref{weakform-space-time-discrete} to $\tilde{b}_* \in \mathcal{D}_p H[u]$ from~\eqref{eq:weakform-space-time}.
The next result shows that if some pre-compactness of $\{\tilde{b}_{k}\}_{k\in\mathbb{N}}$ is assumed, then we obtain strong convergence of approximations of the gradient of the density.

\begin{corollary}[Strong $L^2(H_0^1)$-Convergence for Density Approximations]\label{convergence-cor}
	In addition to the hypotheses of Theorem \ref{conv-main-thm}, suppose that the sequence of transport vector fields $\{\tilde{b}_{k}\}_{k\in\mathbb{N}}$ from \eqref{weakform-space-time-discrete} is pre-compact in $L^1(Q_T;\R^\dim)$. Then, $m_{k}$ converges to $m$ strongly in $ L^2(0,T;H_0^1(\Omega))$ as $k\to \infty$.
\end{corollary}
The proof of Corollary~\ref{convergence-cor} is in~Section~\ref{sec8:proofs_of_convergence}. 
Note that the pre-compactness of $\{\tilde{b}_{k}\}_{k\in\mathbb{N}}$ is assured, for instance, if the Hamiltonian $H$ given by \eqref{H-time-homogeneous} is such that partial derivative $\frac{\partial H}{\partial p}$ exists and is continuous in $Q_T\times\R^\dim$. 
Thus, in the case of continuously differentiable Hamiltonians, Corollary~\ref{convergence-cor} shows that $m_k\tends m$ in norm in $ L^2(0,T;H_0^1(\Omega))$ as $k\tends \infty$.

\section{Proof of Theorems~\ref{discrete-mfg-existence} and \ref{discrete-mfg-uniqueness}}\label{sec6}

In this section we give the proofs of the existence and uniqueness results for the discrete problem \eqref{weakform-space-time-discrete}.

\subsection{Well-posedness of discrete HJB and KFP equations}\label{sec:discFPHJB}

In Lemmas~\ref{KFP_wellposedness_discrete} and~\ref{HJB_wellposedness_discrete} below, we state well-posedness and stability results for the discretized KFP and HJB equations when considered separately. The proofs of these results are given in Appendix~\ref{sec:app:discrete_wellposedness}.
Recall that we make the assumption that $\tau_k\leq \tau_1$ for all $k\in\N$, with $\tau_1$ satisfying \eqref{time-step-size}, in order to simplify the analysis. 

\begin{lemma}[Well-posedness of discrete KFP equation]\label{KFP_wellposedness_discrete}
Let $k\in\N$, let $G\in L^2(0,T;H^{-1}(\Omega))$, let $g\in V_k$ and let $\tilde{b}\in L^\infty(\QT;\R^\dim)$ such that $\norm{\tilde{b}}_{L^\infty(\QT;\R^\dim)}\leq L_H$. Then, there exists a unique $\overline{m}\in \Vkp$ such that $\overline{m}(0)=g$ and such that
\begin{equation}\label{KFP_eqn_discrete}
\begin{aligned}
		\int_0^T\left( \partial_t\mathcal{I}_+\overline{m},\phi\right)_{\Omk}+(A_k\nabla \overline{m},\nabla\phi)_{\Omega}+(\overline{m},\tilde{b}{\cdot}\nabla \phi)_\Omega\mathrm{d}t=\int_0^T\langle G,\phi\rangle \mathrm{d}t &&& \forall \phi \in \Vk.
		\end{aligned}
	\end{equation}
We have the bound
\begin{equation}\label{eq:KFP_apriori_bound}
\norm{\overline{m}}_{\Vkp} \lesssim \norm{G}_{L^2(0,T;H^{-1}(\Omega))}+\norm{g}_{\Omega},
\end{equation}
where the hidden constant depends on $\tau_1$ but is otherwise independent of $k$.
Furthermore, if $g\geq 0 $ a.e.\ in $\Omega$ and if $G$ is nonnegative in the sense of distributions, then $\overline{m}\geq 0$ a.e.\ in $\QT$.
\end{lemma}

\begin{lemma}[Well-posedness of discrete HJB equation]\label{HJB_wellposedness_discrete}
Let $k\in\N$, let $\widetilde{F}\in L^2(0,T;H^{-1}(\Omega))$ and let $\widetilde{S}\in L^2(\Omega)$.
Then, there exists a unique $\overline{u}\in \Vkm$ such that $\overline{u}(T)=R_k\widetilde{S}$ and such that
	\begin{multline}\label{HJB_eqn_discrete}
			\int_0^{T}\left( \partial_t\mathcal{I}_+\psi,\overline{u}\right)_{\Omk}+(A_k\nabla\overline{u},\nabla\psi)_\Omega +({H}[\nabla \overline{u}], \psi)_{\Omega}\mathrm{d}t
			\\=\int_0^T\langle \widetilde{F},\psi\rangle\mathrm{d}t +(R_k\widetilde{S},\psi(T))_{\Omk} \quad \forall \psi \in \mathbb{V}_{k,0}^+.
	\end{multline}
	Furthermore $\overline{u}$ depends continuously on the data: if $\overline{u}_1,\,\overline{u}_2\in \Vkm $ are respective solutions of \eqref{HJB_eqn_discrete} with data $(\widetilde{F}_1,\widetilde{S}_1)$ and $(\widetilde{F}_2,\widetilde{S}_2)$ respectively and if $\overline{u}_1(T) = R_k \widetilde{S}_1$ and $\overline{u}_2(T) = R_k \widetilde{S}_2$, then
	\begin{equation}\label{HJB_cont_dep_discrete}
		\norm{\overline{u}_1-\overline{u}_2}_{\Vkm}\lesssim \norm{\widetilde{F}_1-\widetilde{F}_2}_{L^2(0,T;H^{-1}(\Omega))}+\norm{\widetilde{S}_1-\widetilde{S}_2}_\Omega,
	\end{equation}
	where the hidden constant depends on $\tau_1$ but is otherwise independent of $k$.
\end{lemma}

{\subsection{Proof of Theorem~\ref{discrete-mfg-existence}}
{We now prove Theorem~\ref{discrete-mfg-existence} by an argument that is based on  Kakutani's fixed point theorem \cite[Ch.\ 9, Theorem 9.B]{ZMR0816732}. 
\begin{theorem}[Kakutani's fixed point theorem \cite{ZMR0816732}]\label{thm-kakutani}
    Suppose that 
    \begin{enumerate}
        \item $\mathcal{B}$ is a nonempty, compact, convex set in a locally convex space $X$;
        \item $\mathcal{V}:\mathcal{B}\rightrightarrows\mathcal{B}$ is a set-valued map such that $\mathcal{V}[\tilde{b}]$ is nonempty, closed and convex for all $\tilde{b}\in\mathcal{B}$; and
        \item $\mathcal{V}$ is upper semi-continuous.
    \end{enumerate}
    Then, $\mathcal{V}$ has a fixed point: there exists $\tilde{b}\in\mathcal{B}$ such that $\tilde{b}\in \mathcal{V}[\tilde{b}]$. 
\end{theorem}

To begin, recall that $L_H$ is the Lipschitz constant of the Hamiltonian, c.f.~\eqref{eq:Lipschitz_H}.
We equip the space $L^{\infty}(Q_T;\mathbb{R}^d)$ with its weak-$*$ topology, noting that it is then a locally convex topological vector space. Let $\mathcal{B}$ denote the ball 
    \begin{equation}
        \mathcal{B}\coloneqq \left\{\tilde{b}\in L^{\infty}(Q_T;\mathbb{R}^d): \|\tilde{b}\|_{L^{\infty}(Q_T;\mathbb{R}^d)}\leq L_H\right\},
    \end{equation}
    and note that $\mathcal{B}$ is nonempty, convex, and closed in the weak-$*$ topology. We note further that the weak-$*$ topology on $\mathcal{B}$ is metrisable since $L^1(Q_T;\mathbb{R}^d)$ is separable \cite[Ch.\ 15]{royden2018real}, and that $\mathcal{B}$ is compact by Helly's theorem.

    Let $\mathcal{M}_k: \mathcal{B}\to \mathbb{V}_k^+$ be the map defined as follows: for each $\tilde{b}\in \mathcal{B}$, let $\mathcal{M}_k[\tilde{b}]$ in $\mathbb{V}_k^+$ be the unique solution of
    \begin{multline}\label{M-map-def-time-k}
\int_0^T(\partial_t\mathcal{I}_+\mathcal{M}_k[\tilde{b}],\phi)_{\Omk}+(A_k\nabla \mathcal{M}_k[\tilde{b}],\nabla\phi)_{\Omega}+(\mathcal{M}_k[\tilde{b}],\tilde{b}{\cdot}\nabla \phi)_{\Omega} \mathrm{d}t
\\ =\int_0^T\langle {G},\phi\rangle\mathrm{d}t \quad \forall \phi \in \mathbb{V}_k
\end{multline}
    with $\mathcal{M}_k[\tilde{b}](0)=R_km_0$ in $V_k$. The map $\mathcal{M}_k$ is well-defined thanks to Lemma~\ref{KFP_wellposedness_discrete}. Next, let $\mathcal{U}_k:\mathbb{V}_k^+\to \mathbb{V}_k^-$ be the map defined as follows: for each ${m}_k\in \mathbb{V}_k^+$, let $\mathcal{U}_k[{m}_k]\in \mathbb{V}_k^-$ denote the unique solution of
    \begin{multline}\label{U-map-def-time-k}
\int_0^T(\partial_t\mathcal{I}_+\psi,\mathcal{U}_k[m_k])_{\Omk}+(A_k \nabla \mathcal{U}_k[m_k],\nabla\psi)_{\Omega}+({{H}[\nabla \mathcal{U}_k[m_k]]},\psi)_{\Omega} \mathrm{d}t
\\ =\int_0^T\langle {F}[m_{k}],\psi\rangle\mathrm{d}t 	+(R_kS[m_{k}(T)],\psi(T))_{\Omk} \quad \forall \psi \in \mathbb{V}_{k,0}^+,
    \end{multline}
    with $\mathcal{U}_k[m_k](T)=R_kS[m_{k}(T)]$ in $V_k$.
    The map $\mathcal{U}_k$ is well-defined by Lemma~\ref{HJB_wellposedness_discrete}. 

    Now, we define the set-valued map $\mathcal{K}_k:\mathcal{B}\rightrightarrows L^{\infty}(Q_T;\mathbb{R}^d)$ as follows: for each $\tilde{b}\in \mathcal{B}$, let 
    \begin{equation}
        \mathcal{K}_k[\tilde{b}]\coloneqq  \mathcal{D}_pH\big[\mathcal{U}_k\big[\mathcal{M}_k\big[\tilde{b}\big]\big]\big].
    \end{equation}
    Lemma \ref{Prop1-time} implies that $\mathcal{K}_k[\tilde{b}]\subset\mathcal{B}$ for each $\tilde{b}\in\mathcal{B}$, so $\mathcal{K}_k:\mathcal{B}\rightrightarrows \mathcal{B}$. Moreover, for every $\tilde{b}\in \mathcal{B}$, the set $\mathcal{K}_k[\tilde{b}]$ is nonempty and convex. Indeed, for each $\tilde{b}\in\mathcal{B}$ the set $\mathcal{K}_k[\tilde{b}]$ is nonempty by Lemma \ref{Prop1-time}. Also, $\mathcal{K}_k[\tilde{b}]$ is convex since $\partial_pH$ has convex images. Furthermore, $\mathcal{K}_k[\tilde{b}]$ is closed for all $\tilde{b}\in\mathcal{B}$ thanks to Lemma \ref{closure} and the fact that $\mathcal{B}\subset L^2(Q_T;\mathbb{R}^d)$ where $|Q_T|_{d+1}<\infty$. 
    
    The existence of a solution to the discrete problem \eqref{weakform-space-time-discrete} is equivalent to showing the existence of a fixed point of $\mathcal{K}_k$, i.e.\ that there exists a $\tilde{b}_k\in \mathcal{B}$ such that $\tilde{b}_k\in\mathcal{K}_k[\tilde{b}_k].$ Indeed, if $\tilde{b}_k\in \mathcal{B}$ satisfies $\tilde{b}_k\in\mathcal{K}_k[\tilde{b}_k]$ then a solution pair $(u_k,m_k)$ of the discrete problem \eqref{weakform-space-time-discrete} is given by $m_k\coloneqq \mathcal{M}_k[\tilde{b}_k]$ and $u_k\coloneqq \mathcal{U}_k[m_k]$ with $\tilde{b}_k\in \mathcal{D}_pH[u_k]$, while the converse is obvious. 
    
    We now verify that $\mathcal{K}_k$ is upper semi-continuous. To this end, it suffices to prove that the graph of $\mathcal{K}_k$ is closed; c.f.\ \cite[Ch.\ 1, Corollary 1, p.\ 42]{MR0755330}.  Let $\mathcal{W}_k$ denote the graph of $\mathcal{K}_k$, which is defined by 
    \begin{equation}\label{graph-def-time-k}
        \mathcal{W}_k\coloneqq \left\{(\tilde{b},\overline{b})\in \mathcal{B}\times\mathcal{B}:\overline{b}\in\mathcal{K}_k[\tilde{b}]\right\}.
    \end{equation}
    Since $\mathcal{B}$ is metrisable, to show that the graph $\mathcal{W}_k$ is a closed it is enough to show that whenever a sequence $\{(\tilde{b}_i,\overline{b}_i)\}_{i\in\mathbb{N}}\subset \mathcal{W}_k$ converges weakly-$*$ in $\mathcal{B}\times\mathcal{B}$ to a point $(\tilde{b},\overline{b})$ as $i\to\infty$, then $(\tilde{b},\overline{b})\in\mathcal{W}_k$,  which is equivalent to $\overline{b}\in \mathcal{K}_k[\tilde{b}]$. Let us then suppose that we are given a sequence $\{(\tilde{b}_i,\overline{b}_i)\}_{i\in\mathbb{N}}\subset \mathcal{W}_k$ that converges weakly-$*$ in $\mathcal{B}\times\mathcal{B}$ to a point $(\tilde{b},\overline{b})$ as $i\to\infty$. To begin, we claim that $\mathcal{M}_k[\tilde{b}_i]\to \mathcal{M}_k[\tilde{b}]$ in $L^2(0,T;L^2(\Omega))$ and $\mathcal{M}_k[\tilde{b}_i](T)\to \mathcal{M}_k[\tilde{b}](T)$ in $L^2(\Omega)$ as $i\to\infty$. Indeed, since $\{\tilde{b}_i\}_{i\in\mathbb{N}}\subset\mathcal{B}$, for each $i\in\mathbb{N}$ we apply Lemma \ref{KFP_wellposedness_discrete} and the norm equivalence \eqref{eq:L^2-mass_lumped_bound} to obtain the uniform bound
    \begin{equation}\label{M_n-time-k}
        \sup_{i\in\mathbb{N}}\|\mathcal{M}_k[\tilde{b}_i]\|_{\mathbb{V}_k^+}\lesssim \|m_0\|_{\Omega}+\|G\|_{L^2(0,T;H^{-1}(\Omega))}.
    \end{equation}
    Since $\mathbb{V}_k^+$ is finite dimensional, this uniform bound implies that any given subsequence $\{\mathcal{M}_k[\tilde{b}_{i_j}]\}_{j\in\mathbb{N}}$ contains a further subsequence $\{\mathcal{M}_k[\tilde{b}_{i_{j_s}}]\}_{s\in\mathbb{N}}$ such that $$\mathcal{M}_k[\tilde{b}_{i_{j_s}}]\to {m}_k\quad\text{ in }\mathbb{V}_k^+$$ as $s\to\infty$, for some ${m}_k\in\mathbb{V}_k^+$. In particular, it follows from the $L^{\infty}(L^2)$-bound in Lemma~\ref{lem:vkp_infty_bound} that $m_k(T)=\lim_{s\tends \infty} \mathcal{M}_k[\tilde{b}_{i_{j_s}}](T)$ in $L^2(\Omega)$ and $m_k(0)=R_k m_0$ in $V_k$. We then pass to the limit in the equation \eqref{M-map-def-time-k} satisfied by $\mathcal{M}_k[\tilde{b}_{i_{j_s}}]$ to find that $m_k\in \mathbb{V}_k^+$ satisfies 
    \begin{multline}\label{m-alt-eqn-time-k}
\int_0^T\left(\partial_t\mathcal{I}_+m_k,\phi\right)_{\Omk}+(A_k\nabla m_k,\nabla\phi)_{\Omega}+(m_k,\tilde{b}{\cdot}\nabla \phi)_{\Omega} \mathrm{d}t
 =\int_0^T\langle {G},\phi\rangle\mathrm{d}t \quad \forall \phi \in \mathbb{V}_k
\end{multline} with $m_k(0)=R_k m_0$ in $V_k$.
     But we then see from the definition of $\mathcal{M}_k[\tilde{b}]$ in \eqref{M-map-def-time-k} that $m_k=\mathcal{M}_k[\tilde{b}]$ in $\mathbb{V}_k^+$. It follows that the entire sequence $\{\mathcal{M}_k[\tilde{b}_i]\}_{i\in\mathbb{N}}$ satisfies $\mathcal{M}_k[\tilde{b}_i]\to \mathcal{M}_k[\tilde{b}]$ in $\mathbb{V}_k^+$ and $\mathcal{M}_k[\tilde{b}_i](T)\to \mathcal{M}_k[\tilde{b}](T)$ in $L^2(\Omega)$ as $i\to\infty$. In particular, the strong convergence of  $\{\mathcal{M}_k[\tilde{b}_i]\}_{i\in\mathbb{N}}$ to $\mathcal{M}_k[\tilde{b}]$ in $\mathbb{V}_k^+$, together with the $L^{\infty}(L^2)$-bound in Lemma~\ref{lem:vkp_infty_bound}, implies that $\mathcal{M}_k[\tilde{b}_i]\to \mathcal{M}_k[\tilde{b}]$ in $L^2(0,T;L^2(\Omega))$ as $i\to\infty$. We deduce from continuity of $S:L^2(\Omega)\to L^2(\Omega)$ that $S[\mathcal{M}_k[\tilde{b}_i](T)]\to S[\mathcal{M}_k[\tilde{b}](T)]$ in $L^2(\Omega)$ as $i\to\infty$, while the continuity of $F:L^2(0,T;L^2(\Omega))\to L^2(0,T;H^{-1}(\Omega))$ implies that $F[\mathcal{M}_k[\tilde{b}_i]]\to F[\mathcal{M}_k[\tilde{b}]]$ in $L^2(0,T;H^{-1}(\Omega))$ as $i\to\infty$. We then apply the continuous dependence result of Lemma \ref{HJB_wellposedness_discrete} to conclude that $\mathcal{U}_k[\mathcal{M}_k[\tilde{b}_i]]\to \mathcal{U}_k[\mathcal{M}_k[\tilde{b}]]$ in $\mathbb{V}_k^-$ as $i\to\infty$ with $\mathcal{U}_k[\mathcal{M}_k[\tilde{b}]](T)=R_kS[\mathcal{M}_k[\tilde{b}](T)]$ in $V_k$. By hypothesis, $\overline{b}_i\in \mathcal{K}_k[\tilde{b}_i]=\mathcal{D}_pH[\mathcal{U}_k[\mathcal{M}_k[\tilde{b}_i]]]$ for $i\in\mathbb{N}$ and $\overline{b}_i\rightharpoonup^* \overline{b}$ in $L^{\infty}(Q_T;\mathbb{R}^d)$ as $i\to\infty$. In particular, $\overline{b}_i\rightharpoonup \overline{b}$ in $L^{2}(Q_T;\mathbb{R}^d)$ as $i\to\infty$ since $\{\overline{b}_i\}_{i\in\mathbb{N}}\subset\mathcal{B}\subset L^2(Q_T;\mathbb{R}^d)$ and $|Q_T|_{d+1}<\infty$. We conclude from Lemma \ref{closure} that $\overline{b}\in \mathcal{D}_pH[\mathcal{U}_k[\mathcal{M}_k[\tilde{b}]]]$, i.e.\ $\overline{b}\in \mathcal{K}_k[\tilde{b}]$. We have therefore shown that $\mathcal{W}_k$ is closed, so $\mathcal{K}_k$ is upper semi-continuous.
    
    We have thus shown that the map $\mathcal{K}_k:\mathcal{B}\rightrightarrows\mathcal{B}$ satisfies the conditions of Kakutani's fixed point theorem, so $\mathcal{K}_k$ admits a fixed point and therefore there exists a solution to the discrete problem \eqref{weakform-space-time-discrete}, as required.

    Finally, to deduce the stability bound \eqref{eq:discrete_mfg_existence} satisfied by solutions $(u_k,m_k)$ to the FEM \eqref{weakform-space-time-discrete}, we first observe that Lemma \ref{KFP_wellposedness_discrete} and the discrete KFP equation satisfied by $m_k$ with drift in $\mathcal{B}$ imply that $$\sup_{k\in\mathbb{N}}\|m_k\|_{\mathbb{V}_k^+}\lesssim \|m_0\|_{\Omega}+\|G\|_{L^2(0,T;H^{-1}(\Omega))}.$$ This bound, the linear growth of $F$ and $S$, together with the $L^{\infty}(L^2)$-bound in Lemma \ref{lem:vkp_infty_bound}, allow us to apply \eqref{HJB_cont_dep_discrete} to obtain 
    $$\sup_{k\in\mathbb{N}}\|u_k\|_{\mathbb{V}_k^-}\lesssim \|m_0\|_{\Omega}+\|G\|_{L^2(0,T;H^{-1}(\Omega))}+1.$$
    We then deduce the stability bound \eqref{eq:discrete_mfg_existence}, as desired.
	\hfill\proofbox
	}}

\subsection{Proof of Theorem~\ref{discrete-mfg-uniqueness}}
Suppose that $(u_{k,i},m_{k,i})\in \Vk\times\timeVkforward$, $i\in\{1,2\}$ are solutions of~\eqref{weakform-space-time-discrete}, {for some respective $\tilde{b}_{k,i}\in \DpH[u_{k,i}]$.}
{It follows from the nonnegativity of $m_0$ and of $G$ and from the discrete maximum principle (see Theorem~\ref{DMP-edge-stabilisation-result}), that $m_{k,i}\geq 0$ a.e.\ in $\QT$ for each $i\in\{1,2\}$.
Observe also that $m_{k,1}(0)=m_{k,2}(0)$ so that $m_{k,1}-m_{k,2} \in \mathbb{V}_{k,0}^+$. 
Note that $(R_kS[m_{k,i}(T)],m_{k,j}(T))_{\Omk}=(S[m_{k,i}(T)],m_k^{(j)}(T))_\Omega$ for all $i,\, j\in\{1,2\}$ by~\eqref{eq:Rk_riesz}.
Then, proceeding similarly to the proof of Theorem~\ref{uniqueness-space-time}, we test~\eqref{weakform1-space-time-discrete} with $\psi=m_k^{(1)}-m_{k}^{(2)} \in \mathbb{V}_{k,0}^+$ and \eqref{weakform2-space-time-discrete} with $\phi=u_{k,1}-u_{k,2}$ to find that}
\begin{multline}\label{discrete-difference} 
		\int_0^{T}\int_{\Omega}m_1\lambda_{k,1,2}+m_2\lambda_{k,2,1}\mathrm{d}x\mathrm{d}t =  \int_0^T\langle {F}[m_{k,1}]-{F}[m_{k,2}], m_{k,1}-m_{k,2}\rangle \mathrm{d}t
		\\ +({S}[m_{k,1}(T)]-{S}[m_{k,2}(T)], m_{k,1}(T)-m_{k,2}(T))_{\Omega},
\end{multline}
where
\begin{equation}
\lambda_{k,i,j}\coloneqq -H[\nabla {u}_{k,j}]+H[\nabla u_{k,i}]+ \tilde{b}_{k,i} \cdot \nabla( u_{k,j}-u_{k,i}), \quad i,\,j\in\{1,2\}.
\end{equation}
{Note that $\lambda_{k,i,j}\leq 0 $ a.e.\ in $\QT$ for all $i,\,j \in \{1,2\}$ since $\tilde{b}_{k,i}\in \DpH[u_{k,i}]$.
Therefore $\int_0^{T}\int_{\Omega}m_1\lambda_{k,1,2}+m_2\lambda_{k,2,1}\mathrm{d}x\mathrm{d}t \leq 0$, and so we conclude from strict monotonicity of $F$ and monotonicity of $S$ applied to~\eqref{discrete-difference} that $m_{k,1}=m_{k,2}$ a.e.\ in $\QT$.
Since functions in~$\Vkp$ are left-continuous, it follows that $m_{k,1}(T)=m_{k,2}(T)$.
Therefore $u_{k,1}$ and $u_{k,2}$ both satisfy~\eqref{weakform1-space-time-discrete} with common right hand side, so Lemma~\ref{HJB_wellposedness_discrete} implies that $u_{k,1}=u_{k,2}$.	\hfill\proofbox
}


{
\section{Discrete functional analysis}\label{sec7}
In this section we establish auxiliary compactness results that will be used in the convergence analysis of the finite element scheme \eqref{weakform-space-time-discrete}. We begin by collecting some fundamental weak convergence results for the spaces $\mathbb{V}_k^{\pm}$. 
\begin{theorem}[Weak convergence properties]\label{thm:weak_convergence}
Let $\{v_k\}_{k\in\N}$ be a sequence such that $v_k\in\mathbb{V}_k^+$ for each $k\in\N$ and $\sup_{k\in\N}\norm{v_k}_{\mathbb{V}_k^+}<\infty$. Then, there exists a $v\in Y$ and a subsequence $\{v_{k_j}\}_{j\in\N}$ such that 
\begin{equation}\label{eq:weak_convergence}
\begin{aligned}
&v_{k_j} \rightharpoonup v \text{ in } X, && \p_t \mathcal{I}_+ v_{k_j}  \rightharpoonup \p_t v \text{ in } L^2(0,T;H^{-1}(\Omega)),\\
& v_{k_j} (T) \rightharpoonup v(T) \text{ in } \LOm, && v_{k_j} (0)\rightharpoonup v(0) \text{ in } \LOm,
\end{aligned}
\end{equation}
as $j\tends \infty$.
Furthermore, for any fixed $w_\ell \in \mathbb{V}_\ell$, $\ell\in\N$, there holds 
	\begin{equation}\label{eq:weak_convergence_mass_lumped}
		\lim_{j\to\infty}\int_0^T\left(\p_t \mathcal{I}_+ v_{k_j},w_\ell\right)_{\Omkj}\mathrm{d}t = \int_0^T\langle\partial_t v,w_{\ell} \rangle \mathrm{d}t.
	\end{equation} 
{The properties \eqref{eq:weak_convergence} and \eqref{eq:weak_convergence_mass_lumped} hold analogously for sequences $\{\tilde{v}_k\}_{k\in\mathbb{N}}$ such that $\tilde{v}_k\in \mathbb{V}_k^-$ for each $k\in\mathbb{N}$ and  $\sup_{k\in\N}\norm{\tilde{v}_k}_{\mathbb{V}_k^-}<\infty$, with $\mathcal{I}_-$ in place of $\mathcal{I}_+$.}
\end{theorem}
\begin{proof}
We give the proof only for the case of a sequence in $\Vkp$ as the proof for the case of $\Vkm$ is very similar. The hypothesis $\sup_{k\in\N}\norm{v_k}_{\Vkp}<\infty$ implies that the sequence $\{v_k\}_{k\in\N}$ is uniformly bounded in $X$, and also that $\{\p_t \Ip v_k\}_{k\in \N} $ is uniformly bounded in $L^2(0,T;H^{-1}(\Omega))$ by Lemma~\ref{mass-lump-H10-dual-bound}. Furthermore, Lemma~\ref{lem:vkp_infty_bound} implies that $\{v_k(T)\}_{k\in\N}$ and $\{v_k(0)\}_{k\in\N}$ are uniformly bounded sequences in $L^2(\Omega)$. So, there exists a subsequence such that $v_{k_j}\rightharpoonup v$ in $X$, that $\p_t \Ip v_{k_j} \rightharpoonup \varphi $ in $L^2(0,T;H^{-1}(\Omega))$, that $v_{k_j}(T)\rightharpoonup v_T$ and $v_{k_j}(0)\rightharpoonup v_0$ as $j\tends \infty$, for some $v\in X$, some $\varphi \in L^2(0,T;H^{-1}(\Omega))$, and some $v_T,\,v_0 \in \LOm$.
To show that $\varphi=\p_t v$ so that $v\in Y$, and that $v(T)=v_T$ and $v(0)=v_0$, let $\phi \in C^1([0,T];H^1_0(\Omega))$ be arbitrary. Then, using integration-by-parts for the piecewise constant function $v_k\in \Vkp$ and the identities in~\eqref{derivatives_Ip}, we find that
\begin{multline}\label{eq:weak_convergence_1}
\int_0^T \langle \p_t \Ip v_{k_j},\phi \rangle + (v_{k_j},\p_t \phi)_{\Omega} \mathrm{d}t= (v_k(T),\phi(T))_{\Omega}-(v_k(0),\phi(0))_{\Omega}
\\+\sum_{n=1}^{N_{k_j}}\int_{I_n}(\p_t \Ip v_{k_j},\phi(t)-\phi(t_{n-1}))_{\Omega}\mathrm{d}t.
\end{multline}
It is straightforward to show that the last term on the right-hand side of~\eqref{eq:weak_convergence_1} vanishes in the limit $j\tends \infty$ since $\{\p_t \Ip v_k\}_{k\in \N} $ is uniformly bounded in $L^2(0,T;H^{-1}(\Omega))$ and since $\phi \in C^1([0,T];H^1_0(\Omega))$.
Therefore, passing to the limit in~\eqref{eq:weak_convergence_1}, we get
\begin{equation}\label{eq:weak_convergence_2}
\int_0^T \langle \varphi ,\phi \rangle + (v,\p_t \phi)_{\Omega} \mathrm{d}t = (v_T,\phi(T))_{\Omega}-(v_0,\phi(0))_{\Omega}  \quad \forall \phi \in C^1([0,T];H^1_0(\Omega)).
\end{equation}
It is then easy to deduce from~\eqref{eq:weak_convergence_2} that $v \in Y$ with  $\p_t v = \varphi \in L^2(0,T;H^{-1}(\Omega))$, and moreover that $v(T)=v_T$ and that $v(0)=v_0$.

To show~\eqref{eq:weak_convergence_mass_lumped}, let $v_\ell\in\mathbb{V}_\ell$, $\ell\in\N$ be fixed but arbitrary. Using~\eqref{eq:Rk_riesz}, we have
\begin{equation}
\int_0^T \left(\p_t \Ip v_{k_j},w_\ell\right)_{\Omkj}\mathrm{d}t = \int_0^T ( \p_t \Ip v_{k_j},w_\ell)_\Omega\mathrm{d}t + \int_0^T (\p_t \Ip v_{k_j},w_\ell-R_{k_j} w_\ell)_{\Omkj}\mathrm{d}t.
\end{equation}
It follows from the weak convergence properties above that $\int_0^T ( \p_t \Ip v_{k_j},w_\ell)_\Omega\mathrm{d}t \tends \int_0^T \langle \p_t v,w_\ell\rangle\mathrm{d}t$ as $j\tends \infty$.
Recall that the function $v_\ell$ is piecewise constant-in-time with respect to the partition $\mathcal{J}_\ell$ of $[0,T]$. In a slight abuse of notation, let $\{t_q\}_{q=0}^{M_\ell}$ denote the nodes of $\mathcal{J}_\ell$, i.e.\ $\mathcal{J}_\ell=\{[t_{q-1},t_q]\}_{q=0}^{M_\ell}$. Therefore, 
\begin{equation}\label{eq:weak_convergence_3}
\int_0^T (\p_t \Ip v_{k_j},w_\ell-R_{k_j} w_\ell)_{\Omkj}\mathrm{d}t = \sum_{q=1}^{M_\ell} (\Ip v_{k_j}(t_q)-\Ip v_{k_j}(t_{q-1}) , z_{j,\ell,q})_{\Omkj},
\end{equation}
where $z_{j,\ell,q} \coloneqq w_\ell|_{(t_{q-1},t_q)}-R_{k_j}\left(w_\ell|_{(t_{q-1},t_q)}\right)$. 
The $L^2$-approximation bound in~\eqref{quasi-interp-approx-property} and ~\eqref{eq:L^2-mass_lumped_bound} together imply that $\norm{z_{j,\ell,q}}_{\Omk}\lesssim h_{k_j} \norm{\nabla w_{\ell}|_{(t_{q-1},t_q)}}_\Omega$ for all $q=1,\dots, M_j$ and all $j\in\N$.
Also, Lemma~\ref{lem:vkp_infty_bound} and \eqref{eq:L^2-mass_lumped_bound} imply that $\norm{\Ip v_{k_j}(t_q)}_{\Omk}\lesssim \norm{v_{k_j}}_{\Vkp}$, so there exists a constant $C_*$ independent of $k\in\N$ such that $\norm{\Ip v_{k_j}(t_q)}_{\Omkj}\leq C_*$ for all $j\in\N$ and all $q\in \{0,\dots,N_\ell\}$.
Applying these bounds in~\eqref{eq:weak_convergence_3}, we see that
\begin{equation}\label{eq:weak_convergence_4}
\lim_{j\tends \infty} \left\lvert \int_0^T (\p_t \Ip v_{k_j},w_\ell-R_{k_j} w_\ell)_{\Omkj}\mathrm{d}t \right\rvert \lesssim \lim_{j\tends \infty}C_* h_{k_j} \sum_{q=1}^{M_\ell} \norm{\nabla w_\ell|_{(t_{q-1},t_q)}}_{\Omega}  =0.
\end{equation}
Therefore we obtain~\eqref{eq:weak_convergence_mass_lumped}.
\end{proof}

The following Lemma generalizes a well-known inequality for real-valued functions of bounded variation to the case of functions in $\Vkpm$.
\begin{lemma}\label{lem:translation_bound}
Let $\norm{\cdot}:V_k\tends \R_{\geq 0}$ be a norm on $V_k$. Then, for all $v\in \Vkpm$, we have
\begin{equation}\label{eq:translation_bound}
\sup_{s\in (0,T)}\int_0^{T-s} s^{-1} \norm{v(t+s) -v(t)}\mathrm{d}t \leq \int_0^T \norm{\p_t \Ipm v}\mathrm{d}t.
\end{equation}
\end{lemma}
\begin{proof}
We show the proof for the case of $v\in \Vkp$, as the argument for $\Vkm$ is quite similar.
Let $s\in(0,T)$ be arbitrary, and let $\su,\,\so\in \R$ be any positive real numbers such that $0<\su<\so<T$ and $\so-s>\su$. 
 For each $\epsilon>0 $ such that $\epsilon< \min(\su,T-\so)$, let $\eta_{\epsilon}=\frac{1}{\epsilon}\eta_0\left(\frac{\cdot}{\epsilon}\right)$ where $\eta_0$ is the standard mollifier on $\R$, so $\eta_{\epsilon}\in C^\infty_0(\R)$ is a smooth nonnegative function with $\supp \eta_{\epsilon}\subset [-\epsilon,\epsilon]$ and $\int_\R \eta_\epsilon(t)\mathrm{d}t =1$.
 For $v\in\Vkp$, define $v_{\epsilon}:[\su,\so]\tends V_k$ by $v_{\epsilon}(t)\coloneqq (v * \eta_{\epsilon})(t)=\int_{0}^T v(s)\eta_\epsilon(t-s)\mathrm{d}s$.
  Note that $v_{\epsilon}$ is well defined on $[\su,\so]$ and depends only on $v|_{(0,T)}$. Since $\eta(t-\cdot)\in C^\infty_0(0,T)$ for all $t\in[\su,\so]$, integration-by-parts shows that
  $\partial_t v_{\epsilon}(t) = \sum_{n=1}^{\Nk-1} \jump{v}_{n}\eta_{\epsilon}(t-t_n)$ for all $t\in[\su,\so]$.
  Therefore, the triangle inequality leads to
 \begin{equation}\label{eq:mollification_bound}
 \int_{\su}^{\so} \norm{\p_t v_{\epsilon}}\mathrm{d}t\leq \sum_{n=1}^{\Nk-1}\norm{\jump{v}_n}\int_{\su}^{\so} \eta_\epsilon(t-t_n)\mathrm{d}t\leq  \sum_{n=1}^{\Nk-1}\norm{\jump{v}_{n}}\leq \int_0^T\norm{\p_t\Ip v}\mathrm{d}t.
 \end{equation}
  Furthermore, $\int_{\su}^{\so}\norm{v-v_{\epsilon}}\mathrm{d}t\tends 0$ as $\epsilon\tends 0$. Therefore,
  \begin{multline}
\int_{\su}^{\so-s}\norm{v(t+s)-v(t)}\mathrm{d}t=\lim_{\epsilon\tends 0 }\int_{\su}^{\so-s}\norm{v_\epsilon(t+s)-v_\epsilon(t)}\mathrm{d}t
\\ \leq \limsup_{\epsilon\tends 0}\int_{\su}^{\so-s}\int_{t}^{t+s}{\norm{\p_t v_{\epsilon}({r})}\mathrm{d}{r}}\mathrm{d}t \leq\limsup_{\epsilon\tends 0} s\int_{\su}^{\so}\norm{\p_t v_{\epsilon}}\mathrm{d}{r}
\leq s\int_0^T\norm{\p_t\Ip v}\mathrm{d}{r},
  \end{multline}
  where  we invoke Fubini's Theorem and \eqref{eq:mollification_bound} in the second line above. We then readily conclude~\eqref{eq:translation_bound} by letting $\su\tends 0$ and $\so\tends T$, and recalling that $s\in(0,T)$ is arbitrary.
 \end{proof}

A direct consequence of Lemma~\ref{lem:translation_bound} and the definition of $\norm{\cdot}_{\Vkpm}$ in~\eqref{eq:discrete_vkp_combined} is that
 \begin{equation}\label{eq:translation_bound_2}
\sup_{s\in (0,T)}\int_0^{T-s} s^{-1} \norm{v(t+s) -v(t)}_{\dualk}\mathrm{d}t \leq \sqrt{T}\norm{v}_{\Vkpm} \quad \forall v\in \Vkpm.
 \end{equation}

\begin{theorem}[Compactness in $L^p(L^2)$]\label{precompact-result}
Let $\{v_k\}_{k\in\N}\subset X$ be a sequence such that $v_k\in \mathbb{V}_k^+$ for each $k\in\mathbb{N}$ and $\sup_{k\in\N}\norm{v_k}_{\mathbb{V}_k^+}<\infty$. Then, $\{v_k\}_{k\in\N}$ is pre-compact in $L^p(0,T;L^2(\Omega))$ for any $1\leq p <\infty$. {This result holds analogously for sequences $\{\tilde{v}_k\}_{k\in\mathbb{N}}$ such that $\tilde{v}_k\in \mathbb{V}_k^-$ for each $k\in\mathbb{N}$ and  $\sup_{k\in\N}\norm{\tilde{v}_k}_{\mathbb{V}_k^-}<\infty$.}
\end{theorem}
\begin{proof}
It is enough to prove the result in the case of $\Vkp$ by symmetry of $\Vkp$ and $\Vkm$ under reversal of the time variable. 
We show that $\{v_k\}_{k\in\N}$ satisfies the hypotheses of \cite[Theorem~3]{simon1986compact}, which asserts that $\{v_k\}_{k\in\N}$ is pre-compact in $L^p(0,T;L^2(\Omega))$ provided that $\{v_k\}_{k\in\N}$ is uniformly bounded in $L^1(0,T;B_1)$ for some Banach space $B_1$ that is compactly embedded in $L^2(\Omega)$, and that for all $\epsilon>0$, there is an $s_0>0$ such that for all $s<s_0$, $\norm{v_k(\cdot+s)-v_k(\cdot)}_{L^p(0,T-s;L^2(\Omega))}<\epsilon$ for all $k\in\N$. 
The first hypothesis holds immediately since $\{v_k\}_{k\in\N}$ is uniformly bounded in $X=L^2(0,T;H^1_0(\Omega))$ with $H^1_0(\Omega)$ compactly embedded in $L^2(\Omega)$. 
To show the second hypothesis, let $M_*\coloneqq \sup_{k\in\N} \norm{v_k}_{\Vkp}$, and let $r_{k,s}\coloneqq v_k(\cdot+s)-v_k$ for each $k\in\N$ and $s\in(0,T)$.
Then, we find that
\begin{multline}\label{eq:precompact_result_1}
\int_0^{T-s}\norm{r_{k,s}}_\Omega^p\mathrm{d}t\leq 2^{p-1}M_*^{p-1}\int_0^{T-s}\norm{r_{k,s}}_\Omega\mathrm{d}t \leq 2^{p-1}M_*^{p-1}\int_0^{T-s}\norm{r_{k,s}}_{\Omk}\mathrm{d}t 
\\ \leq  2^{p-1}M_*^{p-1}\left(\int_0^{T-s}\norm{\nabla r_{k,s}}_\Omega\mathrm{d}t \right)^\frac{1}{2}\left(\int_0^{T-s}\norm{r_{k,s}}_{\dualk}\mathrm{d}t\right)^\frac{1}{2}
 \leq 2^{p-\frac{1}{2}}\sqrt{T} M_*^p s^{\frac{1}{2}},
\end{multline}
where we have applied first Lemma~\ref{lem:vkp_infty_bound}, then the first inequality of~\eqref{eq:L^2-mass_lumped_bound}, followed by the Cauchy--Schwarz inequality and finally~\eqref{eq:translation_bound_2}.
Thus we obtain from \eqref{eq:precompact_result_1} that $\norm{v_k(\cdot+s)-v_k(\cdot)}_{L^p(0,T-s,L^2(\Omega))} \lesssim  s^{\frac{1}{2p}} M_*$  for all $k\in\N$, and all $s\in(0,T)$.
This verifies the second hypothesis above and thus completes the proof.
\end{proof}

}

\section{Proofs of Theorems~\ref{existence-space-time} and~\ref{conv-main-thm} and of Corollary~\ref{convergence-cor}}\label{sec8:proofs_of_convergence}

In this section, we assume that the hypotheses of Theorem~\ref{uniqueness-space-time} hold, i.e.\ that $m_0$ is nonnegative a.e.\ in $\Omega$, that $G$ is nonnegative in the sense of distributions in $L^2(0,T;H^{-1}(\Omega))$, that $F$ is strictly monotone, and that $S$ is monotone. We now prove convergence of the finite element approximations $\{({u}_{k},m_{k})\}_{k\in\mathbb{N}}$ that are determined by the scheme \eqref{weakform-space-time-discrete}. 

\subsection{KFP equation} 
For each $k\in\N$, let $\tilde{b}_k $ be an element of $\DpH[u_k]$ for which \eqref{weakform-space-time-discrete} holds; in cases where $\tilde{b}_k$ may be nonunique, we may choose one such element.

\begin{lemma}[Weak convergence of discrete KFP approximations]\label{mk-convergence}
There exists an $m\in Y$ and a $\tilde{b}_* \in L^\infty(\QT;\R^\dim)$ such that, after passing to a subsequence without change of notation, $ \p_t \Ip m_k\rightharpoonup \partial_t m$ in $L^2(0,T;H^{-1}(\Omega))$, $m_k\rightharpoonup m$ in $X$, $m_k(T)\rightharpoonup m(T)$ in $L^2(\Omega)$, $m_k\tends m$ in $L^p(0,T;L^2(\Omega))$ for all $p\in[1,\infty)$, and $\tilde{b}_k \rightharpoonup \tilde{b}_*$ in $L^2(\QT;\R^\dim)$ as $k\tends \infty$.
Furthermore, we have $m(0)=m_0$ in $L^2(\Omega)$ and
	\begin{equation}\label{mfg-pdi-sys-time-classical-intro-thm-discrete}
		\int_0^T\left\langle \partial_t{m},\phi\right\rangle_{}+\nu(\nabla {m},\nabla \phi)_{\Omega}+({m}\tilde{b}_*,\nabla \phi)_{\Omega}\mathrm{d}t=\int_0^T\langle {G},\phi\rangle_{}\mathrm{d}t
	\end{equation}
	for all $\phi\in X$.
\end{lemma}
\begin{proof}
{
Theorem~\ref{thm:weak_convergence} shows the existence of a $m\in Y$ and a subsequence of $\{m_k\}_{k\in\N}$, to which we pass without change of notation, that satisfies the weak convergence properties stated above. In particular, we have $m_k(0)=R_k m_0 \rightharpoonup m(0)$ as $k\tends \infty$, and since $R_k m_0 \tends m_0$ in $L^2(\Omega)$ as $k\tends \infty$, we conclude that $m(0)=m_0$.
Furthermore, the strong convergence $m_k\tends m$ in $L^p(0,T;L^2(\Omega))$ for all $p\in[1,\infty)$ as $k\tends \infty$ follows from Theorem~\ref{precompact-result}. 
Since $\{\tilde{b}_k\}_{k\in\N}$ is uniformly bounded in $L^\infty(\QT;\R^\dim)$ by Lemma~\ref{Prop1-time}, and hence also in $L^2(\QT;\R^\dim)$, we may pass to a further subsequence without change of notation to find that $\tilde{b}_k\rightharpoonup \tilde{b}_*$ for some $\tilde{b}_*\in L^2(\QT;\R^\dim)$. Mazur's Theorem and the Riesz--Fischer Theorem further imply that $\tilde{b}_* \in L^\infty(\QT;\R^\dim)$ with $\norm{\tilde{b}_*}_{L^\infty(\QT;\R^\dim)}\leq L_H$. 
It remains only to show~\eqref{mfg-pdi-sys-time-classical-intro-thm-discrete}. Let $\ell\in\N$ be fixed and let $\phi \in \mathbb{V}_\ell$ be fixed but arbitrary. Recall that $\mathbb{V}_\ell\subset \Vk$ for all $\ell\leq k$. Then, we have $m_k \nabla \phi \tends m \nabla \phi $ in $L^2(\QT;\R^\dim)$ as $k\tends \infty$ since $\nabla \phi \in L^\infty(\QT;\R^\dim)$. This implies that $\int_0^T(m_k,\tilde{b}_k{\cdot} \nabla \phi)_\Omega\mathrm{d}t\tends \int_0^T(m,\tilde{b}_*{\cdot} \nabla \phi)_\Omega \mathrm{d}t$ by weak-times-strong convergence. Next, we have $\int_0^T (\p_t \Ip m_k,\phi)_{\Omk} \mathrm{d}t \tends \int_0^T (\p_t m,\phi)_\Omega \mathrm{d}t$ by~\eqref{eq:weak_convergence_mass_lumped}. We also have $A_k\tends \nu \mathbb{I}$ as $k\tends \infty$ in $L^\infty(\Omega;\R^{\dim\times\dim})$ by Theorem~\ref{DMP-edge-stabilisation-result}.
Therefore we may pass to the limit in~\eqref{weakform2-space-time-discrete} to obtain~\eqref{mfg-pdi-sys-time-classical-intro-thm-discrete} for all $\phi \in \mathbb{V}_\ell$. Since $\bigcup_{\ell\in\N}\mathbb{V}_\ell$ is dense in $X$, we deduce that \eqref{mfg-pdi-sys-time-classical-intro-thm-discrete} holds for all $\phi \in X$.
}
\end{proof}
To establish the strong $L^2$-convergence of the density approximations at terminal time, we find it helpful to first establish some compactness properties associated to a class of related discrete dual problems.

{
\paragraph{Discrete dual problems} For each $k\in\N$, we define $z_k\in\Vkm$ as the unique solution of the discrete dual problem
\begin{equation}\label{eq:discrete_dual_problems}
\int_0^T - (\p_t \In z_k,v)_{\Omk}+(A_k \nabla z_k,\nabla v)_{\Omega}+(\tilde{b}_k{\cdot}\nabla z_k,v)_{\Omega}\mathrm{d}t =0 \quad \forall v \in \Vk,
\end{equation}
with the condition $z_k(T)=m_k(T)$ in $V_k$.

\begin{lemma}[Weak convergence of discrete dual solutions]\label{lem:discrete_dual_convergence}
After passing to a subsequence without change of notation, we have $z_k \rightharpoonup z$ in $X$, $\p_t \In z_k \rightharpoonup \p_t z$ in $L^2(0,T;H^{-1}(\Omega))$, $z_k\tends z \text{ in } L^2(0,T;L^2(\Omega))$ and $z_k(0)\tends z(0)$ in $L^2(\Omega)$ as $k\tends \infty$ where $z\in Y$ satisfies
\begin{equation}\label{eq:dual_problem_limit}
\int_0^T -\langle \p_t z, \phi\rangle +\nu(\nabla z,\nabla \phi)_{\Omega} + (\tilde{b}_*{\cdot} \nabla z,\phi)_{\Omega} \mathrm{d}t = 0 \quad \forall \phi \in X,
\end{equation}
with $z(T)=m(T)$ in  $L^2(\Omega)$, where $\tilde{b}_*\in L^\infty(\QT;\R^\dim)$ is as in Lemma~\ref{mk-convergence}.
\end{lemma}
\begin{proof}
The fact that $\norm{m_k(T)}_{\Omega} \lesssim \norm{m_k}_{\Vkp}$, see Lemma~\ref{lem:vkp_infty_bound}, and the fact that $\sup_{k\in\N}\norm{m_k}_{\Vkp}<\infty$ together imply that $\sup_{k\in\N} \norm{z_k}_{\Vkm}<\infty$. Theorem~\ref{thm:weak_convergence} implies that, after passing to a subsequence without change of notation, there exists a $z \in Y$ such that $z_k \rightharpoonup z$ in $X$, $\p_t \In z_k \rightharpoonup \p_t z$ in $L^2(0,T;H^{-1}(\Omega))$, $z_k(T)\rightharpoonup z(T)$ and $z_k(0)\rightharpoonup z(0)$ as $k\tends \infty$. 
 Moreover, Theorem~\ref{precompact-result} implies that $z_k\tends z$ in $L^2(0,T;L^2(\Omega))$.
 Note that $z(T)=m(T)$ since $z_k(T)=m_k(T) \rightharpoonup m(T)$ as $k\tends \infty$ by Lemma~\ref{mk-convergence}.
Since $\{\tilde{b}_k{\cdot}\nabla z_k\}_{k\in\N}$ is uniformly bounded in $L^2(0,T, L^2(\Omega))$, after possibly passing to a further subsequence without change of notation, there exists a $g\in L^2(0,T;L^2(\Omega))$ such that $\tilde{b}_k{\cdot}\nabla z_k \rightharpoonup g$ as $k\tends \infty$. 
After fixing a test function $v\in \mathbb{V}_j$ in~\eqref{eq:discrete_dual_problems} for some fixed $j\leq k$, we use the weak convergence properties of $z_k$ above, along with~\eqref{eq:weak_convergence_mass_lumped} and~ $\norm{A_k-\nu \mathbb{I}}_{L^\infty(\Omega;\R^{\dim\times\dim})}\lesssim h_k$, to pass to the limit $k\tends \infty$ in~\eqref{eq:discrete_dual_problems} to find that $z\in Y$ satisfies
\begin{equation}\label{eq:discrete_dual_convergence_1}
\int_0^T -\langle \p_t z, v\rangle+\nu(\nabla z,\nabla v)_{\Omega}\mathrm{d}t = -\int_0^T (g,v)_\Omega\mathrm{d}t \quad \forall v \in \mathbb{V}_j.
\end{equation}
Since $\bigcup_{j\in\N} \mathbb{V}_j$ is dense in $X$, we see that \eqref{eq:discrete_dual_convergence_1} further holds for all $v\in X$.

In order to show that $g=\tilde{b}_* {\cdot}\nabla z$ so that~\eqref{eq:dual_problem_limit} holds, it is enough to show that 
\begin{equation}\label{eq:discrete_dual_convergence_5}
\lim_{k\tends\infty}\int_0^{T}(T-t) \norm{\nabla(z-z_k)}_\Omega^2\mathrm{d}t =0.
 \end{equation}
Indeed, supposing momentarily that~\eqref{eq:discrete_dual_convergence_5} is known, we conclude first that $z_k\tends z$ in $L^2(0,T-\epsilon,H^1_0(\Omega))$ for all $\epsilon>0$. Weak-times-strong convergence then implies that $\tilde{b}_k {\cdot} \nabla z_k \rightharpoonup \tilde{b}_*{\cdot} \nabla z$ in $L^2(0,T-\epsilon,L^2(\Omega))$ for all $\epsilon\in(0,T)$, so that $g=\tilde{b}_*{\cdot}\nabla z$ on $(0,T-\epsilon)$ for arbitrarily small $\epsilon$, and thus $g=\tilde{b}_*{\cdot}\nabla z$ on $(0,T)$.

In order to prove~\eqref{eq:discrete_dual_convergence_5}, define
\begin{equation}
\begin{aligned}
w_k|_{I_j} \coloneqq (T-t_j) z_k|_{I_j}, && \widetilde{w}_k|_{I_j}\coloneqq (T-t_j)^{\frac{1}{2}} \,z_k|_{I_j} && \forall j\in\{1,\dots,\Nk\}.
\end{aligned}
\end{equation}
Since $z_k$ converges to $z$ strongly in $L^2(0,T;L^2(\Omega))$ and weakly in $X=L^2(0,T,H^1_0(\Omega))$, it is easy to show that $w_k \tends w $ strongly in $L^2(0,T;L^2(\Omega))$ where $w(t)\coloneqq(T-t)z \in X$, and also that $\widetilde{w}_k \rightharpoonup \widetilde{w}$ in $X$ where $\widetilde{w}\coloneqq (T-t)^{\frac{1}{2}}\, w$. 
Note that $\widetilde{w}_k \rightharpoonup \widetilde{w}$ in $X$ and $z_k(0)\rightharpoonup z(0) $ in $L^2(\Omega)$ as $k\tends \infty$ respectively imply that
\begin{equation}\label{eq:discrete_dual_convergence_7}
\begin{aligned}
\norm{\widetilde{w}}_X^2 \leq \liminf_{k\tends \infty} \norm{\widetilde{w}_k}_X^2, &&& \norm{z(0)}^2_\Omega \leq \liminf_{k\tends \infty}\norm{z_k(0)}_\Omega^2.
\end{aligned}
\end{equation}
We also deduce from strong convergence of  $z_k\tends z$ in $L^2(0,T;L^2(\Omega))$ that
\begin{equation}\label{eq:discrete_dual_convergence_2}
\lim_{k\tends \infty} \int_0^T \norm{z_k}_{\Omk}^2\mathrm{d}t = \int_0^T \norm{z}_\Omega^2\mathrm{d}t.
\end{equation}
Indeed~\eqref{eq:discrete_dual_convergence_2} follows from $\norm{z_k}_{\Omk}^2 = \norm{z_k}_{\Omega}^2 + (z_k,z_k-R_kz_k)_{\Omega,k} $ a.e.\ on $(0,T)$ by \eqref{eq:Rk_riesz}, and from $\lvert (z_k,z_k-R_kz_k)_{\Omega,k}\rvert\lesssim h_k \norm{\nabla z_k}_{\Omega}^2$ a.e.\ on $(0,T)$ by~\eqref{eq:L^2-mass_lumped_bound}, \eqref{quasi-interp-approx-property} and the Poincar\'e inequality.
Then, a straightforward calculation gives
\begin{equation}\label{eq:discrete_dual_convergence_3}
 \frac{T}{2} \norm{z_k(0)}_{\Omk}^2 + \sum_{j=1}^{\Nk}\frac{T-t_j}{2}\norm{\jump{z_k}_j}_{\Omk}^2 = \int_0^T - (\p_t \In z_k,w_k)_{\Omk} +\frac{1}{2} \norm{z_k}_{\Omk}^2 \mathrm{d} t.
\end{equation}
Thus, testing~\eqref{eq:discrete_dual_problems} with $w_k\in \Vk$ and re-arranging terms shows that
\begin{equation}
\label{eq:discrete_dual_convergence_4}
\begin{split}
\frac{T}{2}\norm{z_k(0)}_{\Omega}^2+ \nu \norm{\widetilde{w}_k}_{X}^2 
 &\leq \frac{T}{2}\norm{z_k(0)}_{\Omk}^2+\int_0^T (A_k\nabla z_k,\nabla w_k)_{\Omega}\mathrm{d}t
 \\ &\leq \int_0^T - (\p_t \In z_k,w_k)_{\Omk}+(A_k \nabla z_k,\nabla w_k)_{\Omega}  + \frac{1}{2}\norm{z_k}_{\Omk}^2\mathrm{d}t
\\ &=\int_0^T  -(\tilde{b}_k{\cdot}\nabla z_k, w_k)_{\Omega}+\frac{1}{2}\norm{z_k}_{\Omk}^2\mathrm{d}t.
\end{split}
\end{equation}
Note that in the first inequality above we have used~\eqref{eq:L^2-mass_lumped_bound} as well as the inequality $\nu\norm{\nabla \widetilde{w}_k}_{\Omega}^2 \leq (A_k\nabla z_k,\nabla w_k)_{\Omega}$ since $A_k\geq \nu \mathbb{I}$ in the partial ordering of positive semi-definite matrices. 
We then combine the strong convergence $w_k \tends w = (T-t)z$ in $L^2(0,T;L^2(\Omega))$ with weak convergence $\tilde{b}_k{\cdot}\nabla z_k \rightharpoonup g$ in $L^2(0,T;L^2(\Omega))$, and~\eqref{eq:discrete_dual_convergence_2} to pass to the limit in the last line of~\eqref{eq:discrete_dual_convergence_4} to find that
\begin{equation}
\begin{aligned}
\lim_{k\tends \infty} \int_0^T  -(\tilde{b}_k{\cdot}\nabla z_k, w_k)_{\Omega}+\frac{1}{2}\norm{z_k}_{\Omk}^2\mathrm{d}t & = \int_0^T - (g,(T-t)z)_{\Omega} + \frac{1}{2}\norm{z}_{\Omega}^2 \mathrm{d}t
\\ & = \frac{T}{2}\norm{z(0)}_{\Omega}^2+\int_0^T (T-t)\nu\norm{\nabla z}_{\Omega}^2 \mathrm{d}t
\\ & = \frac{T}{2}\norm{z(0)}_{\Omega}^2+\nu\norm{\widetilde{w}}_X^2,
\end{aligned}
\end{equation}
where the second equality is obtained by using the test function $w=(T-t)z$ in~\eqref{eq:discrete_dual_convergence_1}.
This shows that
\begin{equation}\label{eq:discrete_dual_convergence_8}
\limsup_{k\tends \infty} \left[\frac{T}{2}\norm{z_k(0)}_{\Omega}^2+ \nu \norm{\widetilde{w}_k}_X^2  \right] \leq \frac{T}{2}\norm{z(0)}_{\Omega}^2+ \nu   \norm{\widetilde{w}}_X^2.
\end{equation}
We then conclude from \eqref{eq:discrete_dual_convergence_7} and \eqref{eq:discrete_dual_convergence_8} that $\norm{z_k(0)}_\Omega\tends \norm{z(0)}_\Omega$ and $\norm{\widetilde{w}}_X\tends \norm{\widetilde{w}}_X$ as $k\tends \infty$, which combined with the weak convergence properties above imply that $\norm{z(0)-z_k(0)}_\Omega\tends 0$ and $\norm{\widetilde{w}-\widetilde{w}_k}_X\tends 0$ as $k\tends \infty$.
The strong convergence $\widetilde{w}_k\tends \widetilde{w}$ in $X$ and the triangle inequality imply that
\[
\begin{aligned}
\lim_{k\tends \infty}\int_0^T(T-t)\norm{\nabla(z-z_k)}_\Omega^2\mathrm{d}t & \lesssim \lim_{k\tends \infty}\left[\norm{\widetilde{w}-\widetilde{w}_k}^2_X + \tau_k \norm{z_k}_{X}^2\right] =0,
\end{aligned}
\]
where we have used the fact that $\sup_{k\in\N}\norm{z_k}_X\lesssim \sup_{k\in\N}\norm{z_k}_{\Vkm}<\infty$. This completes the proof of~\eqref{eq:discrete_dual_convergence_5} and thus of the lemma.
\end{proof}
}

We can now prove the strong $L^2$-convergence of the density function approximations at terminal time.
\begin{theorem}[Terminal-time $L^2$-convergence of discrete KFP approximations]\label{main-thm-2}
After possibly passing to a subsequence without change of notation, we have $m_k(T)\tends m(T)$ in $L^2(\Omega)$ as $k\tends \infty$, where $m$~solves \eqref{mfg-pdi-sys-time-classical-intro-thm-discrete}.
\end{theorem}
\begin{proof}

Since $m_k(T)\rightharpoonup m(T)$ in $L^2(\Omega)$ by Lemma~\ref{mk-convergence}, it is enough to show that $\limsup_{k\tends\infty}\norm{m_k(T)}_{\Omega}\leq \norm{m(T)}_{\Omega}$.
The discrete integration-by-parts identity~\eqref{eq:discrete_ibp} implies that
\begin{multline*}
\norm{m_k(T)}_{\Omk}^2 = (m_k(T),z_k(T))_{\Omk}\\ = (m_{k}(0),z_k(0))_{\Omk}+\int_0^T(\p_t \Ip m_k,z_k)_{\Omk}+(m_k,\p_t \In z_k)_{\Omk}\mathrm{d}t.
\end{multline*}
Then, since $z_k \in \Vkm$ and $m_k\in \Vkp$ naturally embed into $\Vk$, we use $z_k$ as a test function in~\eqref{weakform2-space-time-discrete} with $m_k$ as a test function in~\eqref{eq:discrete_dual_problems} to obtain
\begin{equation}\label{eq:terminal_time_1}
\norm{m_k(T)}_{\Omk}^2 = (m_k(0),z_k(0))_{\Omk} + \int_0^T \langle G, z_k \rangle \mathrm{d}t.
\end{equation}
Similarly, since $Y$ embeds continuously in $X$, we use $z$ as a test function in~\eqref{mfg-pdi-sys-time-classical-intro-thm-discrete} and $m$ as test function in \eqref{eq:dual_problem_limit} to obtain
\begin{equation}\label{eq:terminal_time_2}
\norm{m(T)}_\Omega^2 = (m_0,z(0))_\Omega + \int_0^T\langle G, z \rangle \mathrm{d}t.
\end{equation} 
The convergence $z_k(0)\tends z(0)$ in $L^2(\Omega)$ and $z_k\rightharpoonup z$ in $X=L^2(0,T;H^1_0(\Omega))$ from Lemma~\ref{lem:discrete_dual_convergence}, together with the fact that $m_k(0) = R_km_0$ in $V_k$, then implies that 
\begin{equation}
 (m_k(0),z_k(0))_{\Omk}=(m_{0},z_k(0))_{\Omega} \tends (m_0,z(0))_{\Omega}, \quad \int_0^T \langle G, z_k \rangle \mathrm{d}t\tends \int_0^T\langle G, z \rangle \mathrm{d}t,
\end{equation}
as $k\tends\infty$.
Therefore, passing to the limit in \eqref{eq:terminal_time_1}, we  see that
\begin{equation}
\limsup_{k\tends \infty}\norm{m_k(T)}_{\Omega}^2\leq \limsup_{k\tends \infty}\norm{m_k(T)}_{\Omk}^2 \leq \norm{m(T)}_{\Omega}^2.
\end{equation}
where we recall that the first inequality above follows from~\eqref{eq:L^2-mass_lumped_bound}. This implies that $m_k(T)$ converges in norm to $m(T)$ in $L^2(\Omega)$.
\end{proof}
	
\subsection{HJB equation}
With the terminal time compactness result Theorem~\ref{main-thm-2} established for the density approximations, we can now prove convergence of the value function approximations to a weak solution to the HJB equation.
\begin{lemma}[Convergence to the HJB Equation]\label{uk-convergence}
	After passing to a subsequence without change of notation, we have $u_{k}\to u$ in $X$,  $u_{k}\to u$ in $L^p(0,T;\LOm)$ for any $1\leq p<\infty$, and $u_{k}(0)\to u(0) $ in $\LOm$ as $k\to\infty$ where $u\in Y$ solves $u(T)=S[m(T)]$ in $L^2(\Omega)$ and
	\begin{equation}\label{HJB-final}
				\int_0^T - \langle\partial_t u,\psi \rangle+\nu(\nabla u,\nabla  \psi)_{\Omega}+( H[\nabla u],\psi)_{\Omega} \mathrm{d}t
			= \int_0^T\langle {F}[m],\psi\rangle \mathrm{d}t	\quad\forall \psi \in X,
		\end{equation}
		 where $m\in Y$ satisfies~\eqref{mfg-pdi-sys-time-classical-intro-thm-discrete}. 
\end{lemma}
{Note that an integration-by-parts argument on the time derivative in~\eqref{HJB-final} shows that $u$ equivalently solves~\eqref{weakform1-space-time} for all $\psi\in Y_0$.}

\begin{proof}
{
Since $\sup_{k\in\N}\norm{u_k}_{\Vkm}<\infty$ by~\eqref{eq:discrete_mfg_existence}, Theorems~\ref{thm:weak_convergence} and~\ref{precompact-result} imply that, after passing to a subsequence without change of notation, there exists a $u\in Y$ such that $\p_t \In u_k \rightharpoonup \p_t u$ in $L^2(0,T;H^{-1}(\Omega))$, $u_k\rightharpoonup u$ in $X$, $u_k(0)\rightharpoonup u(0)$ in $L^2(\Omega)$, $u_k(T)\rightharpoonup u(T)$ in $L^2(\Omega)$ and $u_k\tends u$ in $L^p(0,T;L^2(\Omega))$ for all $p\in[1,\infty)$. Furthermore, since $m_k(T)\tends m(T)$ by Theorem~\ref{main-thm-2}, the continuity of $S$ on $L^2(\Omega)$ and the $L^2$-stability of $R_k$ imply that $u_k(T)=R_k S[m_k(T)]\tends S[m(T)]$ in norm in $L^2(\Omega)$, hence $u(T)=S[m(T)]$. 
It follows from~\eqref{eq:Lipschitz_H} that $\{H[\nabla u_k]\}_{k\in\N}$ is uniformly bounded in $L^2(0,T;L^2(\Omega))$ and thus, after passing to a further subsequence without change of notation, there exists a $H_* \in L^2(0,T;L^2(\Omega))$ such that $H[\nabla u_k]\rightharpoonup H_*$ in $L^2(0,T;L^2(\Omega))$.
Furthermore, the continuity of $F$ on $L^2(0,T;L^2(\Omega))$ and the strong convergence $m_k\tends m$ in $L^2(0,T;L^2(\Omega))$ from Lemma~\ref{mk-convergence} imply that $F[m_k]\tends F[m]$ in $L^2(0,T;H^{-1}(\Omega))$. We may then use density of $\bigcup_{\ell\in\N} \mathbb{V}_\ell$ in $X$ and \eqref{eq:weak_convergence_mass_lumped} to pass to the limit as $k\tends \infty$ in~\eqref{eq:strong_in_time_discrete_HJB} to find} that $u$ solves
\begin{equation}\label{almost-HJB-equation}
\begin{aligned}
\int_0^{T}-\left\langle \partial_tu,{\psi}\right\rangle + \nu(\nabla {u},\nabla {\psi})_{\Omega} + (H_*, {\psi})_{\Omega}\mathrm{d}t=\int_0^{T}\langle F[{m}],{\psi}\rangle\mathrm{d}t  &&& \psi\in X.
\end{aligned}
\end{equation}
{We now claim that $u_k\tends u$ in norm in $X$ as $k\tends \infty$ which will imply $H_*=H[\nabla u]$. First, note that
\begin{equation}
\begin{split}
\frac{1}{2}\norm{u_k(0)}_\Omega^2 &\leq \frac{1}{2}\norm{u_k(0)}_{\Omk}^2+\frac{1}{2}\sum_{n=1}^{N_k}\norm{\jump{u_k}_n}_{\Omk}^2
 \\ =  &\int_0^T- (\p_t \In u_k,u_k)_{\Omk} \mathrm{d}t + \frac{1}{2}\norm{u_k(T)}_{\Omk}^2.
\end{split}
\end{equation}
where the first inequality above follows from~\eqref{eq:L^2-mass_lumped_bound}.
Observe also that that $\nu\norm{u_k}_X^2=\int_0^T\nu\norm{\nabla u_k}_{\Omega}^2\mathrm{d}t \leq \int_0^T(A_k\nabla u_k,\nabla u_k)_{\Omega}\mathrm{d}t$ since $A_k\geq \nu \mathbb{I}$ a.e.\ in $\Omega$ in the sense of semi-definite matrices.
Next, note that $\norm{u_k(T)}_{\Omk}^2=\norm{R_kS[m_k(T)]}_{\Omk}^2 \tends \norm{S[m(T)]}_\Omega^2$ as $k\tends \infty$ since $\norm{R_kS[m_k(T)]}_{\Omk}^2=(S[m_k(T)],R_k S[m_k(T)])_\Omega$ and both $R_kS[m_k(T)]$ and $S[m_k(T)]$ converge to $S[m(T)]$ in $L^2(\Omega)$ as $k\tends \infty$.
Furthermore, weak convergence $H[\nabla u_k]\rightharpoonup H_*$ and strong convergence $u_k\tends u$ in $L^2(0,T;L^2(\Omega))$ imply that $\int_0^T (H[\nabla u_k],u_k)_\Omega \mathrm{d}t \tends \int_0^T (H_*,u)_\Omega \mathrm{d}t$. Since $F[m_k]\tends F[m]$ in $L^2(0,T;H^{-1}(\Omega))$ we also have $\int_0^T\langle F[m_k],u_k\rangle\mathrm{d}t \tends \int_0^T \langle F[m],u\rangle\mathrm{d}t$ as $k\tends \infty$.
Therefore, testing~\eqref{eq:strong_in_time_discrete_HJB} with $u_k$ gives
\begin{equation}\label{eq:uk_convergence_1}
\begin{split}
\limsup_{k\tends \infty} &\left[\frac{1}{2}\norm{u_k(0)}_\Omega^2+\nu\norm{u_k}_X^2 \right]
\\ & \leq \lim_{k\tends\infty}\left[  \int_0^T- (\p_t \In u_k,u_k)_{\Omk} +(A_k\nabla u_k,\nabla u_k)_{\Omega} \mathrm{d}t + \frac{1}{2}\norm{u_k(T)}_{\Omk}^2 \right]
\\ &= \lim_{k\tends \infty}\left[ \int_0^T \langle F[m_k],u_k\rangle - (H[\nabla u_k],u_k)_{\Omega}\mathrm{d}t + \frac{1}{2}\norm{R_kS[m_k(T)]}_{\Omk}^2\right]
\\ &= \int_0^T \langle F[m],u\rangle -(H_*,u)_{\Omega}\mathrm{d}t+\frac{1}{2}\norm{S[m(T)]}_\Omega^2
\\ &= \frac{1}{2}\norm{u(0)}_{\Omega}^2+\nu\norm{u}_X^2,
\end{split}
\end{equation}
where the last line follows from testing~\eqref{almost-HJB-equation} with $u$.
In conjunction with the inequalities $\norm{u(0)}_\Omega\leq \liminf_{k\tends \infty}\norm{u_k(0)}_\Omega$ and $\norm{u}_X\leq\liminf_{k\tends\infty}\norm{u_k}_X$ owing to weak convergence, we see that \eqref{eq:uk_convergence_1} implies that $u_k(0)\tends u(0)$ in norm in $L^2(\Omega)$ and that $u_k\tends u$ in norm in $X$ as $k\tends \infty$. In turn, this implies that $H_*=H[\nabla u]$ since $H[\nabla u_k]\tends H[\nabla u]$ in $L^2(0,T;L^2(\Omega))$ by Lipschitz continuity of $H$, and thus $u$ solves~\eqref{HJB-final}.
}
\end{proof}

\subsection{Proof of Theorems~\ref{conv-main-thm} and~\ref{existence-space-time}}
	To begin, we observe that since $\tilde{b}_{k}\in {\Dp}H[{u}_{k}]$ for all $k\in\mathbb{N}$, $\{\tilde{b}_{k}\}_{k\in\mathbb{N}}$ is uniformly bounded in $L^{\infty}(Q_T;\R^\dim)$ (see Lemma \ref{Prop1-time}).  We may then pass to a subsequence, without change of notation such that $\{\tilde{b}_{k}\}_{k\in\mathbb{N}}$ converges weakly in $L^2(Q_T;\R^\dim)$ to some $\tilde{b}_*\in L^{\infty}(Q_T;\R^\dim)$.
	{}
	We first obtain from Lemma~\ref{mk-convergence} {and Theorem~\ref{main-thm-2} } that $m_{k}\to m$ in $L^p(0,T;\LOm)$ for any $1\leq p<\infty$, and  $m_{k}(T)\to m(T) $ in $\LOm$ as $k\to\infty$ where $m\in Y$ uniquely satisfies the KFP equation
	\eqref{mfg-pdi-sys-time-classical-intro-thm-discrete} and the initial condition $m(0)=m_0$.
	Consequently, Lemma \ref{uk-convergence} implies that $u_{k}\to u$ in $X$,  $u_{k}\to u$ in $L^p(0,T;\LOm)$ for any $1\leq p<\infty$, and $u_{k}(0)\to u(0) $ in $\LOm$ as $k\to\infty$ where $u\in Y$ uniquely satisfies the HJB equation \eqref{HJB-final}. 
	{S}ince $\tilde{b}_{k}\in {\Dp}H[{u}_{k}]$ for all $k\in\mathbb{N}$, while $\{\tilde{b}_{k}\}_{k\in\mathbb{N}}$ converges weakly in $L^2(Q_T;\R^\dim)$ to  $\tilde{b}_*$ and ${u}_{k}\to u$ in $X$, we  can apply Lemma \ref{closure} to obtain $\tilde{b}_*\in {\Dp}H[u]$. This shows that $(u,m)\in Y$ is a weak solution of~\eqref{weakdef-space-time}, and thus proves Theorem~\ref{existence-space-time}.
	
	In {summary}, we have shown that {a subsequence of the numerical} approximations converges to a solution $(u,m)$ of \eqref{weakdef-space-time} in the sense of~\eqref{eq:eulerFEMconv}. 
	{Then,} uniqueness of the solution, {c.f.\ } Theorem \ref{uniqueness-space-time}, implies that the {whole sequence $\{(u_k,m_k)\}\}_{k\in\N}$} converges to $(u,m)$ in the sense of~\eqref{eq:eulerFEMconv}. This completes the proof of Theorem~\ref{conv-main-thm}.
	 \hfill\proofbox

\subsection{Proof of Corollary \ref{convergence-cor}}
Under the hypotheses of Corollary \ref{convergence-cor}, we now {prove} strong convergence in $L^2(0,T;H_0^1(\Omega))$ of the density approximations.

	We know by hypothesis that $\{\tilde{b}_{k}\}_{k\in\mathbb{N}}$ is pre-compact in $L^1(Q_T;\R^\dim)$. Moreover, as a consequence of Lemma \ref{Prop1-time} and the definition of the scheme \eqref{weakform-space-time-discrete}, the sequence is uniformly bounded in $L^{\infty}(Q_T;\R^\dim)$. Since $\text{meas}(\Omega)<\infty$ and $T<\infty$, it can then be shown that $\{\tilde{b}_{k}\}_{k\in\mathbb{N}}$ is pre-compact in $L^q(Q_T;\R^\dim)$ for any {$q \in [1,\infty)$}. Since the density $m\in Y$, we have by the embedding $Y\to C([0,T];\LOm)$ that $m\in X\cap C([0,T];\LOm)$. It then follows from the parabolic embedding \cite[Equation (6.39), p. 466]{ladyvzenskaja1968linear} that $m\in L^{2+\frac{4}{d}}(Q_T)$. Passing to a suitable subsequence, without change of notation, we get that $\tilde{b}_{k}\to \tilde{b}_*$ in $L^{d+2}(Q_T;\R^\dim)$ as $k\to\infty$ for some $\tilde{b}_*\in L^{d+2}(Q_T;\R^\dim)$. It then follows that $m_{k}\tilde{b}_{k}\to m\tilde{b}_*$ in $L^2(Q_T;\R^\dim)$ as $k\to\infty$.
	
	We now show that $m_{k}\to m$ in $X$ as $k\to\infty$, along a subsequence. For given $k\in\mathbb{N}$, testing the equation~\eqref{weakform2-space-time-discrete} with $\phi=m_{k}\in\timeVkforward$ leads to
	\begin{multline}\label{lemma-1-zk-energy*}
		\limsup_{k\tends\infty}\left[ {\frac{1}{2}\norm{m_{k}(T)}_{\Omega}^2} +{\nu\norm{m_k}_X^2}\right] 
		\\ \leq \lim_{k\tends\infty}\left[ \int_0^T\langle G,m_{k}\rangle - (m_{k}\tilde{b}_{k},\nabla m_{k})_{\Omega}\mathrm{d}t
		+\frac{1}{2}\|{R_km_{0}}\|_{\Omk}^2\right]
		\\ = \int_0^T\langle G,m \rangle - (m\tilde{b}_{*},\nabla m)_{\Omega}\mathrm{d}t + \frac{1}{2}\norm{m_0}_\Omega^2
		 \\ {= \frac{1}{2}\norm{m(T)}_{L^2(\Omega)}^2+\nu\norm{m}_X^2,}
	\end{multline} 
		{where the last line above is obtained by testing~\eqref{weakform2-space-time} with $m$.}
	Since $m_{k}(T)\to m(T)$ in $\LOm$ as $k\to\infty$,	it follows from \eqref{lemma-1-zk-energy*} that
	$\limsup_{k\to\infty}{\norm{m_k}_X}\leq {\norm{m}_X}$.
	Together with the fact that $m_{k}\rightharpoonup m$ in $X$ as $k\to\infty$, we obtain that $ m_{k}\to m$ in $X$ as $k\to\infty$. This completes the proof. \hfill\proofbox

\begin{figure}[tbhp]
	\centering
	\begin{tabular}{c c} 
		\begin{subfigure}[b]{0.45\textwidth}
			\begin{adjustbox}{width=0.95\linewidth}
				\begin{tikzpicture}
					\begin{loglogaxis}[
						title={Value function $L^2(H_0^1)$-norm},
						xlabel={mesh size $h_k$},
						ylabel={relative error},
						xmax=1,
						ymax=1,
						legend pos=north west,
						ymajorgrids=true,
						grid style=dashed,
						]
						
						\addplot[
						color=blue,
						mark=square,]
						coordinates {
							(0.7071068,0.7006612)
							(0.3535534,0.4321884)
							(0.1767767,0.2347566)
							(0.08838835,0.1222251)
							(0.04419417,0.06239205)
							(0.02209709,0.03153363)
							(0.01104854,0.01585455)
							(0.005524272,0.007949726)

						};
						
						\logLogSlopeTriangle{0.4}{0.25}{0.08}{1}{black};
					\end{loglogaxis}  
				\end{tikzpicture}
			\end{adjustbox}
		\end{subfigure}
		&
		\begin{subfigure}[b]{0.45\textwidth}
			\begin{adjustbox}{width=0.95\linewidth} 
				\begin{tikzpicture}
					\begin{loglogaxis}[
						title={Transport vector $L^2(L^2)$-norm},
						xlabel={mesh size $h_k$},
						ylabel={relative error},
						xmax=1,
						ymax= 1,
						legend pos=north west,
						ymajorgrids=true,
						grid style=dashed,
						]
						
						\addplot[
						color=blue,
						mark=square,]
						coordinates {
							(0.7071068,0.6149608)
							(0.3535534,0.4093697)
							(0.1767767,0.2559373)
							(0.08838835,0.1331117)
							(0.04419417,0.07477188)
							(0.02209709,0.04198062)
							(0.01104854,0.02243097)
							(0.005524272,0.01190763)
							
						};

						\logLogSlopeTriangle{0.4}{0.25}{0.08}{1}{black};
					\end{loglogaxis}  
				\end{tikzpicture}
			\end{adjustbox}
		\end{subfigure}
		\\
		\begin{subfigure}[b]{0.45\textwidth}
			\begin{adjustbox}{width=0.95\linewidth} 
				\begin{tikzpicture}
					\begin{loglogaxis}[
						title={Density function $L^2(L^2)$-norm},
						xlabel={mesh size $h_k$},
						ylabel={relative error},
						xmax=1,
						ymax=1,
						legend pos=north west,
						ymajorgrids=true,
						grid style=dashed,
						]
						
						\addplot[
						color=blue,
						mark=square,]
						coordinates {
							(0.7071068,0.7018799)
							(0.3535534,0.378799)
							(0.1767767,0.1899179)
							(0.08838835,0.09433648)
							(0.04419417,0.04696553)
							(0.02209709,0.0234262)
							(0.01104854,0.01169438)
							(0.005524272,0.005845161)
							
						};

						\logLogSlopeTriangle{0.4}{0.25}{0.08}{1}{black};
					\end{loglogaxis}  
				\end{tikzpicture}
			\end{adjustbox}
		\end{subfigure} 
		&
		\begin{subfigure}[b]{0.45\textwidth}
			\begin{adjustbox}{width=0.95\linewidth} 
				\begin{tikzpicture}
					\begin{loglogaxis}[
						title={Density function $L^2(H_0^1)$-norm},
						xlabel={mesh size $h_k$},
						ylabel={relative error},
						xmax=1,
						ymax=1,
						legend pos=north west,
						ymajorgrids=true,
						grid style=dashed,
						]
						
						\addplot[
						color=blue,
						mark=square,]
						coordinates {
							(0.7071068,0.7756508)
							(0.3535534,0.4737106)
							(0.1767767,0.2608107)
							(0.08838835,0.1365934)
							(0.04419417,0.06991479)
							(0.02209709,0.03544724)
							(0.01104854,0.01782139)
							(0.005524272,0.008935388)
							
						};
						
						\logLogSlopeTriangle{0.4}{0.25}{0.08}{1}{black};
					\end{loglogaxis}  
				\end{tikzpicture}
			\end{adjustbox}
		\end{subfigure} 
		\\
		\begin{subfigure}[b]{0.45\textwidth}
			\begin{adjustbox}{width=0.95\linewidth} 
				\begin{tikzpicture}
					\begin{loglogaxis}[
						title={Value function initial time $L^2$-norm},
						xlabel={mesh size $h_k$},
						ylabel={relative error},
						xmax=1,
						ymax=1,
						legend pos=north west,
						ymajorgrids=true,
						grid style=dashed,
						]
						
						\addplot[
						color=blue,
						mark=square,]
						coordinates {
							(0.7071068,0.6241043)
							(0.3535534,0.2682495)
							(0.1767767,0.1108376)
							(0.08838835,0.0480436)
							(0.04419417,0.02192813)
							(0.02209709,0.01041001)
							(0.01104854,0.005062558)
							(0.005524272,0.00249511)
							
						};
						
						\logLogSlopeTriangle{0.4}{0.25}{0.08}{1}{black};
					\end{loglogaxis}  
				\end{tikzpicture}
			\end{adjustbox}
		\end{subfigure} 
		&
		\begin{subfigure}[b]{0.45\textwidth}
			\begin{adjustbox}{width=0.95\linewidth} 
				\begin{tikzpicture}
					\begin{loglogaxis}[
						title={Density function terminal time $L^2$-norm},
						xlabel={mesh size $h_k$},
						ylabel={relative error},
						xmax=1,
						ymax=1,
						legend pos=north west,
						ymajorgrids=true,
						grid style=dashed,
						]
						
						\addplot[
						color=blue,
						mark=square,]
						coordinates {
							(0.7071068,0.7016681)
							(0.3535534,0.3803955)
							(0.1767767,0.1911599)
							(0.08838835,0.0950924)
							(0.04419417,0.04734029)
							(0.02209709,0.02360347)
							(0.01104854,0.01177831)
							(0.005524272,0.005885448)
							
						};

						\logLogSlopeTriangle{0.4}{0.25}{0.08}{1}{black};
					\end{loglogaxis}  
				\end{tikzpicture}
			\end{adjustbox}
		\end{subfigure} 
	\end{tabular} 
	\caption{Convergence plots for approximations of the value function, density function, and transport vector. Optimal rates of convergence are observed for the value function and density function errors in the $H^1$-norm.}
	\label{exp:rel-err-plots}
\end{figure}
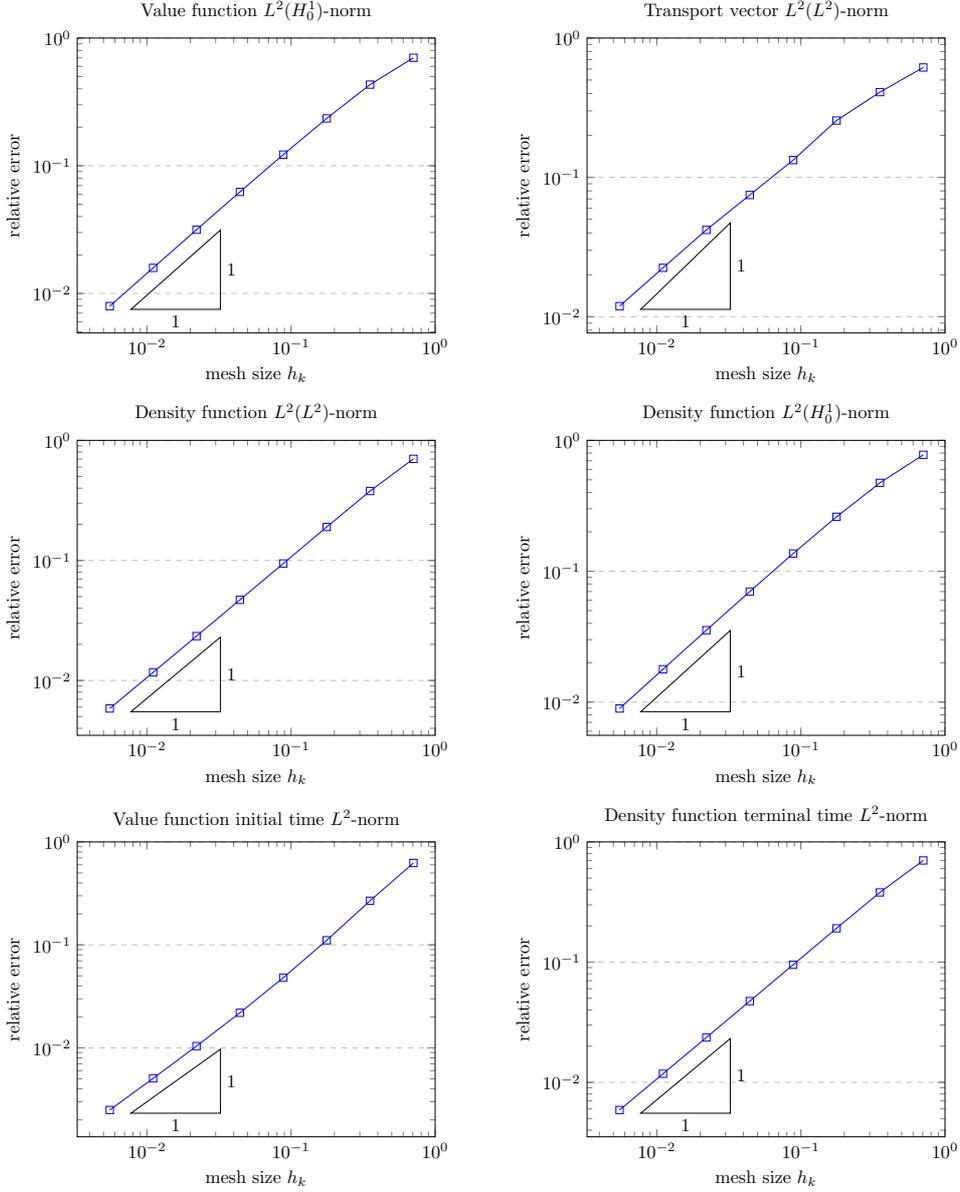

\section{Numerical Experiment}\label{sec9}
In this section we test the performance of the finite element scheme \eqref{weakform-space-time-discrete} in approximating a smooth solution to a MFG PDI of the form \eqref{eq:weakform-space-time}. 
Let $\Omega = (0,1)^2 \subset \mathbb{R}^2$ be the unit square, and let the time horizon $T=1$.
We set $\nu =1$, and we set
\begin{equation}\label{H-test}
	H(t,x,p)\coloneqq\max_{\alpha\in \overline{B_1(0)}}\left(\alpha\cdot p\right)=|p|\quad\forall (t,x,p)\in\overline{Q_T}\times\mathbb{R}^2,
\end{equation}
where $\overline{B_1(0)}$ denotes the closed unit ball in $\mathbb{R}^2$ and $L_H=1$.
We choose the data such that the exact solution {to \eqref{eq:weakform-space-time}} is
\begin{subequations}\label{exact_pair}
	\begin{gather} 
	{u}(x,y,t)\coloneqq \left(2\arctan(t) + 1\right)xy(e^{1.2}-e^{1.2x})(e^{0.7} - e^{0.7y}), 
	\\
	{m}(x,y,t)\coloneqq \left(\frac{1}{4}\tanh(t)+1\right)\sinh(x)\sinh(1-x)y\ln(2-y),
	\end{gather} 
\end{subequations}
for all $(x,y,t)\in \overline{\QT}$, {where the vector field $\tilde{b}_*\in \mathcal{D}_pH[u]$ in \eqref{eq:weakform-space-time} is unique and given by $\tilde{b}_*=|\nabla u|^{-1}\nabla u$ a.e.\ in $Q_T$.}
In particular, we choose of $F:L^2(\QT)\to L^2(\QT)$ and $S:L^2(\Omega)\to L^2(\Omega)$ to be given by $F[m] = m + F_0$ and   $S[m]\coloneqq \tanh(m) + S_0$ respectively, where $F_0 \in L^2(\QT)$ and $S_0 \in L^2(\Omega)$ are determined from $u$ and $m$ above. The source term $G\in L^2(0,T;H^{-1}(\Omega))$ is also determined from $m$ above and can be shown to be nonnegative in the sense of distributions {(see \cite[Sections 3.3 \& 5.3]{YohancePhD}).}
Overall, the data satisfy the hypotheses of Theorem \ref{uniqueness-space-time}, so the solution of the problem is unique.
We use a shape-regular sequence of uniform conforming meshes of $\Omega$ where the elements are right-angled triangles. 
The weights in the diffusion tensor \eqref{edge-tensor-formula} with are chosen as $\omega_{k,E}\coloneqq L_H\text{diam}(E)$, for all $E\in \mathcal{E}_k$, $k\in\mathbb{N}$.
We take the time-step to be comparable to the mesh-size, namely $\tau_k = \frac{h_k}{(1+2^{-k})\sqrt{2}} \eqsim h_k$ {for $k\in\mathbb{N}$}. 
The computations are performed using Firedrake \cite{rathgeber2016firedrake}. 

The results of the experiment are displayed in Figure \ref{exp:rel-err-plots}. {We observe that the approximations for the value function converge to $u$ in the $L^2(H_0^1)$-norm with a rate of order one, and likewise we observe a first-order rate of convergence for the approximations of $m$ in the $L^2(L^2)$-norm. These observed rates are the best possible to be expected since the discretization spaces $\{\Vkpm\}_{k=1}^8$ consist of functions that are piecewise affine in space and piecewise constant in time, and the stabilization term in the scheme \eqref{weakform-space-time-discrete} is consistent to first order in the mesh-size.

Furthermore, we observe strong convergence of the transport vector fields to $\tilde{b}_*$ in the $L^2(L^2)$-norm. This strong convergence is due to the fact that the value function satisfies $\nabla u\neq 0$ a.e.\ in $Q_T$ along with uniqueness of $\tilde{b}_*$ (see \cite[Chapter 6]{YohancePhD} for further details). In light of Corollary \ref{convergence-cor}, the strong convergence of the transport vector fields implies the strong convergence of the density function approximations in the $L^2(H_0^1)$-norm, which we observe in Figure \ref{exp:rel-err-plots} occurs with a convergence rate of order one. Finally, we observe strong convergence of the value function approximations at the initial time and of the density function approximations at the terminal time in the $L^2$-norm, which is in agreement with Theorem \ref{conv-main-thm}.}

\section*{Acknowledgments}
This work was supported by the Engineering and Physical Sciences Research Council [grant number EP/Y008758/1]. This work was also completed with the support of The Royal Society Career Development Fellowship. We would like to acknowledge the use of the UCL Myriad High Performance Computing Facility (Myriad@UCL), and associated support services, in the completion of this work.

\appendix
\section{Proof of Theorem~\ref{DMP-edge-stabilisation-result}}\label{sec:app:dmp}
In this section we prove Theorem \ref{DMP-edge-stabilisation-result}. For this, we will need the following key formula associated with the diffusion tensor $\Dk$ that was defined in \eqref{edge-tensor-formula}. Recall that for a given vertex $x_i \in \calVk$, the corresponding piecewise affine nodal basis function is denoted by $\xi_i$.
\begin{lemma}\label{d-global-non-positivity}
	Let $k\in\mathbb{N}$ be given. 
	Let $\{x_i,x_j\} \subset \calVk$ be a pair of distinct neighbouring vertices of $\Tk$, and let $E\in\mathcal{E}_k$ denote the edge between $x_i$ and $x_j$. Then 
	\begin{equation}\label{edge-tensor-off-diag}
		\int_{\Omega}\Dk\nabla \xi_j\cdot\nabla \xi_i\mathrm{d}x= -\frac{\wEk}{(\diam E )^2}\sum_{K\in \Tke}|K|_d,
	\end{equation}
	where $|\cdot|_d$ denotes the $d$-dimensional Lebesgue measure.
\end{lemma}
\begin{proof}
Consider an element $K \in \Tk$ containing both vertices $x_i$ and $x_j$. Note that $\nabla \xi_i|_K$ is orthogonal to all vectors that are tangent to $F_{K,i}$ the face of $K$ that is opposite $x_i$. Furthermore, any edge $\widetilde{E}\in \calE_K$ that does not contain $x_i$ is contained in $F_{K,i}$.
Therefore $\nabla \xi_i|_K \cdot \bm{t}_{\widetilde{E}} =0$ for all edges $\widetilde{E}\in \calE_K$ that do not contain $x_i$.
Likewise, $\nabla \xi_j|_K \cdot \bm{t}_{\widetilde{E}} =0$ for all edges $\widetilde{E} \in \calE_K$ that do not contain $x_j$. Since the only edge of $\calE_K$ that contains both $x_i$ and $x_j$ is $E$, we see that
\[
\int_K \Dk \nabla\xi_i \cdot \nabla \xi_j \mathrm{d}x = \sum_{\widetilde{E}\in\calE_K} \wtEk |K|_d (\nabla\xi_i\cdot \bm{t}_{\widetilde{E}}) (\nabla\xi_j\cdot \bm{t}_{\widetilde{E}})= \wEk |K|_d (\nabla\xi_i\cdot \bm{t}_{E} )(\nabla\xi_j\cdot \bm{t}_{E}),
\]
i.e.\ the sum above simplifies to a single term.
It is easy to show that $\nabla \xi_i|_K\cdot\bm{t}_{E}=\pm \frac{1}{\diam E},$ where the sign depends on the choice of orientation of the vector $\bm{t}_{E}$.
Since $\xi_i+\xi_j=1$ identically on $E$, we also have $\nabla \xi_j|_K\cdot\bm{t}_{E}= -\nabla \xi_i|_K\cdot\bm{t}_{E}$. 
We therefore conclude that $\int_K \Dk \nabla\xi_i \cdot \nabla \xi_j \mathrm{d}x = - \frac{\wEk|K|_d}{(\diam E)^2}$.
The identity~\eqref{edge-tensor-off-diag} then follows by summation over all elements of the mesh.
\end{proof}

\paragraph{Proof of Theorem \ref{DMP-edge-stabilisation-result}}
It is clear that \eqref{eq:weight_condition} implies that $\norm{\Dk}_{L^\infty(K;\R^{d\times d})}\lesssim h_K$ for all $K\in\Tk$, so it remains only to prove the DMP property for arbitrary elements of $W(V_k,\Dk)$.
Fix $k\in\mathbb{N}$ and let $L\in W(V_k,\Dk)$ be an operator taking the form \eqref{L-operator} for some vector field $\tilde{b}:\Omega\to\R^\dim$ that satisfies the bound  $\|\tilde{b}\|_{L^{\infty}(\Omega;\R^\dim)}\leq L_H$. 
Note that $L$ {satisfies} the DMP {if and only if} the adjoint operator $L^*$ {satisfies} the DMP. {We therefore restrict our attention to showing that $L$ satisfies the DMP.}
Suppose $v\in V_{k}$ is such that $\langle Lv,\xi_i\rangle_{V_k^*\times V_k}\geq 0$ for all $ i\in\{1,\cdots,M_k\}$. Since $v$ is piecewise affine and continuous on $\Omega$, there exists a vertex $x_i$ of the mesh $\mathcal{T}_k$ where $v$ {achieves a global minimum value over $\overline{\Omega}$.}
In the case where $x_i \in \partial \Omega$ then the result is immediate since $v(x_i)=0$. Consider now the case where $x_i \in \Omega$ is an interior vertex. 
Let $\mathbb{S}_i$ denote the index set of all vertices $x_j\in\calVk$ that are neighbours to $x_i$. Note that by definition $\mathbb{S}_i$ does not include $x_i$, but $\mathbb{S}_i$ possibly includes neighbouring vertices on the boundary $\partial \Omega$.

Since the nodal basis functions form a partition of unity, we have $\nabla v = \sum_{j\in\mathbb{S}_i} \nabla \xi_j (v(x_j)-v(x_i))$ in $\supp \xi_i$.
Therefore, we find that
\begin{equation}\label{eq:DMP_proof_1}
\langle L v, \xi_i \rangle \leq \sum_{j\in\mathbb{S}_i} \left[(D_k \nabla \xi_j,\nabla\xi_i)_{\Omega} + (\tilde{b}{\cdot} \nabla\xi_j, \xi_i)_{\Omega} \right] (v(x_j)-v(x_i)),
\end{equation}
where we have used the fact that $\int_\Omega \nu \nabla v\cdot \nabla \xi_i\mathrm{d}x \leq 0$ as a result of the condition~\eqref{XZ-condition}, c.f.\ \cite{xu1999monotone}.
We will show below that the condition on the weights $\wEk$ in~\eqref{eq:weight_condition} implies that, for all $i\in\{1,\dots,M_k\}$,
\begin{equation}\label{eq:DMP_term_sign}
(D_k \nabla \xi_j,\nabla\xi_i)_{\Omega} + (\tilde{b}{\cdot} \nabla\xi_j, \xi_i)_{\Omega} < 0  \quad \forall j\in\mathbb{S}_i,
\end{equation}
for any $\tilde{b}\in L^\infty(\Omega;\R^d)$ satisfying $\norm{\tilde{b}}_{L^\infty(\Omega;\R^d)}\leq L_H$.
Assuming~\eqref{eq:DMP_term_sign} for the moment, we then deduce from $\langle L v, \xi_i\rangle \geq 0$ and from~\eqref{eq:DMP_proof_1} that $v(x_j)=v(x_i)$ for all $j\in\mathbb{S}_i$, i.e.\ $v$ achieves a minimum at all neighbouring vertices.

Repeating the above argument, we can eventually deduce that the global minimum of $v$ is also attained on the boundary $\partial\Omega$, where $v$ vanishes. 
This shows that $v\geq 0$ in $\Omega$ and thus $L$ satisfies the DMP.

Returning to the proof of~\eqref{eq:DMP_term_sign}, let $x_i$ and $x_j$ be a pair of distinct neighbouring vertices, and let $E$ be the edge between $x_i$ and $x_j$. Then, it is straightforward to show that $\norm{\nabla \xi_j}_{L^\infty(K)}\leq \frac{\delta}{2\diam E} $ for each element $K$ containing the edge $E$.
Therefore, using the bounds $\norm{\tilde{b}}_{L^\infty(\Omega;\R^d)}\leq L_H$ and $\norm{\xi_i}_{L^1(K)}\leq \frac{\abs{K}_d}{d+1}$, we find that $(\tilde{b}{\cdot}\nabla \xi_j,\xi_i)_{\Omega}\leq \sum_{K\in\Tke}\frac{\delta L_H|K|_d}{2(d+1)\diam E} $. Using~Lemma~\ref{d-global-non-positivity}, we therefore obtain
\begin{equation}\label{eq:DMP_proof_3}
\begin{split}
(D_k \nabla \xi_j,\nabla\xi_i)_{\Omega} + (\tilde{b}{\cdot} \nabla\xi_j, \xi_i)_{\Omega} \leq  \sum_{K\in\Tke} \left( \frac{\delta L_H  \diam E}{2(d+1)} -  \wEk\right)\frac{|K|_d}{(\diam E)^2}.
\end{split}
\end{equation}
We then see that~\eqref{eq:DMP_term_sign} follows from~\eqref{eq:DMP_proof_3} and~\eqref{eq:weight_condition}.
\hfill\proofbox

\section{Proofs of Lemmas~\ref{KFP_wellposedness_discrete} and~\ref{HJB_wellposedness_discrete}}\label{sec:app:discrete_wellposedness}
{In order to prove Lemmas~\ref{KFP_wellposedness_discrete} and~\ref{HJB_wellposedness_discrete}, we introduce some weighted norms on the discrete spaces that will help to provide a sharper analysis.
Recall that the matrix-valued function $A_k \in L^\infty(\Omega;\R^{d\times d})$ is defined by $A_k\coloneqq \nu\mathbb{I}+\Dk\in L^{\infty}(\Omega;\mathbb{R}^{d\times d})$ {with $\Dk$ defined in~\eqref{edge-tensor-formula}}.}
{Let the norm~$\norm{\cdot}_{A_k}\coloneqq L^2(\Omega;\R^\dim)\tends \R_{\geq 0} $ be defined by $\norm{\bm{w}}_{A_k}=\sqrt{(A_k\bm{w},\bm{w})_\Omega}$ for all $\bm{w}\in L^2(\Omega;\R^d)$. Note that $\norm{\cdot}_{A_k}$ is equivalent to the standard norm on $L^2(\Omega;\R^\dim)$, i.e.\ $\norm{\bm{w}}_{A_k}\eqsim \norm{\bm{w}}_{\Omega}$ for all $\bm{w}\in L^2(\Omega;\R^d)$, since $A_k\in L^\infty(\Omega;\R^{d\times d})$ and since $A_k\geq \nu \mathbb{I}$.}
{Next} we define the dual norm of $\|\cdot\|_{A_k} \colon V_k\tends \R_{\geq 0}$ by
\begin{equation}
\|w\|_{{\Adualk}}\coloneqq \sup_{v\in V_k\setminus\{0\}}\frac{{(w,v)_{\Omk}}}{\|\nabla v\|_{A_k}}\quad\forall w\in V_k. 
\end{equation}
{Observe that~$\norm{w}_{\Adualk} \eqsim \norm{w}_{\dualk}$ for all $w\in V_k^*$ where $\norm{\cdot}_{\dualk}$ is defined in~\eqref{eq:dualknorm}.}
Define the bilinear form $B_k\colon \Vkp\times \Vk\tends \R$ by 
\begin{equation}\label{b-bilinear-form}
	B_k(u,v)\coloneqq \sum_{n=1}^{\Nk}\int_{I_n}(\partial_t\mathcal{I}_+u,v)_{\Omk}+(A_k \nabla u, \nabla v)_{\Omega}\mathrm{d}t\quad\forall (u,v)\in \timeVkforward\times \mathbb{V}_k.
\end{equation}

Given $k\in\mathbb{N}$, let $\tak\in L^\infty(0,T)$ denote the piecewise constant time-dependent weight function that is defined by 
\begin{equation}\label{a-weight-definition}
	\tak|_{I_n}\coloneqq \TAK{n}= \frac{1}{(1+\nu^{-1} L_H^{2}\tau_k)^n}\quad \forall n\in\{1,\cdots,\Nk\}.
\end{equation}
Since the time-steps are assumed to be uniform, it is straight-forward to show that $a_k \in L^{\infty}(0,T)$  is strictly positive and satisfies the uniform bounds 
\begin{equation}\label{eq:weight_bounds}
\exp(-\nu^{-1}L_H^2T) \leq \tak \leq 1 \quad \text{a.e.\ in } (0,T),\quad \forall k\in\N.
\end{equation}
{We now define the following weighted discrete norms on the spaces $\Vk$ and $\Vkp$ that will help to simplify the analysis.}
First, let the norm $\norm{\cdot}_{\Xk}\colon \Vk \tends \R_{\geq 0}$ be defined by
\begin{equation}\label{discrete-weighted-X-norm}
	\|v \|_{\Xk}^2\coloneqq  \int_0^T\tak^{-1} \norm{\nabla v}_{A_k}^2\mathrm{d}t \quad \forall v\in \Vk.
\end{equation}
Note that $\norm{\cdot}_{\Xk}$ is equivalent to $\norm{\cdot}_X$ on $\Vk$, where we recall that $\norm{\cdot}_X$ is defined in~\eqref{eq:continuous_norms}.
Second, let the norm $\norm{\cdot}_{\Yk}\colon \Vkp \tends \R_{\geq 0}$ be defined by
\begin{equation}\label{discrete-weighted-Y-norm}
	\begin{split}
		\norm{w}_{\Yk}^2 \coloneqq & \int_0^T \tak\left(\norm{\p_t \Ip  w }_{\Adualk}^2 + \norm{\nabla w}_{A_k}^2 + \frac{\nu^{-1}L_H^2}{1+\nu^{-1}L_H^2 \tau_k}  \norm{w}_{\Omk}^2 \right)\dd t\\
		&+ \frac{\TAK{N_k}\norm{w(T)}_{\Omk}^2}{1+\nu^{-1}L_H^2 \tau_k} + \frac{1}{1+\nu^{-1}L_H^2 \tau_k} \sum_{n=0}^{N_k-1} \TAK{n} \norm{\jump{w}_{n}}_{\Omk}^2,
	\end{split}
\end{equation}
for all $w\in\Vkp$.
{The next result gives parabolic inf-sup identities that show the relation between the weighted norms $\norm{\cdot}_{\Xk}$ and $\norm{\cdot}_{\Yk}$.
Recall that $\mathbb{V}_{k,0}^+= \{v\in \Vkp,\; v(0)=0\}$.}
\begin{theorem}\label{inf-sup-theorem}
	For all $w\in\timeVkforward$ we have 
	\begin{equation}\label{YX-inf-sup}
		\|w\|_{\Yk}^2=\left[\sup_{v\in\mathbb{V}_{k}\backslash\{0\}}\frac{B_k(w,v)}{\|v\|_{\Xk}}\right]^2+ \frac{\|w(0)\|_{\Omk}^2}{1+\nu^{-1} L_H^{2}\tau_k},
	\end{equation}
	and for all $v\in\mathbb{V}_{k}$
	\begin{equation}\label{XY-inf-sup}
		\|v\|_{\Xk}=\sup_{w\in{\mathbb{V}_{k,0}^{+}}\backslash\{0\}}\frac{B_k(w,v)}{\|w\|_{\Yk}}.
	\end{equation}
\end{theorem}

\begin{proof}
{Let $w\in\Vkp$ be arbitrary. It is clear that the supremum on the right-hand side in~\eqref{YX-inf-sup} is achieved by the test function $v_\dagger=\tak(z+w) \in \Vk$, where $z\in\Vk$ denotes the unique element of $\Vk$ such that $ (A_k \nabla z|_{I_n},\nabla v) = (\p_t \Ip w|_{I_n},v)_{\Omk}$ for all $v\in V_k$, for all $n\in\{1,\dots, N_k\}$. Observe that $v_\dagger\in\Vk$ is justified by the fact that $\tak$ is piecewise constant in time.
We therefore deduce that
\begin{equation}\label{eq:inf-sup-1}
\begin{split}
\left[\sup_{v\in\mathbb{V}_{k}\backslash\{0\}}\frac{B_k(w,v)}{\|v\|_{\Xk}}\right]^2 & = \int_0^T \tak \norm{\nabla(z+w)}_{A_k}^2\mathrm{d}t
\\ &= \int_0^T \tak \left(\norm{\p_t \Ip w}_{\Adualk}^2+\norm{\nabla w}_{A_k}^2 + 2(\p_t\Ip w,w)_{\Omk}\right)\mathrm{d}t,
\end{split}
\end{equation}
where we have expanded the square in the second-line above.
A straightforward calculation shows that
\begin{multline}\label{eq:inf-sup-2}
2\int_0^T \tak (\p_t \Ip w,w)_{\Omk} \mathrm{d}t  = \frac{\nu^{-1}L_H^2}{1+\nu^{-1}L_H^2\tau_k} \int_0^T \tak \norm{w}_{\Omk}^2\mathrm{d}t \\ + \frac{\TAK{N_k}\norm{w(T)}_{\Omk}^2}{1+\nu^{-1}L_H^2 \tau_k}   +\frac{1}{1+\nu^{-1}L_H^2\tau_k} \sum_{n=0}^{N_k-1}\TAK{n}\norm{\jump{w}_{n}}_{\Omk}^2 -  \frac{\norm{w(0)}_{\Omk}^2}{1+\nu^{-1}L_H^2\tau_k} .
\end{multline}
Thus \eqref{eq:inf-sup-1} and \eqref{eq:inf-sup-2} imply~\eqref{YX-inf-sup}. The second identity~\eqref{XY-inf-sup} then follows by duality.}
\end{proof}

{
\begin{corollary}\label{cor:equiv_norms}
We have $\norm{w}_{\Yk}\eqsim \norm{w}_{\Vkp}$ for any $w\in\Vkp$.
\end{corollary}
\begin{proof}
Let $w\in\Vkp$ be arbitrary. The upper bound $\norm{w}_{\Yk}\lesssim \norm{w}_{\Vkp}$ is a consequence of the uniform bounds on $\tak$ in~\eqref{eq:weight_bounds}, the maximum norm bound of Lemma~\ref{lem:vkp_infty_bound}, and the bound for the jumps $\sum_{n=0}^{N_k-1}\norm{\jump{w}_n}_{\Omk}^2 \leq \norm{w}_{\Vkp}^2$. The converse bound follows again from~\eqref{eq:weight_bounds}, and from~\eqref{YX-inf-sup} which implies that $\norm{w(0)}_{\Omk}\lesssim \norm{w}_{\Yk}$.
\end{proof}}

\begin{lemma}\label{tech-result-discrete}
For each $k\in\N$, let the constant $\gamma_k>0$ be defined by
\begin{equation}
\gamma_k\coloneqq \sqrt{\frac{1+\nu^{-1}L_H^2\tau_k}{2}}.
\end{equation}
Then, for any $w\in\timeVkforward$, we have
	\begin{equation}\label{tech-bound-discrete}
		\int_0^T\nu^{-1}L_H^2\tak\norm{w}_{\Omega}^2\mathrm{d}t\leq \gamma_k^2\|w\|_{\Yk}^2+ {\frac{1}{2}} \|w(0)\|_{\Omk}^2.
	\end{equation}
\end{lemma}
{
\begin{proof}
Let $w\in\Vkp$ be arbitrary.
The definition of the dual norm $\norm{\cdot}_{\Adualk}$ and the Cauchy--Schwarz inequality, applied to the left-hand side of~\eqref{eq:inf-sup-2}, imply that
\begin{equation*}
\frac{\nu^{-1}L_H^2}{1+\nu^{-1}L_H^2 \tau_k}\int_0^T \tak \norm{w}_{\Omk}^2\mathrm{d}t -  \frac{\norm{w(0)}_{\Omk}^2}{1+\nu^{-1}L_H^2\tau_k}  \leq \int_0^T \tak\left(\norm{\p_t \Ip w}_{\Adualk}^2+\norm{\nabla w}_{A_k}^2\right)\mathrm{d}t.
\end{equation*}
Therefore, we add $\frac{\nu^{-1}L_H^2}{1+\nu^{-1}L_H^2 \tau_k}\int_0^T \tak \norm{w}_{\Omk}^2\mathrm{d}t$ to both sides of the equation above to find that
\begin{equation}\label{eq:tech_result_1}
 \frac{2\nu^{-1}L_H^2}{1+\nu^{-1}L_H^2 \tau_k}\int_0^T \tak \norm{w}_{\Omk}^2\mathrm{d}t -  \frac{\norm{w(0)}_{\Omk}^2}{1+\nu^{-1}L_H^2\tau_k} \leq \norm{w}_{\Yk}^2.
\end{equation} 
We then obtain~\eqref{tech-bound-discrete} from~\eqref{eq:tech_result_1} after noting that $\norm{w}_{\Omega}\leq\norm{w}_{\Omk}$ by~\eqref{eq:L^2-mass_lumped_bound}.
\end{proof}
}

\paragraph{Proof of Lemma~\ref{KFP_wellposedness_discrete}}
{This result is essentially well-known so we sketch only the main ideas. Note that~\eqref{KFP_eqn_discrete} can be re-written in time-stepping form as the sequence of equations
\begin{equation}
(\overline{m}|_{I_n},v|_{I_n})_{\Omk} + \tau_k\langle L_n \overline{m}|_{I_n},v|_{I_n}\rangle_{V_k^*\times V_k} = (\overline{m}|_{I_{n-1}},v|_{I_n})_{\Omk} + \int_{I_n}\langle G,v|_{I_n}\rangle \mathrm{d}t,
\end{equation}
for each $n\in\{1,\dots,N_k\}$ for some operator $L_n  \in W(V_k,\Dk)$ of the form of~\eqref{L-operator}.
Therefore, the existence and uniqueness of the solution $\overline{m}$ along with its nonnegativity in the case of nonnegative data is a consequence of the DMP, c.f.\ Theorem~\ref{DMP-edge-stabilisation-result}.
It remains only to show the bound~\eqref{eq:KFP_apriori_bound}.
Taking the square-roots to both sides of the identity in~\eqref{YX-inf-sup} and applying the triangle inequality implies that
\[
\norm{\overline{m}}_{\Yk}\leq \sup_{v\in\Vk\setminus\{0\}}\frac{B_k(\overline{m},v)}{\|v\|_{\Xk}} + \frac{\norm{g}_{\Omk}} {\sqrt{1+\nu^{-1}L_H^2\tau_k}},
\]
where we have used the fact $\overline{m}(0)=g$, and where $B_k(\overline{m},v)$ is found by rewriting~\eqref{KFP_eqn_discrete} equivalently as
\begin{equation}
B_k(\overline{m},v) = \int_0^T \langle G,v \rangle - (\overline{m},\tilde{b}{\cdot}\nabla v)_\Omega \mathrm{d}t \quad \forall v\in \Vk.
\end{equation}
Note that $\sup_{v} \frac{B_k(\overline{m},v)}{\norm{v}_{\Xk}} \leq \sup_{v} \frac{\int_0^T \langle G,v \rangle\mathrm{d}t}{\norm{v}_{\Xk}}+\sup_{v} \frac{\int_0^T(\overline{m},\tilde{b}{\cdot}\nabla v)_\Omega \mathrm{d}t}{\norm{v}_{\Xk}} $, where the suprema are over all functions in $\Vk\setminus\{0\}$.
Since $ \int_0^T \tak^{-1}\nu L_H^{-2}\norm{\tilde{b}\cdot\nabla v}_\Omega^2\mathrm{d}t\leq \norm{v}_{\Xk}^2$ for all $v\in\Vk$, we have 
\begin{equation}
\sup_{v \in \Vk\setminus\{0\}} \frac{\int_0^T (\overline{m},\tilde{b}{\cdot}\nabla v)_\Omega \mathrm{d}t}{\norm{v}_{\Xk}} \leq \left(\int_0^T \nu^{-1}L_H^2 \tak \norm{\overline{m}}_{\Omega}^2\mathrm{d}t \right)^\frac{1}{2}
 \leq \gamma_k \norm{\overline{m}}_{\Yk}+\frac{1}{\sqrt{2}}\norm{g}_{\Omk},
\end{equation}
where the second inequality follows from~\eqref{tech-bound-discrete} of Lemma \ref{tech-result-discrete}.
Therefore, we find that
\[
\norm{\overline{m}}_{\Yk} \leq \gamma_k \norm{\overline{m}}_{\Yk}+ C\left( \norm{G}_{L^2(0,T;H^{-1}(\Omega))}+\norm{g}_{\Omega}\right),
\]
where the constant~$C$ depends only on $\nu$, $L_H$, $d$, $T$, and the shape-regularity of the meshes. We therefore obtain~\eqref{eq:KFP_apriori_bound} after noting that $\gamma_k\leq \gamma_1<1$ by~\eqref{time-step-size} and using the equivalence of norms of Corollary~\ref{cor:equiv_norms}.}
\hfill\proofbox

\paragraph{Proof of Lemma~\ref{HJB_wellposedness_discrete}}
{Let $k\in\N$, $\widetilde{F} \in L^2(0,T;H^{-1}(\Omega))$ and let $\widetilde{S}\in L^2(\Omega)$ be fixed. 
For each $w\in\Vk$, let $\Gamma_k(w)\in \Vk$ be defined as the unique solution of
\begin{equation}\label{eq:discrete_HJB_wellposedness_1}
B_k(\psi,\Gamma_k(w)) = \int_0^T\langle \widetilde{F},\psi\rangle  - (H[\nabla w],\psi)_\Omega\mathrm{d}t +(R_k\widetilde{S},\psi(T))_{\Omega,k}\quad \forall \psi\in\Vkpo,
\end{equation}
where $B_k$ is defined in~\eqref{b-bilinear-form} above.
It is clear that $\Gamma_k\colon \Vk\tends \Vk$ is well-defined, and moreover it is clear that $\overline{u}\in\Vk$ solves~\eqref{HJB_eqn_discrete} if and only if $\overline{u}$ is a fixed point of of $\Gamma_k$, i.e. $\overline{u}=\Gamma_k(\overline{u})$. We now show that $\Gamma_k$ is a contraction on $\Vk$ with respect to the norm $\norm{\cdot}_{\Xk}$, which implies existence and uniqueness of the the solution $\overline{u}$ of~\eqref{HJB_eqn_discrete} by the Banach Fixed Point Theorem. To see that $\Gamma_k$ is a contraction map on $\Vk$, let $w_1,\,w_2 \in \Vk$ be arbitrary, and note that~\eqref{XY-inf-sup} implies that
\begin{equation}\label{eq:discrete_HJB_wellposedness_2}
\norm{\Gamma_k(w_1)-\Gamma_k(w_2)}_{\Xk} = \sup_{\psi\in\Vkpo}\frac{\int_0^T (H[\nabla w_2]-H[\nabla w_1],\psi)_\Omega\mathrm{d}t}{\norm{\psi}_{\Yk}}.
\end{equation}
The Cauchy--Schwarz inequality then implies that, for any $\psi\in\Vkpo$,
\begin{multline}\label{eq:discrete_HJB_wellposedness_3}
\int_0^T(H[\nabla w_2]-H[\nabla w_1],\psi)_\Omega\mathrm{d}t \\ \leq \left(\int_0^T \nu L_H^{-2}\tak^{-1} \norm{H[\nabla w_2]-H[\nabla w_1]}_\Omega^2\mathrm{d}t\right)^{\frac{1}{2}}\left(\int_0^T\nu^{-1}L_H^2\tak\norm{\psi}_\Omega^2\mathrm{d}t\right)^\frac{1}{2}.
\end{multline}
We then see that $\left(\int_0^T \nu L_H^{-2}\tak \norm{H[\nabla w_2]-H[\nabla w_1]}_\Omega^2\mathrm{d}t\right)^{\frac{1}{2}}\leq \norm{w_1-w_2}_{\Xk}$ since $H$ is Lipschitz continuous by~\eqref{eq:Lipschitz_H} and since $\nu\mathbb{I}\leq A_k$. Furthermore, Lemma~\ref{tech-result-discrete} applied to $\psi \in \Vkpo$ implies that $\left(\int_0^T\nu^{-1}L_H^2\tak\norm{\psi}_\Omega^2\mathrm{d}t\right)^\frac{1}{2}\leq \gamma_k \norm{\psi}_{\Yk}$.
 Thus we find that
\begin{equation}
\norm{\Gamma_k(w_1)-\Gamma_k(w_2)}_{\Xk} \leq \gamma_k \norm{w_1-w_2}_{\Xk},
\end{equation}
which shows that $\Gamma_k$ is a contraction on $\Vk$, since $\gamma_k\leq \gamma_1<1$ by \eqref{time-step-size}.

It remains only to show the continuous dependence bound~\eqref{HJB_cont_dep_discrete}, where $\overline{u}_1$ and $\overline{u}_2$ are respective solutions of~\eqref{HJB_eqn_discrete} with respective data $(\widetilde{F}_1,\widetilde{S}_1)$ and $(\widetilde{F}_2,\widetilde{S}_2)$. Then, we have
\[
B_k(\psi,u_1-u_2) = \int_0^T  (H[\nabla u_2]-H[\nabla u_1],\psi)_\Omega + \langle \widetilde{F}_1-\widetilde{F}_2,\psi\rangle \mathrm{d}t + (R_k\widetilde{S}_1-R_k\widetilde{S}_2,\psi(T))_{\Omega,k},
\]
for all $\psi\in\Vkpo$.
Combining Lemma~\ref{lem:vkp_infty_bound}, the equivalence of norms $\norm{\cdot}_{\Yk}\eqsim \norm{\cdot}_{\Vkp}$, the $L^2$-stability of $R_k$, and Lemma~\ref{tech-result-discrete}, we obtain
\begin{equation}\label{eq:discrete_HJB_wellposedness_4}
\norm{u_1-u_2}_{\Xk} \leq \gamma_k\norm{u_1-u_2}_{\Xk} + C \left(\norm{\widetilde{F}_1-\widetilde{F}_2}_{L^2(0,T;H^{-1}(\Omega))}+\norm{\widetilde{S}_1-\widetilde{S}_2}_{\Omega} \right),
\end{equation}
where $C$ is a constant that depends on $d$, $\nu$, $L_H$, $T$ and the shape-regularity of the meshes. Therefore, $\norm{u_1-u_2}_{\Xk}\lesssim \norm{\widetilde{F}_1-\widetilde{F}_2}_{L^2(0,T;H^{-1}(\Omega))}+\norm{\widetilde{S}_1-\widetilde{S}_2}_{\Omega}$ since $\gamma_k\leq \gamma_1<1$ by~\eqref{time-step-size}.

Next, since $u_1$ and $u_2$ are extended to functions in $\Vkm$ with $u_1(T)=R_k \widetilde{S}_1$ and $u_2(T)=R_k \widetilde{S}_2$, the discrete integration-by-parts formula~\eqref{eq:discrete_ibp} implies that
\begin{multline}\label{eq:discrete_HJB_wellposedness_5}
\int_0^T -(\p_t \In (u_1-u_2),\psi)_{\Omk} \mathrm{d}t = \int_0^T (H[\nabla u_2]-H[\nabla u_1],\psi)_\Omega + \langle \widetilde{F}_1-\widetilde{F}_2,\psi\rangle \mathrm{d}t 
\\- \int_0^T (A_k\nabla(u_1-u_2),\nabla \psi)_\Omega\mathrm{d}t ,
\end{multline}
which extends to all test functions $\psi\in \Vk$.
Therefore $\int_0^T \norm{\p_t \In (u_1-u_2)}_{\dualk}^2\mathrm{d}t\lesssim \norm{u_1-u_2}^2_{\Xk}+\norm{\widetilde{F}_1-\widetilde{F}_2}^2_{L^2(0,T;H^{-1}(\Omega))}$.  Combining this bound with  $\norm{u_1(T)-u_2(T)}_\Omega\lesssim \norm{\widetilde{S}_1-\widetilde{S}_2}_\Omega$, we conclude that $\norm{u_1-u_2}_{\Vkm}\lesssim \norm{\widetilde{F}_1-\widetilde{F}_2}_{L^2(0,T;H^{-1}(\Omega))}+\norm{\widetilde{S}_1-\widetilde{S}_2}_{\Omega}$ which completes the proof of~\eqref{HJB_cont_dep_discrete}.}
	\hfill\proofbox

\bibliographystyle{siamplain}
\bibliography{p-2-references-revision}
\end{document}